\documentclass[a4paper,11pt]{article}

\usepackage[utf8]{inputenc}
\usepackage[T1]{fontenc}
\usepackage[english]{babel}
\usepackage{amsmath, amsfonts, amsthm,amssymb,here,dsfont, hyperref}
\usepackage{mathtools}
\usepackage{graphicx,color,subfigure,multirow, here}
\usepackage{natbib}
\usepackage{enumitem}
\usepackage{comment}
\usepackage{mathtools}
\usepackage{booktabs}

\newtheorem{theo}{Theorem}[section]

\newtheorem{prop}[theo]{Proposition}

\newtheorem{lemme}[theo]{Lemma}

\newtheorem{assumption}[theo]{Assumption}

\newcommand{\w}{\widetilde}

\newcommand{\R}{\mathbb{R}}

\title{Minimax convergence rates of a binary classification procedure for time-homogeneous SDE paths}

\author{Eddy Michel Ella-Mintsa$^{(1),(2)}$}

\begin{document}

\maketitle
\begin{center}
(1) LAMA, Université Gustave Eiffel\\ 
(2) MIA Paris-Saclay, AgroParisTech
\end{center}

\begin{abstract}
In the context of binary classification of trajectories generated by time-homogeneous stochastic differential equations, we consider a mixture of two diffusion processes characterized by a stochastic differential equation (SDE) whose drift coefficient depends on the class and whose diffusion coefficient is independent of the class. We assume that the drift and diffusion coefficients are unknown as well as the law of the discrete random variable that models the class. In this paper, we study the minimax convergence rates for the excess risk of the resulting plug-in classifier under different sets of assumptions on the diffusion model. As the plug-in classifier is based on nonparametric estimators of drift and diffusion coefficients, we established rates of convergence for projection estimators of drift coefficients on the real line. We propose a new methodology for the study of the lower bound on the excess risk. The theoretical study is completed with a numerical experiment over simulated data.   
\end{abstract}

\vspace*{0.25cm} 



\section{Introduction}
\label{sec:intro}

For several decades, functional data have been the subject of numerous analyzes and statistical studies in fields such as chemistry (\cite{rivera2000time}), physics (\cite{romanczuk2012active}), biology (\cite{etchegaray2016modelisation}), finance (\cite{lamberton2011introduction}) or meteorology (\cite{venalainen2002meteorological}). Diffusion paths are a particular type of functional data generated by Markov processes, solutions of stochastic differential equations. These data are generally used to model various phenomena such as population genetics (see, e.g. \cite{crow2017introduction}), population dynamics (see, e.g. \cite{nagai1983asymptotic}, \cite{etchegaray2016modelisation}) or asset price modeling in finance (see, e.g. \cite{karoui1998robustness}). This article deals with supervised classification of trajectories generated by a mixture of time-homogeneous diffusion processes. To this end, we consider a random pair $(X,Y)$, where the characteristic $X = (X_t)_{t \in [0,1]}$ is a continuous-time diffusion process defined in a filtered probability space $\left(\Omega, \mathcal{F}, \left(\mathcal{F}_t^X\right)_{t \in [0,1]}, \mathbb{P}_X\right)$, where $\left(\mathcal{F}_t^X\right)_{t \in [0,1]}$ is its natural filtration. The process $X$ is observed at discrete times, with time step $\Delta_n = 1/n$. Its drift coefficient $b_Y^*$ depends on the label $Y$, which is a discrete random variable taking values in $\mathcal{Y} = \{0,1\}$, and its diffusion coefficient $\sigma^*$ is independent of the label $Y$. The joint distribution $\mathbb{P}_{X,Y}$ of the random couple $(X,Y)$ is assumed to be unknown, and the random pair $(X,Y)$ is built on the probability space $\left(\Omega, \mathcal{F}, \mathbb{P}_{X,Y}\right)$. A binary classification procedure is a measurable function $g$ that predicts the label of a path $X$ in the set $\mathcal{Y} = \{0,1\}$. The prediction of $g$ has a cost that is measured by the quantity $\mathbb{P}_{X,Y}(g(X) \neq Y)$ called the classification error of $g$. The Bayes classifier is a classification rule $g^*$ that minimizes the classification error, that is, 
\begin{equation*}
    g^* \in \underset{g}{\arg\min}{~\mathbb{P}_{X,Y}(g(X) \neq Y)}.
\end{equation*}
In this paper, we study minimax rates of an empirical classification procedure $\widehat{g}$, built from a learning sample $\mathcal{D}_N$ such that its excess risk
\begin{equation*}
    \mathbb{P}\left(\widehat{g}(X) \neq Y\right) - \mathbb{P}(g^*(X) \neq Y)
\end{equation*}
tends to zero as the size $N$ of $D_N$ tends to infinity.  

\paragraph*{Related works.}

Until now, few papers have tackled the supervised classification problem for diffusion paths. The first article to appear on this topic is \cite{cadre2013supervised} who proposed a nonparametric binary classification procedure for diffusion paths based on minimization of the empirical classification error. The author established an upper bound of order $N^{-k/(2k+1)}$ for the excess risk of the empirical classifier over the space $\mathcal{C}^k(\mathbb{R})$ of $k$ times continuously differentiable functions. \cite{denis2020classif} studied a multiclass classification procedure for diffusion paths in a parametric setting and obtained the consistency of the empirical classifier. In addition, they proposed a nonparametric multiclass classification procedure based on empirical risk minimization and established an upper bound on the excess risk of order $N^{-\beta/(2\beta+1)}$ over the H\"older class of smoothness parameter $\beta \geq 1$. Close to our statistical framework, \cite{denis2024nonparametric} proposed a plug-in-type multiclass classification procedure for time-homogeneous diffusion paths and derived a sub-optimal convergence rate of order $\exp\left(c\sqrt{\log(N)}\right)N^{-1/5}$ where $c>0$ is constant and $\beta \geq 1$ is the smoothness parameter of the H\"older space, the drift and diffusion coefficients being unknown. Then, under stronger assumptions on the drift coefficients (bounded) and the diffusion coefficient ($\sigma^* = 1$), the authors established a faster rate of order $\exp\left(\sqrt{c\log(N)}\right)N^{-\beta/(2\beta+1)}$ when the drift coefficients are bounded and $\sigma^* \equiv 1$. Finally, \cite{gadat2020optimal} focused on trajectories generated by Gaussian processes, solutions of the white-noise model. They proposed a plug-in type binary classifier and obtained an optimal convergence rate of order $N^{-s/(2s+1)}$ over the Sobolev space of smoothness parameter $s>0$. The purpose of the present paper is to study minimax convergence rates for the plug-in classifier under different sets of assumptions on the drift and diffusion coefficients. These rates extend and improve the rates established in the literature, especially in \cite{denis2024nonparametric} whose statistical framework is similar to ours. 

The plug-in classifier resulting from the mixture model considered in this paper is based on nonparametric estimation of drift and diffusion coefficients. This estimation problem is widely investigated in the literature. Focusing only on the nonparametric framework, projection estimators of drift coefficient on a compact interval are proposed in \cite{comte2007penalized} and \cite{denis2021ridge}. For the case of the diffusion coefficient, projection estimators on a compact interval are proposed in \cite{hoffmann1999lp}, \cite{comte2007penalized}, and \cite{ella2024nonparametric}. In this paper, we are interested in nonparametric estimation of drift and diffusion coefficients on the real line. In this context, projection estimators of the drift coefficient on the real line are proposed in \cite{denis2024nonparametric} from $N$ independent high-frequency observations of the diffusion process, and its risk bound is shown to reach a rate of order $\log^{6\beta}(N)N^{-\beta/(2\beta+1)}$ over the H\"older space of regularity parameter $\beta \geq 1$ where $\sigma^* \equiv 1$. Finally, \cite{ella2024nonparametric} built projection estimators of the diffusion coefficient on $\mathbb{R}$ from $N$ independent observations of the diffusion process at discrete times with time step $\Delta_n = \mathrm{O}(N^{-1})$, and derived a risk bound of order $(Nn)^{-\beta/4(\beta+1)}$ on the H\"older space of smoothness parameter $\beta \geq 1$. \cite{ella2025minimax} followed by the study of a faster rate of order $\exp\left(\sqrt{c\log(N)}\right)N^{-3\beta/2(2\beta+1)}$. In this paper, we extend the study of risk bounds for drift estimators on the real line when $\sigma^*$ is unknown, and we improve the risk bound established in \cite{denis2021ridge} when the estimation interval is compact.  

\paragraph*{Main contributions. }

The main contributions of this paper are as follows. 

\begin{enumerate}
    \item Focusing on the nonparametric drift estimation on the real line, we extend the result established in \cite{denis2024nonparametric} to stochastic differential equations with unknown diffusion coefficient $\sigma^*$. This extension requires the use of the exact formula for the transition density for a unique strong solution of a Stochastic Differential Equation (S.D.E.) provided in \cite{dacunha1986estimation} and strong assumptions on the drift and diffusion coefficients. We also derive rates of orders $\log^{3\beta}(N)N^{-\beta/(2\beta+1)}$ and $N^{-\beta/(2\beta+1)}$ providing additional assumptions on both S.D.E. coefficients.
    \item We establish different upper bounds on the excess risk of the plug-in classifier under different sets of assumptions on the coefficients of the diffusion model. These upper bounds are resumed in Table~\ref{tab:rates}.
    \item We propose a new approach for the study of a lower bound on the excess risk of the plug-in classifier. This approach is specifically suited to plug-in classification procedures for diffusion trajectories, as it relies on analyzing a lower bound for the estimation risk of the drift coefficients. We establish a lower bound of order $N^{-\beta/(2\beta+1)}$ over the H\"older space of smoothness parameter $\beta \geq 1$. 
\end{enumerate}

\paragraph*{Outline of the paper. }

Section~\ref{sec:statSetting} is devoted to the statistical setting of our mixture model of diffusion processes. In Section~\ref{sec:upper-bound-plug-in} we study upper and lower bounds on the excess risk of the plug-in classifier. Finally, Sections~\ref{sec:conclusion} and \ref{sec:proofs} are devoted, respectively, to the conclusion and proofs of the results of the paper. 

\begin{table}[h]
\centering
\begin{tabular}{lcc}
\toprule
 & \textbf{Lower bound} & \textbf{Upper bound} \\
\midrule
$b_0^*$ and $b_1^*$ are bounded and $\sigma^*$ unknown & $cN^{-\beta/(2\beta+1)}$ & $C\exp\left(\mathfrak{c}\sqrt{\log(N)}\right)N^{-\beta/(2\beta+1)}$ \\
\midrule
$b_0^*$ and $b_1^*$ are integrable on $\mathbb{R}$ and $\sigma^*$ unknown & $cN^{-\beta/(2\beta+1)}$ & $C\log^{3\beta+2}(N)N^{-\beta/(2\beta+1)}$ \\
\midrule
$b_0^*$ and $b_1^*$ are compactly supported and $\sigma^* = 1$ & $cN^{-\beta/(2\beta+1)}$ & $CN^{-\beta/(2\beta+1)}$ \\
\midrule
$b_0^*$ and $b_1^*$ are unbounded and $\sigma^* = 1$ & $cN^{-\beta/(2\beta+1)}$ & $C\exp\left(\mathfrak{c}\sqrt{\log(N)}\right)N^{-\beta/(2\beta+1)}$ \\
\bottomrule
\end{tabular}
\caption{\small \textit{This table shows the achieved minimax convergence rates for the excess risk of the plug-in classifier under different sets of assumptions on the diffusion model for binary classification of the generated diffusion paths. The upper and lower bounds are given with constants $C, c, \mathfrak{c} > 0$ independent of $N$.}}
\label{tab:rates}
\end{table}

\section{Statistical setting}
\label{sec:statSetting}

In this paper, we consider a binary classification model in which the characteristic $X = (X_t)_{t \in [0,1]}$ is a solution of a stochastic differential equation such that for any $w \in \Omega,$ $X(w)$ belongs to the space $\mathcal{C}^{0}\left([0,1], \mathbb{R}\right)$ of continuous functions on the compact interval $[0,1]$, taking values in  $\mathbb{R}$. More precisely, the diffusion process $X$ is a solution of the following mixture model,
\begin{equation}\label{eq:Diff-Model}
    dX_t = b_Y^*(X_t)dt + \sigma^{*}(X_t)dW_t, ~~ t \in [0,1],  ~~ X_0 = x_0,
\end{equation}
where $W = (W_t)_{t \in [0,1]}$ is the standard Brownian motion, $Y \in \mathcal{Y} = \{0,1\}$ is a binary random variable independent of $W$ that represents the class, following an unknown discrete law ${\bf p}^* = (p_0^*, p_1^*)$ with the assumption that $p_0^*, p_1^* \in (0,1)$. Denote by $\mathcal{X}$ the set of diffusion paths. The performance of any classifier $g : \mathcal{X} \longrightarrow \mathcal{Y}$ is measured by its average classification error given for any couple $(X,Y)$ by
\begin{equation}\label{eq:RiskClassi}
    \mathcal{R}(g) = \mathbb{P}\left(g(X) \neq Y\right) = \mathbb{E}\left[\mathds{1}_{g(X) = 0}\left(1 - \Phi^*(X)\right) + \mathds{1}_{g(X) = 1}\Phi^*(X)\right],
\end{equation}
where $\Phi^*$ is a regression function given for all $x \in \mathbb{R}$ by $\Phi^*(x) = \mathbb{P}(Y = 1 | X = x)$. The Bayes classifier $g^*$ is the classifier that minimizes the classification error, that is,
\begin{equation}\label{eq:Bayes}
    g^* = \underset{g \in \mathcal{G}}{\arg\min}{~\mathcal{R}(g)},
\end{equation}
where $\mathcal{G}$ is a chosen set of classification rules. Since the distribution of the random pair $(X,Y)$ is assumed to be unknown, we build, from a learning sample $\mathcal{D}_N$, an empirical classification procedure of plug-in type $\widehat{g}$ that mimics the Bayes classifier $g^*$ and such that its excess risk $\mathcal{R}(\widehat{g}) - \mathcal{R}(g^*)$ tends to zero as the size $N$ of the learning sample $\mathcal{D}_N$ tends to infinity. \\

We start in Section~\ref{subsec:Notations-Definitions} with some necessary definitions and notation for the study of minimax rates of $\widehat{g}$. In Section~\ref{subsec:ass}, we present the assumptions on the diffusion model \eqref{eq:Diff-Model}, and we give more details on the Bayes classifier $g^*$ in Section~\ref{subsec:BayesClassifier}.

\subsection{Notation and definitions}
\label{subsec:Notations-Definitions}

Consider the learning sample $\mathcal{D}_N = \left\{(\bar{X}^j, Y_j), ~ j = 1, \ldots, N\right\}$ of size $N$, where $\bar{X}^j = (X_{k\Delta_n}^j)_{0\leq k\leq n}$ is a discrete observation of the process $X$ with time step $\Delta_n = 1/n$. We denote by $\mathbb{P}$, the joint distribution of the random vector $\left((\bar{X}^1, Y_1), \ldots, (\bar{X}^N, Y_N), (\bar{X},Y)\right)$ and by $\mathbb{E}$ its corresponding expectation. $\mathbb{P}_X$ is the marginal distribution of the process $X$ and $\mathbb{E}_X$ is its corresponding expectation. For each $i \in \mathcal{Y}$, we denote by $\mathbb{P}^{(i)}$ and $\mathbb{E}^{(i)}$ the following conditional probability and conditional expectation:
\begin{equation*}
    \mathbb{P}^{(i)} := \mathbb{P}\left(.\biggm\vert \mathds{1}_{Y_1 = i}, \ldots, \mathds{1}_{Y_N = i}\right), ~~ \mathbb{E}^{(i)}\left[. \biggm\vert \mathds{1}_{Y_1 = i}, \ldots, \mathds{1}_{Y_N = i}\right].
\end{equation*}
For any continuous function $h$ such that $\mathbb{E}\left[\int_{0}^{1}{h^2(X_s)ds}\right] < \infty$, we define the following metrics:
\begin{equation*}
    \left\|h\right\|_{n}^2 := \mathbb{E}_X\left[\dfrac{1}{n}\sum_{k=0}^{n-1}{h^2(X_{k\Delta_n})}\right], ~~ \left\|h\right\|_{n,i}^2 := \mathbb{E}_{X|Y=i}\left[\dfrac{1}{n}\sum_{k=0}^{n-1}{h^2(X_{k\Delta_n})}\right], 
\end{equation*}
where $\mathbb{E}_{X|Y=i}$ is the corresponding expectation of the marginal distribution of $X$ given the event $\{Y=i\}$, with $i \in \mathcal{Y} = \{0,1\}$. In the event $\{N_i > 1\}$, we also define the following pseudo-norms:
\begin{equation*}
    \left\|h\right\|_{n,N}^2 := \dfrac{1}{Nn}\sum_{j=1}^{N}{\sum_{k=0}^{n-1}{h^2(X_{k\Delta_n}^j)}}, ~~ \left\|h\right\|_{n,N_i}^2 := \dfrac{1}{N_in}\sum_{j=1}^{N_i}{\sum_{k=0}^{n-1}{h^2(X_{k\Delta_n}^{j,i})}}, 
\end{equation*}
where $N_i = \dfrac{1}{N}\sum_{j=1}^{N}{\mathds{1}_{Y_j = i}}$ is the random size of the learning sub-sample $\mathcal{D}_N^i = \left\{(\bar{X}^{j,i}, i), ~ j \in \mathcal{I}_i\right\}$, with $i \in \mathcal{Y}$ and $\mathcal{I}_i \subset \{1, \ldots, N\}$ the set of $j \in \{1, \ldots, N\}$ such that $Y_j = i$. In the sequel, without loss of generality, we set $x_0 = 0$. We also consider the pseudo-norm $\|.\|_X$ and the $L^2-$norm $\|.\|$ given for any continuous function such that $h \in \mathbb{L}^2(\mathbb{R})$ by
\begin{equation*}
    \|h\|_X^2 := \int_{0}^{1}{h^2(X_t)dt}, ~~ \|h\|^2 := \int_{-\infty}^{+\infty}{h^2(x)dx}.
\end{equation*}
The above metrics are used intensively in the study of the lower bound of the excess risk of the plug-in classifier. Finally, for any square matrix $P \in \mathcal{M}_r(\R)$ of size $r \in \mathbb{N}\setminus\{0,1\}$ such that $P$ is symmetric and for all $u \in \R^r, ~ u^{\prime}Pu > 0$ with $u^{\prime}$ the transpose vector of $u$, the operator norm $\|P^{-1}\|_{\mathrm{op}}$ of the inverse $P^{-1}$ is given by
\begin{equation*}
    \|P^{-1}\|_{\mathrm{op}} := \dfrac{1}{\min\{\lambda_1, \ldots, \lambda_d\}},
\end{equation*}

where $\{\lambda_1, \ldots, \lambda_d\}$ is the spectrum of $P$, with $d \leq r$. Finally, we adopt the following notation.
\begin{itemize}
    \item For any $p,q \in \mathbb{N}$ such that $p < q$, we set $[\![p,q]\!] = \{p, p+1, \ldots,q\}$.
    \item For any compact interval $I \subset \mathbb{R}, ~ |I| = \max I - \min I$.
    \item $\mathrm{Supp}(f)$ denotes the support of any function $f$.
    \item For any subset $E \subset \mathbb{R}$, $\mathrm{int}(E)$ denotes the topological interior of $E$.
    \item For all $u,v \in \mathbb{R}^{d}$ with $d \geq 2$, we denote by $<u,v> = u^{\prime}v$ the scalar product of $u$ and $v$.
\end{itemize}

\subsection{Assumptions}
\label{subsec:ass}

We make the following assumptions on our diffusion model.
\begin{assumption}
\label{ass:Reg}
There exists a constant $L_0 > 0$ such that the functions $b_0^*, ~ b_1^*$ and $\sigma^*$ are $L_0-$Lipschitz, that is:
    $$ \forall x,y \in \mathbb{R}, ~ \left|b_0^*(x) - b_0^*(y)\right| + \left|b_1^*(x) - b_1^*(y)\right| + |\sigma^*(x) - \sigma^*(y)| \leq L_0|x-y|$$
\end{assumption}

\begin{assumption}\label{ass:Ell}
    There exist two constants $\sigma_0^*, \sigma_1^* > 0$ such that 
    $$\sigma_0^* \leq \sigma^{*}(x) \leq \sigma_1^*, ~ \forall x \in \mathbb{R}.$$
\end{assumption}
We draw three important consequences from Assumptions~\ref{ass:Reg} and \ref{ass:Ell}. First, the diffusion model~\eqref{eq:Diff-Model} admits a unique strong solution $X = (X_t)_{t \in [0,1]}$ (see \cite{karatzas2014brownian}, \textit{Chapter 5, Theorem 2.9, p.289}). Second, the diffusion process $X = (X_t)_{t \in [0,1]}$ admits a transition density $(t,x) \mapsto \Gamma_X(t, x)$ given as follows:
\begin{equation*}
    \Gamma_X(t, x) = \mathbb{P}(Y=0)\Gamma_{0,X}(t, x) + \mathbb{P}(Y=1)\Gamma_{1,X}(t, x),
\end{equation*}
where $\Gamma_{0,X}$ and $\Gamma_{1,X}$ are the transition densities of the process $X$ respectively under the events $\{Y = 0\}$ and $\{Y = 1\}$ (see \cite{friedman1975stochastic}, \textit{Chapter 5, Theorem 3.1, p.109}). Moreover, that is, there exist constants $c,C>1$ such that for all $t \in (0,1]$ and for all $x\in \mathbb{R}$, we have:
\begin{equation}\label{eq:Transition-density}
    \dfrac{1}{C\sqrt{t}}\exp\left(-c\dfrac{x^2}{t}\right) \leq \Gamma_{X}(t, x) \leq \dfrac{C}{\sqrt{t}}\exp\left(-\dfrac{x^2}{ct}\right).
\end{equation}
(see \cite{gobet2002lan}, \textit{Proposition 1.2}). Third, the diffusion process $X$ admits moments of any order. In fact, under Assumption~\ref{ass:Reg}, we obtain from \cite{ella2024nonparametric}, \textit{Lemma 2.2} that
\begin{equation}\label{eq:Moments}
    \forall~ q \geq 1, ~~ \mathbb{E}\left[\underset{t \in [0,1]}{\sup}{|X_t|^{q}}\right] < C_q,
\end{equation}
where $C_q > 0$ is a constant depending on the order $q$ of the moment. 
\begin{assumption}
    \label{ass:Restrict-Model}
    For all $i \in \mathcal{Y}$, $b_i^* \in \mathcal{C}^{1}\left(\mathbb{R}\right)$, the functions $b_i^*$ and $b_i^{*\prime}$ are bounded on $\mathbb{R}$. Moreover,
    $\sigma^* \in \mathcal{C}^{2}\left(\mathbb{R}\right)$ such that $\sigma^{*\prime}$ and $\sigma^{*\prime\prime}$ are bounded on $\mathbb{R}$, and 
    $$\forall~ A \in \mathbb{R}\setminus\{0\}, ~~ \exp\left[-\dfrac{1}{2}\left(\int_{0}^{A}{\dfrac{1}{\sigma^*(u)}du}\right)^2\right] \leq \exp\left(-\dfrac{A^2}{2}\right).$$
\end{assumption}
Assumption~\ref{ass:Restrict-Model} is restrictive enough on the choice of the diffusion model to build a classification rule for diffusion paths. However, this assumption is required in the sequel for the establishment of rates of order $N^{-\beta/(2\beta+1)}$ (up to extra factors) for the plug-in classifier (see Section~\ref{sec:upper-bound-plug-in}). 
\begin{assumption}[Novikov's condition]
\label{ass:Novikov}
$$\forall i\in\mathcal{Y}, \ \ \mathbb{E}\left[\exp\left(\frac{1}{2}\int_{0}^{1}{\frac{b^{*2}_{i}}{\sigma^{*2}}(X_s)ds}\right)\right] < \infty.$$
\end{assumption}

Under Assumption~\ref{ass:Novikov} also called Novikov's condition, one can apply the Girsanov Theorem (see \cite{revuz2013continuous}, \textit{Chapter VIII}) on our diffusion model, with the aim of establishing a formula that will be used to compute the Bayes classifier $g^*$ with respect to the parameter of the model, as we prove in the next section.

\subsection{Bayes Classifier}
\label{subsec:BayesClassifier}

The Bayes classifier $g^*$ defined in Equation~\eqref{eq:Bayes} is given from Equation~\eqref{eq:RiskClassi} and for each diffusion process $X$ by
\begin{equation*}
    g^*(X) = \mathds{1}_{\left\{\Phi^*(X) \geq 1/2\right\}}. 
\end{equation*}
From \cite{denis2020classif}, Proposition 2.3, the regression function $\Phi^*$ is given for all diffusion processes $X$ by 
\begin{equation*}
    \Phi^*(X) = \mathbb{P}(Y = 1 | X) = \phi^*(F_0^*(X), F_1^*(X)) ~~ \mathbb{P}-a.s.
\end{equation*}
where $\phi^*$ is the softmax function given for all $(x_0,x_1) \in \mathbb{R}^2$ by,
\begin{equation}\label{eq:softmax-function}
    \phi^*(x_0,x_1) := \dfrac{p_1^*\mathrm{e}^{x_1}}{(1-p_0^*)\mathrm{e}^{x_0} + p_1^*\mathrm{e}^{x_1}},
\end{equation}
and for $i \in \mathcal{Y} = \{0,1\}$, 
\begin{equation*}
    F_i^*(X) := \int_{0}^{1}{\dfrac{b_i^*}{\sigma^{*2}}(X_s)dX_s} - \dfrac{1}{2}\int_{0}^{1}{\dfrac{b_i^{*2}}{\sigma^{*2}}}(X_s)ds. 
\end{equation*}
The above results show that the Bayes classifier $g^*$ depends on the unknown parameters $\mathbf{b}^* = \left(b_0^*, b_1^*\right), ~ \sigma^*$ and $\mathbf{p}^*$. Consequently, $g^*$ is not computable in practice. As a result, we propose from the learning sample $\mathcal{D}_N$, a consistent empirical classification procedure of plug-in type $\widehat{g} = g_{\widehat{\mathbf{b}}, \widehat{\sigma}^2, \widehat{\mathbf{p}}}$, based on the estimators $\widehat{\mathbf{b}}, ~ \widehat{\sigma}^2$ and $\widehat{\mathbf{p}}$ of the respective parameters $\mathbf{b}^* = \left(b_0^*, b_1^*\right), ~ \sigma^*$ and $\mathbf{p}^*$. From now on, consider the following notation:
\begin{equation*}
    \begin{aligned}
        & g^*(X) = g_{\mathbf{f}^*}(X), ~~ F_i^*(X) = F_{\bf{b}^*, \sigma^{*2}}^i(X), ~~ \widehat{F}_i(X) = F_{\widehat{\bf{b}}, \widehat{\sigma}^{2}}^i(X),\\
        & \widehat{g}(X) = g_{\widehat{\mathbf{f}}}(X), ~~ \Phi^*(X) = \Phi_{\mathbf{f}^*}(X), ~~ \widehat{\Phi}(X) = \Phi_{\widehat{\mathbf{f}}}(X),
    \end{aligned}
\end{equation*}
where $\mathbf{f}^* = (\mathbf{b}^*, \sigma^{*2}, \mathbf{p}^*), ~ \widehat{\mathbf{f}} = (\widehat{\mathbf{b}}, \widehat{\sigma}^{2}, \widehat{\mathbf{p}}), ~ \mathbf{b^*} = (b_0^*, b_1^*), ~ \widehat{\mathbf{b}} = (\widehat{b}_0, \widehat{b}_1)$ and $\mathbf{p}^* = (p_0, p_1) \in (0,1)^2$ with $\widehat{\mathbf{p}} = (\widehat{p}_0, \widehat{p}_1)$.
Our objective is to establish minimax convergence rates for $\widehat{g}$ under different sets of assumptions on the diffusion model~\eqref{eq:Diff-Model}.

\subsection{The empirical plug-in type classifier}
\label{subsec:PlugInClassifier}

This section is devoted to the construction of the plug-in-type empirical classifier $\widehat{g}$ that mimics the Bayes classifier $g^*$ described in Section~\ref{subsec:BayesClassifier} . To this end, we suppose to have at our disposal a learning sample 
$$\mathcal{D}_N = \left\{\left(\bar{X}^j, Y_j\right), ~~ j = 1, \ldots, N\right\}$$
composed of $N$ independent observations of the random pair $(X,Y)$, where the $\bar{X}^j = (X_{k\Delta_n}^j)_{0 \leq k \leq n}$ are observations of the diffusion process $X$ at discrete times with time step $\Delta_n = 1/n$. We suppose that the diffusion paths $\bar{X}^j, ~ j = 1, \ldots, N$ are high-frequency observations of the diffusion process $X$, that is, the time step $\Delta_n$ tends to zero, which is equivalent to $n$ tends to infinity. From the learning sample $\mathcal{D}_N$, the distribution $\mathbf{p}^* = (p_0^*, p_1^*)$ of the label $Y$ is estimated by $\widehat{\mathbf{p}} = (\widehat{p}_0, \widehat{p}_1)$, with 
\begin{equation*}
    \widehat{p}_0 = \dfrac{1}{N}\sum_{j=1}^{N}{\mathds{1}_{Y_j = 0}} ~~ \mathrm{and} ~~ \widehat{p}_1 = \dfrac{1}{N}\sum_{j=1}^{N}{\mathds{1}_{Y_j = 1}}.
\end{equation*}
We then deduce an estimator of the softmax function $\phi^*$ as follows:
\begin{equation*}
    \widehat{\phi}(x_0,x_1) = \dfrac{\widehat{p_1}\mathrm{e}^{x_1}}{\widehat{p}_0\mathrm{e}^{x_0}+\widehat{p}_1\mathrm{e}^{x_1}}.
\end{equation*}
Finally, let $\widehat{\mathbf{b}} = \left(\widehat{b}_0, \widehat{b}_1\right), ~ \widehat{\sigma}^2$ be the respective nonparametric estimators of the functions $\mathbf{b}^* = \left(b_0^*, b_1^*\right)$ and $\sigma^{*2}$. The empirical classifier $\widehat{g}$ is defined for any diffusion path $\bar{X} = (X_{k\Delta_n})_{0 \leq k\leq n}$ by $\widehat{g}(\bar{X}) = g_{\widehat{\mathbf{f}}}(\bar{X}) = \mathds{1}_{\Phi_{\widehat{\mathbf{f}}}(\bar{X}) \geq 1/2}$, where
\begin{equation*}
    \Phi_{\widehat{\mathbf{f}}}(\bar{X}) = \widehat{\phi}(F_{\widehat{\mathbf{b}}, \widehat{\sigma}^2}^0(\bar{X}), F_{\widehat{\mathbf{b}}, \widehat{\sigma}^2}^1(\bar{X})),
\end{equation*}
and for each $i$ in $\mathcal{Y}$, we have 
\begin{equation*}
    F_{\widehat{\mathbf{b}}, \widehat{\sigma}^2}^i(\bar{X}) := \sum_{k=0}^{n-1}{\dfrac{\widehat{b}_i}{\widehat{\sigma}^2}(X_{k\Delta_n})(X_{(k+1)\Delta_n} - X_{k\Delta_n})} - \dfrac{\Delta_n}{2}\sum_{k=0}^{n-1}{\dfrac{\widehat{b}_i^2}{\widehat{\sigma}^2}(X_{k\Delta_n})}.
\end{equation*}
In the next section, we focus on the establishment of upper bounds on the excess risk of the empirical classifier $\widehat{g}$.

\section{Minimax rate of the excess risk of the empirical plug-in classifier}
\label{sec:upper-bound-plug-in}

The empirical classifier $\widehat{g}$ has been shown to be consistent in \cite{denis2024nonparametric} with a convergence rate of order $N^{-1/5}$ in the general case where coefficients $b_0^*, ~ b_1^*, ~ \sigma^*$ are assumed to be unknown. They showed that the excess risk of $\widehat{g}$ reaches a rate of order $\exp\left(\sqrt{c\log(N)}\right)N^{-\beta/(2\beta+1)}, ~ c>0$, where the diffusion coefficient $\sigma^*$ is known and $\sigma^* = 1$. In this section, we study minimax rates for the excess risk of $\widehat{g}$ where the coefficients $b_0^*, ~ b_1^*$ are $\sigma^*$ are unknown. 

In Section~\ref{subsec:estimators}, we study the risk bounds of projection estimators of the drift coefficients $b_0^*$ and $b_1^*$. Section~\ref{subsec:rate-plug-in} focuses on the minimax convergence rates for the excess risk of the plug-in classifier $\widehat{g}$ for diffusion models with bounded drift coefficients. Section~\ref{subsec:unboundedDrift} tackles the case of diffusion models with unbounded drift coefficients.

\subsection{Nonparametric estimators of the drift coefficients}
\label{subsec:estimators}

We propose, for $i \in \mathcal{Y}$, a nonparametric estimator of the drift function $b_i^*$ that comes from the following stochastic differential equation
\begin{equation*}
    dX_t = b_i^*(X_t)dt + \sigma^*(X_t)dW_t, ~~ X_0 = 0,
\end{equation*}
where $b_i^*$ and $\sigma^*$ satisfy Assumptions~\ref{ass:Reg} and \ref{ass:Ell} and the diffusion coefficient $\sigma^*$ is assumed to be unknown. Consider the sample paths $\left\{\bar{X}^{j,i}, ~ j \in \mathcal{J}_i\right\}$ extracted from the learning sub-sample 
$$\mathcal{D}_N^i = \left\{(\bar{X}^{j,i}, i), ~ j \in \mathcal{I}_i\right\} \subset \mathcal{D}_N$$ 
of random size $N_i := \frac{1}{N}\sum_{j=1}^{N}{\mathds{1}_{Y_j = i}}$ such that $N_i \sim \mathcal{B}(N, p_i^*), ~ \mathrm{and} ~ N_i \underset{N \rightarrow \infty}{\longrightarrow} \infty ~~ a.s.$ The drift coefficient $b_i^*$ is solution of the following regression model
\begin{equation*}
    Z_{k\Delta_n}^{j,i} = b_i^*(X_{k\Delta_n}^{j,i}) + \xi_{k\Delta_n}^{j,i}, ~~ j = 1, \ldots, N_i,
\end{equation*}
where for each $j \in \{1, \ldots, N_i\}$,
\begin{equation}\label{eq:Response-Z}
    Z_{k\Delta_n}^{j,i} = \dfrac{X_{(k+1)\Delta_n}^{j,i} - X_{k\Delta_n}^{j,i}}{\Delta_n}, ~~ k = 0, \ldots, n-1
\end{equation}
is the response variable, and 
\begin{equation*}
    \xi_{k\Delta_n}^{j,i} = \int_{k\Delta_n}^{(k+1)\Delta_n}{(b_i^{*}(X_s^{j,i}) - b_i^{*}(X_{k\Delta}^{j,i}))ds} + \int_{k\Delta_n}^{(k+1)\Delta_n}{\sigma^{*}(X_s^{j,i})dW_s^{j,i}}, ~~ k = 0, \ldots, n-1
\end{equation*}
is the error term.

\subsubsection{Spaces of approximation}

Consider a class $i \in \mathcal{Y}$, and let $K_{N_i} \geq 1$ and $M\geq 1$ be two integers, with $K_{N_i}$ depending on the random size $N_i$ of the sample paths $\left\{\bar{X}^{j,i}, ~ j \in \mathcal{J}_i\right\}$ and such that $K_{N_i}$ tends almost surely to infinity, as $N$ tends to infinity. Consider a compact interval $I_i$ such that $\mathrm{int}(I_i) \neq \emptyset$, and a vector $\mathbf{u} = \left(u_{-M}, \ldots, u_{K_i+M}\right)$ of knots of the compact interval $I_i$ such that
\begin{align*}
    &~ u_{-M} = \ldots = u_0 = \min I_i ~~ \mathrm{and} ~~ u_{K_{N_i}} = \ldots = u_{K_{N_i}+M} = \max I_i \\
    &~ \forall ~ k \in [\![0,K_{N_i}]\!], ~~ u_k = - \min I_i + k\dfrac{|I_i|}{K_{N_i}}.
\end{align*}
\textbf{B}-spline functions of degree $M \geq 1$ are positive piecewise polynomials $\left(B_{\ell}\right)_{\ell \in [\![-M, K_{N_i}-1]\!]}$ with compact supports (see, e.g., \cite{gyorfi2006distribution}, \textit{Chapter 14, p.252}), and satisfying,  .
\begin{equation*}
    \sum_{\ell = -M}^{K_{N_i}-1}{B_{\ell}(x)} = 1, ~~ x \in I_i.
\end{equation*}
We consider the finite-dimensional space $\mathcal{S}_{K_{N_i}}$ spanned by the \textbf{B}-spline basis $\left(B_{\ell}\right)_{\ell \in [\![-M, K_{N_i}-1]\!]}$, that is,
\begin{equation*}
    \mathcal{S}_{K_{N_i}} = \left\{h = \sum_{\ell = -M}^{K_{N_i}-1}{a_{\ell}B_{\ell}}, ~~ \mathbf{a} = (a_{-M}, \ldots, a_{K_{N_i}-1}) \in \mathbb{R}^{K_{N_i}+M}\right\}, ~~ i \in \mathcal{Y}.
\end{equation*}
Then, any element $h$ of the space $\mathcal{S}_{K_{N_i}}$ is a $M-1$ continuously differentiable function. For more details on the \textbf{B}-spline basis, we refer the reader to \cite{gyorfi2006distribution}, \textit{Chapter 14}. 

We propose a projection estimator for the drift function $b_i^*$ in the vector space $\mathcal{S}_{K_{N_i}}$ of dimension $K_{N_i}+M$. Moreover, we add a constraint of $\ell^{2}-$ type to the coordinate vectors of functions $h$ to control their supremum norm. More precisely, we consider the following constraint sub-space
\begin{equation*}
    \mathcal{S}_{K_{N_i}, I_{i}} = \left\{h = \sum_{\ell = -M}^{K_{N_i}-1}{a_{\ell}B_{\ell}} \in \mathcal{S}_{K_{N_i}}, ~ \|\mathbf{a}\|_{2}^2 = \sum_{\ell = -M}^{K_{N_i}-1}{a_{\ell}^2} \leq (K_{N_i}+M)|I_i|^2\log(N_i)\right\}, ~~ i \in \mathcal{Y}.
\end{equation*}
In this paper, we assume that the coefficients $b^{*}_{i}, ~ i \in \mathcal{Y}$ and $\sigma^{*2}$ belong to the H\"older space $\Sigma(\beta, R)$ of smoothness parameter $\beta \geq 1$, with $R > 0$, defined as follows:
\begin{equation*}
    \Sigma(\beta, R) := \left\{f \in \mathcal{C}^{d}(\mathbb{R}, \mathbb{R}), ~ \left|f^{(d)}(x) - f^{(d)}(y)\right| \leq R|x - y|^{\beta - d}, ~~ x,y \in \mathbb{R}\right\},
\end{equation*}
where $d = \lfloor \beta \rfloor$ is the highest integer strictly smaller than $\beta$. The approximation of the H\"older space $\Sigma(\beta, R)$ by the finite dimensional sub-space $\mathcal{S}_{K_{N_i}, I_{i}}$ induces a bias term whose upper-bound is established in \cite{denis2021ridge}, \textit{Lemma D.2} as follows:
\begin{equation}\label{eq:Approx-Error}
    \underset{h\in\mathcal{S}_{K_{N_i}, I_{i}}}{\inf}{\|h-b^{*}_{I_{i}}\|^{2}_{n,i}}\leq C\left(\frac{|I_i|}{K_{N_i}}\right)^{2\beta}, ~~ i \in \mathcal{Y},
\end{equation}
where $C>0$ is a constant independent of $N$ et $n$, and $b^{*}_{I_{i}} = b^{*}_{i}\mathds{1}_{I_i}$ is the restriction of the drift function $b_i^*$ on the compact interval $I_i$. Note that $I_i = \mathrm{Supp}(b_i^*)$ when the drift functions $b_i^*$ are compactly supported, otherwise, $I_i = [-A_{N_i}, A_{N_i}]$ where $A_{N_i}$ is strictly positive and tends to infinity almost surely as $N$ tends to infinity. Moreover, we consider the following subspaces of the H\"older space $\Sigma(\beta, R)$:
\begin{equation*}
    \begin{aligned}
        \Sigma_{\mathrm{int}}(\beta, R) := \left\{f \in \Sigma(\beta, R) : ~ \int_{\mathbb{R}}f(x)dx < \infty\right\}, ~~ \Sigma_c(\beta, R) := \left\{f \in \Sigma(\beta, R) : \mathrm{Supp}(f) ~ \mathrm{is} ~ \mathrm{compact}\right\}.
    \end{aligned}
\end{equation*}
We then consider the following spaces for the model parameter $\mathbf{f}^* = (\mathbf{b}^*, \sigma^{*2}, \mathbf{p}^*)$:
\begin{equation*}
    \begin{aligned}
        &\mathbf{F}(\beta, R) := \Sigma(\beta, R)^2 \times \Sigma(\beta, R) \times (0,1)^2, ~~ \mathbf{F}_{\mathrm{int}}(\beta, R) := \Sigma_{\mathrm{int}}(\beta, R)^2 \times \Sigma(\beta, R) \times (0,1)^2\\
        &\mathbf{F}_{c}(\beta, R) := \Sigma_c(\beta, R)^2 \times \{1\} \times (0,1)^2,
    \end{aligned}
\end{equation*}
and we have $\mathbf{F}_{c}(\beta, R) \subset \mathbf{F}(\beta, R)$ and $\mathbf{F}_{\mathrm{int}}(\beta, R) \subset \mathbf{F}(\beta, R)$. Finally, the set of classifiers $\mathcal{G}$ is given by
\begin{equation*}
    \mathcal{G} := \left\{g_{\mathbf{f}^*}: X \mapsto \mathds{1}_{\left\{\Phi_{\mathbf{f}^*}(X) \geq 1/2\right\}} : ~ \mathbf{f}^* \in \mathbf{F}(\beta, R)\right\}.
\end{equation*}
We establish finer bounds for the excess risk of the plug-in classifier $\widehat{g}$ when the drift function is assumed to belong to each of the spaces $\Sigma_{\mathrm{int}}(\beta, R)$ and $\Sigma_{c}(\beta, R)$. In particular, we show that the empirical classifier $\widehat{g}$ reaches a minimax optimal rate when the drift coefficients belong to $\Sigma_c(\beta, R)$ and $\sigma^* = 1$.

\subsubsection{Minimum contrast estimators of the drift coefficients}
\label{subsec:Contrast-Estimator}

We consider the contrast function $\gamma_{N_i,n}$ defined for any function $h$ by
\begin{equation*}
    \gamma_{N_i,n}(h) := \dfrac{\mathds{1}_{N_i>1}}{N_i,n}\sum_{j=1}^{N_i}{\sum_{k=0}^{n-1}{\left(Z_{k\Delta_n}^{j,i} - h(X_{k\Delta_n}^{j,i})\right)^2}}.
\end{equation*}
The projection estimator $\widehat{b}_i$ of $b_i^*$ on the sub-space $\mathcal{S}_{K_{N_i}, I_i}$ is given by
\begin{equation*}
    \widehat{b}_i = \underset{h \in \mathcal{S}_{K_{N_i}, I_{i}}}{\arg\min}{\gamma_{N_i,n}(h)} ~~ \mathrm{if} ~~ N_i > 1, ~~~ \mathrm{and} ~~~ \widehat{b}_i = 0 ~~ \mathrm{if} ~~ N_i \leq 1. 
\end{equation*}
More precisely, conditional on the event $\{N_i > 1\}$, $\widehat{b}_i$ is given by 
$$\widehat{b}_i = \left<\widehat{\mathbf{a}}, \left(B_{-M}, \ldots, B_{K_{N_i}-1}\right)^{\prime}\right> = \sum_{\ell = -M}^{K_{N_i}-1}\widehat{a}_iB_{\ell},$$ 
where the coordinate vector $\widehat{\mathbf{a}} = (\widehat{a}_{-M}, \ldots, \widehat{a}_{K_{N_i}})$ is given by 
\begin{equation*}
    \widehat{\mathbf{a}} := \underset{\left\|\mathbf{a}\right\|_2^2 \leq (K_{N_i} + M)|I_i|^2\log(N_i)}{\arg\min}\left\|\mathbf{Z} - \Phi_{K_{N_i}}\mathbf{a}\right\|_2^2,
\end{equation*}
with
$$\mathbf{Z} := \left(Z_{0\Delta_n}^1, \ldots, Z_{(n-1)\Delta_n}^1, \ldots, Z_{0\Delta_n}^{N_i}, \ldots, Z_{(n-1)\Delta_n}^{N_i}\right)^{\prime} \in \mathbb{R}^{N_in}$$ and 
$$\Phi_{K_{N_i}} := \left(\left(B_{\ell}(X_{0\Delta_n}^j), \ldots, B_{\ell}(X_{(n-1)\Delta_n}^j)\right)^{\prime}\right)_{-M \leq \ell \leq K_{N_i}-1, 1\leq j \leq N_i} \in \mathbb{R}^{N_in \times (K_{N_i}+M)}.$$
An explicit formula for $\widehat{\mathbf{a}}$ is provided in \cite{denis2021ridge}. 
Consider the Gram matrix $\mathbf{\Psi}_{K_{N_i}}$ of the \textbf{B-spline} basis given by
\begin{equation*}
    \mathbf{\Psi}_{K_{N_i}} = \left(\mathbb{E}_X\left[\dfrac{1}{n}\sum_{k=0}^{n-1}{B_{\ell}(X_{k\Delta})B_{\ell^{\prime}}(X_{k\Delta})}\right]\right)_{-M \leq \ell, \ell^{\prime}, K_{N_i}-1}.
\end{equation*}
For each $i \in \mathcal{Y}$, and on the event $\{N_i > 1\}$, the Gram matrix $\mathbf{\Psi}_{K_{N_i}}$ is invertible (see \cite{denis2024nonparametric}, \textit{Lemma 1}). 

\subsubsection{Risk bounds of drift projection estimators}
\label{subsec:Minimax-Rate-Drift}

In this section, we study upper bounds on the estimation risk of $\widehat{b}_i, ~ i \in \mathcal{Y}$ under some sets of assumptions on the coefficients $b_i^*$ and $\sigma^*$. The goal is to extend the result established in \cite{denis2024nonparametric}, \textit{Theorem 5} to the case of diffusion models with unknown $\sigma^*$. We also improve this result under an additional assumption of the drift coefficient. Finally, we improve the upper bound derived in \cite{denis2021ridge}, \textit{Theorem 4.5} when the drift coefficients are compactly supported. We obtain the following result.

\begin{theo}\label{thm:upper-bound-drift}
    Suppose that $\Delta_n = \mathrm{O}\left(N^{-1}\right)$. For each label $i \in \mathcal{Y}$, and on the event $\{N_i > 1\}$, set $A_{N_i} = \sqrt{\frac{2\beta}{2\beta+1}\log(N_i)}$ and $K_{N_i} = N_i^{1/(2\beta+1)}\log^{-5/2}(N_i)$. Under Assumptions~\ref{ass:Reg}, \ref{ass:Ell} and \ref{ass:Restrict-Model}, there exist constants $C,c>0$ such that
    \begin{equation*}
        \underset{b_i^* \in \Sigma(\beta, R)}{\sup}{~\mathbb{E}\left[\left\|\widehat{b}_i - b_i^*\right\|_{n,i}\right]} \leq C\exp\left(c\sqrt{\log(N)}\right)N^{-\beta/(2\beta+1)}.
    \end{equation*}
    If $b_i^*$ is integrable on $\mathbb{R}$, then, under Assumptions~\ref{ass:Reg}, \ref{ass:Ell} and \ref{ass:Restrict-Model}, there exists a constant $C>0$ such that
    \begin{equation*}
        \underset{b_i^* \in \Sigma_{\mathrm{int}}(\beta, R)}{\sup}{~\mathbb{E}\left[\left\|\widehat{b}_i - b_i^*\right\|_{n,i}\right]} \leq C\log^{3\beta}(N)N^{-\beta/(2\beta+1)}.
    \end{equation*}
    If $b_i^*$ is compactly supported and $\sigma^* = 1$, then, under Assumption~\ref{ass:Reg}, there exists a constant $c>0$ such that
    \begin{equation*}
       \underset{b_i^* \in \Sigma_c(\beta, R)}{\sup}{~\mathbb{E}\left[\left\|\widehat{b}_i - b_i^*\right\|_{n,i}\right]} \leq CN^{-\beta/(2\beta+1)}.
    \end{equation*}
\end{theo}
Theorem~\ref{thm:upper-bound-drift} provides different risk bounds on the estimation risk of each drift coefficient $b_i^*$ under different assumptions on the diffusion model. Thus, under Assumptions~\ref{ass:Reg}, \ref{ass:Ell} and \ref{ass:Restrict-Model}, we have an upper bound of order $\exp\left(c\sqrt{\log(N)}\right)N^{-\beta/(2\beta+1)}$ over the space of H\"older functions of smoothness parameter $\beta \geq 1$. recall that for all $\delta > 0$, the extra factor satisfies
\begin{equation*}
    \exp\left(c\sqrt{\log(N)}\right) = o(N^{\delta}), ~~ \mathrm{as} ~~ N \rightarrow \infty.
\end{equation*}
In \cite{denis2024nonparametric}, an upper bound of order $\log^{6\beta}(N)N^{-\beta/(2\beta+1)}$ has been established when $b_i^*$ is bounded and $\sigma^* = 1$. Theorem~\ref{thm:upper-bound-drift} provides a rate almost of the same order in a more general case with an unknown $\sigma^*$, and under stronger assumptions on the diffusion model leading to the establishment of an upper bound of 
\begin{equation*}
    \underset{t \in [0,1]}{\sup}{~\mathbb{P}(|X_t| > A_N)}
\end{equation*}
which is of the same order as the obtained rate. An additional assumption of the integrability of $b_i^*$ on the real line improves the risk bound since the extra factor $\exp\left(c\sqrt{\log(N)}\right)$ is replaced by a logarithmic factor. Note that the construction of drift estimators does not require a dimension truncation, as in \cite{comte2020nonparametric}. Indeed, the condition on the Gram matrix assumed in \cite{comte2020nonparametric}, which is essential to establish the optimal rate, has been proved in \cite{ella2025minimax} for the \textbf{B}-spline basis and the Fourier basis, and we reuse this result in the present paper. Finally, considering only compactly supported drift coefficients that satisfy Assumption~\ref{ass:Reg}, We obtain the optimal rate of order $N^{-\beta/(2\beta+1)}$, which constitutes a clear improvement over the risk bound established in \cite{denis2021ridge}, \textit{Theorem 4.5}. The logarithmic factor in their result arises from an upper bound on $\mathbb{P}^{(i)}\left(\Omega_{n,N_i,K_{N_i}}^c\right)$ that is not sufficiently sharp, where $\Omega_{n,N_i,K_{N_i}}$ is a random event in which the empirical norm $\|.\|_{n,i}$ and the pseudo-norm $\|.\|_{n,N_i,K_{N_i}}$ are equivalent in the approximation space $\mathcal{S}_{K_{N_i}, I_i}$. In the present paper, the corresponding bound decays faster than any polynomial rate (see Section~\ref{subsec:ProofTheorem3.1}). In the next section, we derive the corresponding upper bounds on the excess risk of the plug-in classifier. 

\subsection{Minimax rates for the plug-in classifier}
\label{subsec:rate-plug-in}

We address the study of convergence rates of the excess risk of the plug-in classifier $\widehat{g}$ built from the respective estimators $\widehat{\bf b} = \left(\widehat{b}_0, \widehat{b}_1\right)$, $\widehat{\sigma}^2$ and $\widehat{\bf p}$ of ${\bf b}^* = \left(b_0^*, b_1^*\right)$, $\sigma^{*2}$ and $\mathbf{p}^* = \left(p_0^*, p_1^*\right)$ respectively. We consider the truncated estimator $\w{\sigma}^2$ of $\sigma^{*2}$ proposed in \cite{denis2024nonparametric} and \cite{ella2024nonparametric}, and given by
\begin{equation*}
    \w{\sigma}^2(x) := \dfrac{1}{\log(N)}\mathds{1}_{\widehat{\sigma}^2(x) \leq 1/\log(N)} + \widehat{\sigma}^2(x)\mathds{1}_{1/\log(N) \leq \widehat{\sigma}^2(x) \leq \log(N)} + \log(N)\mathds{1}_{\widehat{\sigma}^2(x) > \log(N)},~~ x \in \mathbb{R},
\end{equation*}
where the projection estimator $\widehat{\sigma}^2$ satisfies
\begin{equation*}
    \widehat{\sigma}^2 \in \underset{h \in \mathcal{S}_{\w{K}_N, \widetilde{I}}}{\arg\min}{~\w{\gamma}_{N,n}(h)} ~~ \mathrm{with} ~~ \widetilde{I} = [-\widetilde{A}_N, \widetilde{A}_N],
\end{equation*}
with $\w{\gamma}_{N,n}$, the minimum contrast function defined by
\begin{equation*}
    \w{\gamma}_{N,n}(h) := \dfrac{1}{Nn}\sum_{j=1}^{N}{\sum_{k=0}^{n-1}{\left(U_{k\Delta_n}^j - h(X_{k\Delta_n}^j)\right)^2}}, 
\end{equation*}
and
\begin{equation}\label{eq:Response-U}
    U_{k\Delta_n}^j := \dfrac{\left(X_{(k+1)\Delta_n}^j - X_{k\Delta_n}^j\right)^2}{\Delta_n}, ~~ j = 1, \ldots, N, ~~ k = 0, \ldots, n-1.
\end{equation}
Then we have $\widehat{\sigma}^2 = \left<\widehat{\mathbf{a}}, \left(B_{-M}, \ldots, B_{\widetilde{K}_N-1}\right)^{\prime}\right> = \sum_{\ell = -M}^{\widetilde{K}_N-1}\widehat{a}_iB_{\ell}$, where $\widehat{\mathbf{a}} = \left(a_{-M}, \ldots, a_{\widetilde{K}_N-1}\right)$ satisfies
\begin{equation*}
    \widehat{\mathbf{a}} := \underset{\left\|\mathbf{a}\right\|_2^2 \leq (\widetilde{K}_{N} + M)|I|^2\log(N)}{\arg\min}\left\|\mathbf{U} - \Phi_{\widetilde{K}_{N}}\mathbf{a}\right\|_2^2,
\end{equation*}
with $\mathbf{U} := \left(U_{0\Delta_n}^1, \ldots, U_{(n-1)\Delta_n}^1, \ldots, U_{0\Delta_n}^{N_i}, \ldots, U_{(n-1)\Delta_n}^{N_i}\right)^{\prime} \in \mathbb{R}^{N_in}$. We also consider the truncated estimator $\w{\bf b} = \left(\w{b}_0, \w{b}_1\right)$, where for each $i \in \mathcal{Y}$, 
\begin{equation}\label{eq:truncdrift}
    \w{b}_i(x) := \widehat{b}_i(x)\mathds{1}_{|\widehat{b}_i(x)| \leq \log(N)} + \log(N)\dfrac{\widehat{b}_i(x)}{|\widehat{b}_i(x)|}\mathds{1}_{|\widehat{b}_i(x)| > \log(N)}.
\end{equation}
Then, we obtain $\widehat{g} = g_{\w{\bf b}, \w{\sigma}^2, \widehat{\bf p}}$. We derive the following result.
\begin{theo}\label{thm:Minimax-PlugIn}
    For each label $i \in \mathcal{Y}$ and on the event $\{N_i > 1\}$, set $A_{N_i} = \sqrt{\frac{2\beta}{2\beta + 1}\log(N_i)}$, $K_{N_i} = N_i^{1/(2\beta+1)}\log^{-5/2}(N_i)$, $\w{A}_N = \sqrt{\frac{3\beta}{2\beta+1}\log(N)}, ~ \w{K}_N = N^{2/(2\beta+1)}\log^{-5/2}(N)$ and $\Delta_n = \mathrm{O}(N^{-1})$. Under Assumptions~\ref{ass:Reg}, \ref{ass:Ell}, \ref{ass:Restrict-Model} and \ref{ass:Novikov}, there exist constants $C, c_1>0$ such that
    \begin{equation*}
        \underset{\mathbf{f}^* \in \mathbf{F}(\beta, R)}{\sup}{~\mathbb{E}\left[\mathcal{R}(g_{\widehat{\mathbf{f}}}) - \mathcal{R}(g_{\mathbf{f}^*})\right]} \leq C\exp\left(c_1\sqrt{\log(N)}\right)N^{-\beta/(2\beta+1)}.
    \end{equation*}
    If the drift coefficients $b_0^*$ and $b_1^*$ are integrable on $\mathbb{R}$, then under Assumptions~\ref{ass:Reg}, \ref{ass:Ell}, \ref{ass:Restrict-Model} and \ref{ass:Novikov}, there exists a constant $C>0$ such that
    \begin{equation*}
        \underset{\mathbf{f}^* \in \mathbf{F}_{\mathrm{int}}(\beta, R)}{\sup}{~\mathbb{E}\left[\mathcal{R}(g_{\widehat{\mathbf{f}}}) - \mathcal{R}(g_{\mathbf{f}^*})\right]} \leq C\log^{3\beta+2}(N)N^{-\beta/(2\beta+1)}.
    \end{equation*}
    If the drift coefficients $b_0^*$ and $b_1^*$ are compactly supported and twice continuously differentiable, and $\sigma^* = 1$, then under Assumptions~\ref{ass:Reg}, there exists a constant $C>0$ such that
    \begin{equation*}
        \underset{\mathbf{f}^* \in \mathbf{F}_{c}(\beta, R)}{\sup}{~\mathbb{E}\left[\mathcal{R}(g_{\widehat{\mathbf{f}}}) - \mathcal{R}(g_{\mathbf{f}^*})\right]} \leq CN^{-\beta/(2\beta+1)}.
    \end{equation*}
\end{theo}
The upper bound obtained under Assumption~\ref{ass:Reg}, \ref{ass:Ell} and \ref{ass:Restrict-Model} is of the same order as the rate established in \cite{denis2024nonparametric} in the context where the diffusion coefficient $\sigma^*$ is assumed to be known and $\sigma^*=1$. Theorem~\ref{thm:Minimax-PlugIn} extends this result to the case of unknown diffusion coefficients. The establishment of the upper bound has required strong assumptions on both drift and diffusion coefficients that must be sufficiently regular, with the diffusion coefficients taking values in the interval $(0, 1]$. Assuming that the drift coefficients are integrable on $\mathbb{R}$, the corresponding risk bound is slightly sharper than that derived in \cite{denis2024nonparametric}. If we restrict ourselves to compactly supported drift functions, the extra logarithmic factor disappears, and the resulting rate $N^{-\beta/(2\beta+1)}$ is optimal (see Theorem 3.3). 

Note that the upper bounds obtained for the worst excess risk of $\widehat{g}$ can be directly extended to the nonparametric procedure for multiclass classification of diffusion paths studied in \cite{denis2024nonparametric}. Moreover, when $\sigma^* = 1$, the optimal rate of order $N^{-\beta/(2\beta+1)}$ can be reached if the drift functions are estimated using the method proposed in \cite{comte2020nonparametric} based on the truncation of the dimension of the approximation spaces, and the study of the upper bound of the excess risk of $\widehat{g}$ follows the proof method used for the case of compactly supported drift coefficients (see Section~\ref{subsec:ProofTheorem3.2}). The next theorem provides a lower bound of the excess risk of the plug-in classifier showing that the rate of order $N^{-\beta/(2\beta+1)}$ is optimal.
\begin{theo}\label{thm:Minimax-PlugIn2}
    Under Assumptions~\ref{ass:Reg}, \ref{ass:Ell} and \ref{ass:Novikov}, there exists a constant $c > 0$ such that
    \begin{equation*}
       \underset{\widehat{\mathbf{f}}}{\inf}{~\underset{\mathbf{f}^* \in \mathbf{F}(\beta, R)}{\sup}{~\mathbb{E}\left[\mathcal{R}(g_{\widehat{\mathbf{f}}}) - \mathcal{R}(g_{\mathbf{f}^*})\right]} } \geq cN^{-\beta/(2\beta+1)}.
    \end{equation*}
\end{theo}
Theorem~\ref{thm:Minimax-PlugIn2} provides a lower bound of the worst excess risk of the plug-in classifier $\widehat{g} = g_{\widehat{\mathbf{f}}}$, leading, particularly, to an optimal convergence rate of order $N^{-\beta/(2\beta+1)}$ for diffusion models whose drift coefficients belong to space $\Sigma_c(\beta, R)$ and satisfy Assumption~\ref{ass:Reg}. A similar result has been established in \cite{gadat2020optimal} in the context of binary classification of trajectories generated by Gaussian processes, solution of the following white noise diffusion model
\begin{equation*}
    dX_t = f_Y(t)dt + dW_t, ~~ t \in [0,1].
\end{equation*}
In this case, $X$ is a Gaussian process whose transition density is known. In addition, the drift functions $f_0$ and $f_1$ have a compact support, which makes it a simple model. However, in our diffusion model, the coefficients $b_Y^*$ and $\sigma^*$ of the unbound diffusion process $X$ are space dependent. We then had to deal with two major issues such as nonparametric estimation of $b_i^*$ and $\sigma^*$ on $\mathbb{R}$ and the manipulation of non-standard empirical pseudo-norms through the use of the local time of the diffusion process $X$ solution of Equation~\eqref{eq:Diff-Model}.

The next section extends the present study to diffusion models with unbounded drift coefficients.

\subsection{Case of diffusion models with unbounded drift functions}
\label{subsec:unboundedDrift}

In this section, our aim is to extend the result of Section~\ref{subsec:rate-plug-in} to the case of unbounded drift functions. To this end, we assume the diffusion coefficient to be known and $\sigma^* = 1$. Moreover, we need to establish an equivalence relation between the norms $\|.\|_{n, i}, ~ i \in \mathcal{Y}$ similar to the result of \textit{Proposition 4} in \cite{denis2024nonparametric}. We prove that the result of this proposition still holds for unbounded drift functions satisfying the following assumption. 

\begin{assumption}\label{ass:Reg2}
 The functions $b_0^*$ and $b_1^*$ are unbounded, and there exists a constant $C^* > 1$ such that
 \begin{equation*}
     \forall ~ x \in \mathbb{R}, ~~ |(b_0^* - b_1^*)(x)| \leq C^*.
 \end{equation*}
\end{assumption}
The above assumption states that there exist bounded functions $f_0^*$ and $f_1^*$, and an unbounded function $\psi^*$ such that $f_0^* \neq f_1^*$ and 
 \begin{equation*}
     b_0^* = f_0^* + \psi^*, ~~ \mathrm{and} ~~ b_1^* = f_1^* + \psi^*.
 \end{equation*}
Consider, for example, the Ornstein-Uhlenbeck diffusion model given by
\begin{equation*}
    dX_t = -(X_t - \mu_Y^*)dt + \sigma^* dW_t, ~~ X_0 = x_0,
\end{equation*}
where $\mu_Y^*$ is a constant depending on the label $Y \in \mathcal{Y} = \{0,1\}$. In this case, we have $f_0^*(x) = \mu_0^*, ~ f_1^*(x) = \mu_1^*$ and $\psi^*(x) = -x$, for all $x \in \mathbb{R}$. Thus, the functions $b_0^*$ and $b_1^*$ are unbounded and satisfy
\begin{equation*}
    b_0^* - b_1^* = \mu_0^* - \mu_1^*.
\end{equation*}
We establish the following result.
\begin{prop}\label{prop:ChangeProba}
    Suppose that $\sigma^* = 1$. Under Assumptions~\ref{ass:Reg}, \ref{ass:Ell}, \ref{ass:Novikov} and \ref{ass:Reg2}, and for all $i,j \in \mathcal{Y}$ such that $i \neq j$, there exist constants $C,c>0$ such that
    \begin{equation*}
        \left\|\w{b}_i - b_i^*\right\|_{n,j}^2 \leq C\exp\left(\sqrt{c\log(N)}\right)\left\|\w{b}_i - b_i^*\right\|_{n,i}^2 + C\exp\left(\sqrt{c\log(N)}\right)N^{-1}.
    \end{equation*}
\end{prop}
The above result allows us to establish the consistency and convergence rate of the risk of estimation $\mathbb{E}\left[\left\|\w{b}_i - b_i^*\right\|_n\right]$ for each label $i \in \mathcal{Y}$, from the result of Theorem~\ref{thm:upper-bound-drift}. This result extends the equivalence relation obtained in \cite{denis2024nonparametric} \textit{Proposition 4} to a wider range of diffusion models, including those with drift coefficients that satisfy Assumption~\ref{ass:Reg2}. Consequently, for each label $i \in \mathcal{Y} = \{0,1\}$, we have
\begin{equation*}
    \mathbb{E}\left[\left\|\w{b}_i - b_i^*\right\|_n^2\right] \leq C\exp\left(\sqrt{c\log(N)}\right)\mathbb{E}\left[\left\|\w{b}_i - b_i^*\right\|_{n,i}^2\right] + C\exp\left(\sqrt{c\log(N)}\right)N^{-1}.
\end{equation*}
Note that the result of Proposition~\ref{prop:ChangeProba} can be extended to the multiclass setup with $\mathcal{Y} = \{1, \ldots, K\}$ for any integer $K \geq 3$. We establish below a risk bound for the plug-in classifier $\widehat{g}$ built from diffusion models with unbounded drift coefficients. 
\begin{theo}\label{thm:PlugIn-Unbounded}
    Suppose that $\sigma^* = 1$ and $\Delta_n = \mathrm{O}(N^{-1})$. For all $i \in \mathcal{Y}$ and on the event $\{N_i > 1\}$, set
    \begin{equation*}
        A_{N_i} = \sqrt{\dfrac{6\beta}{2\beta+1}\log(N_i)}, ~~ K_{N_i} = N_i^{1/(2\beta+1)}\log^{-5/2}(N_i).
    \end{equation*}
    Under Assumptions~\ref{ass:Reg}, \ref{ass:Novikov} and \ref{ass:Reg2}, there exist constants $C, c > 0$ such that
    \begin{equation*}
       \underset{\widehat{\mathbf{f}}}{\inf}{~\underset{\mathbf{f}^* \in \mathbf{F}(\beta, R)}{\sup}{~\mathbb{E}\left[\mathcal{R}(g_{\widehat{\mathbf{f}}}) - \mathcal{R}(g_{\mathbf{f}^*})\right]}} \leq C\exp\left(c\sqrt{\log(N)}\right)N^{-\beta/(2\beta+1)}.
    \end{equation*}
\end{theo}
The above theorem provides a convergence rate of order $\exp\left(c\sqrt{\log(N)}\right)N^{-\beta/(2\beta+1)}$ that is faster than the rate established in \cite{denis2024nonparametric} of order $\log^{3\beta+1}(N)N^{-3\beta/4(2\beta+1)}$. In addition, the excess risk of $\widehat{g}$ also has a lower bound of order $N^{-\beta/(2\beta+1)}$. However, in the present paper, the result of Theorem~\ref{thm:PlugIn-Unbounded} cannot be extended to diffusion models with unknown diffusion coefficients and unbounded drift functions since the study of a convergence rate for the projection estimator of $\sigma^{*2}$ requires the drift function to be bounded (see \cite{ella2025minimax}, \textit{Proposition 3.2}).

\section{Numerical illustration}
\label{sec:numeric}

We complete the theoretical results established in Section~\ref{sec:upper-bound-plug-in} with a numerical analysis on simulated data. Note that a numerical study of the consistency of the empirical classifier $\widehat{g}$ has been carried out in \cite{denis2024nonparametric}. In this paper, we focus essentially on the numerical illustration of the convergence rate of the excess risk of $\widehat{g}$, comparing the performance of the empirical classifiers related to the different types of diffusion models whose characteristics are listed in Table~\ref{tab:rates}. To this end, we consider learning samples of size $N \in \{100k, ~ k \in [\![1, 10]\!]\}$ with time step $\Delta_n = 1/n$ with $n = 100$. We do not need to consider multiple values of $n$ since the time step does not have a significant impact on the performance of the plug-in classifier (see \cite{denis2024nonparametric}, Section 5, Table 3). We consider the following diffusion models.

\begin{itemize}
    \item Model 1 : $b_0^*(x) = 1/4+(3/4)\cos(x), ~~ b_1^*(x) = 1/2 + \sin^2(x), ~~ \sigma^*(x) = \dfrac{3}{\pi}\left(1 - \dfrac{1}{2+x^2}\right)$
    \item Model 2 : $b_0^*(x) = 4/(\pi(1+x^2)), ~~ b_1^*(x) = 1/(3 + x^2), ~~ \sigma^*(x) = \dfrac{3}{\pi}\left(1 - \dfrac{1}{2+x^2}\right)$\\
    \item Model 3 : $b_0^*(x) = 1 - x, ~~ b_1^*(x) = -1 - x, ~~ \sigma^*(x) = 1$\\
    \item Model 4 : $b_0^*(x) = \sin(1-|x|)\mathds{1}_{|x|\leq 1}, ~~ b_1^*(x) = -2\sin(1 - 4x^2)\mathds{1}_{|x| \leq 1/2}, ~~ \sigma^*(x) = 1$.
\end{itemize}
The above diffusion models satisfy Assumptions~\ref{ass:Reg}, \ref{ass:Ell} and \ref{ass:Novikov}, and Models 1 and 2 satisfy Assumption~\ref{ass:Restrict-Model}. Model 1 is an example of a diffusion model with bounded drift coefficients and an unknown diffusion coefficient. Model 2 differs from Model 1 by the integrability condition on the drift coefficients. Model 3 is a classical and simple diffusion model with unbounded drift coefficients that satisfy Assumption~\ref{ass:Reg2}. Finally, Model 4 with compactly supported drift coefficients will be particularly used to study the optimality of the convergence rate over the set of Lipschitz function. In fact, in Model 4, the drift coefficients are Lipschitz but not derivative on the real line, which implies that the smoothness parameter $\beta$ equals $1$ ($\beta=1$). We then study the evolution of the excess risk of $\widehat{g}$ with respect to that of $N^{-1/3}$ as $N$ varies in $[\![100, 1000]\!]$. 

\begin{figure}[hbtp]
\begin{minipage}[t]{.45\textwidth}
\raggedright
  \centering
  \includegraphics[width=0.7\linewidth, height=0.2\textheight]{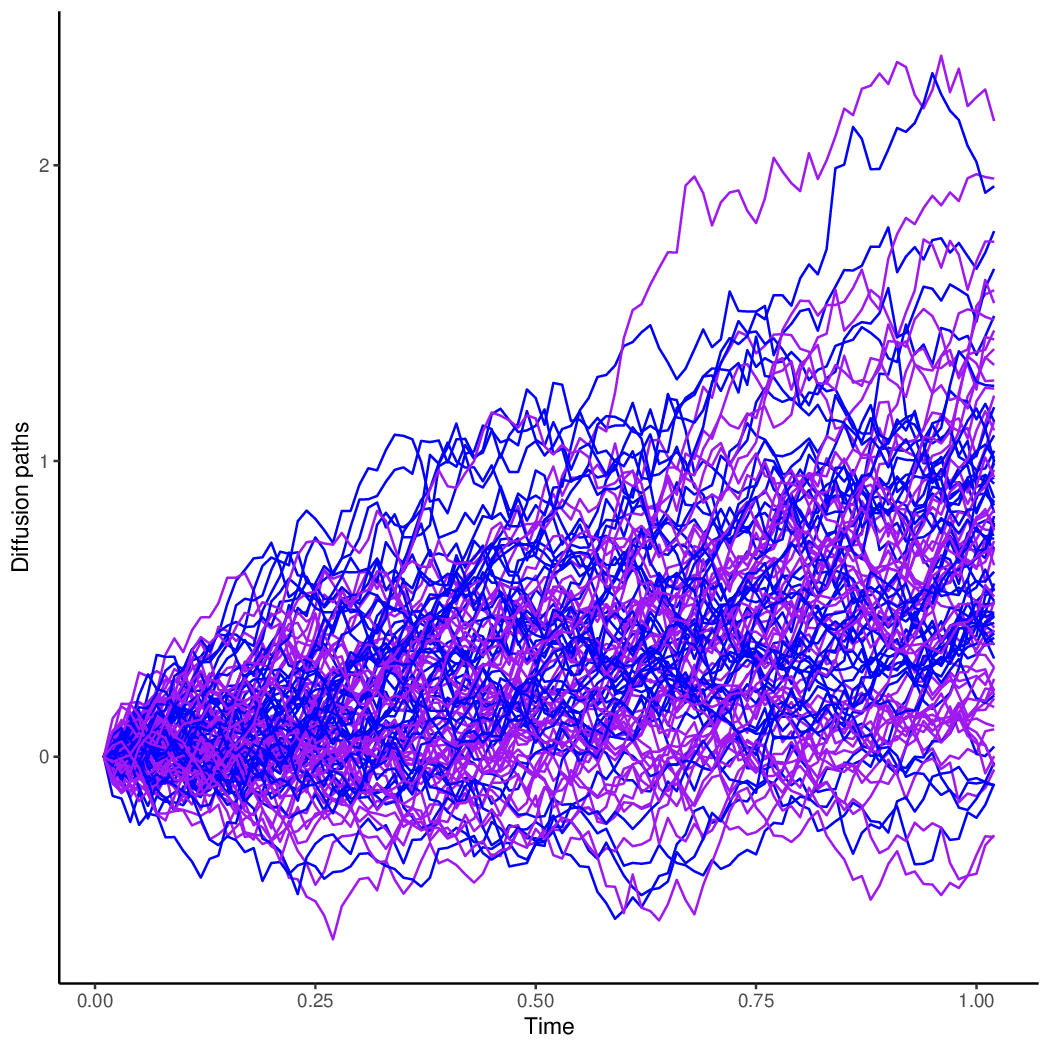}
\end{minipage}
\hfill
\noindent
\begin{minipage}[t]{.55\textwidth}
\raggedleft
  \centering
  \includegraphics[width=0.7\linewidth, height=0.2\textheight]{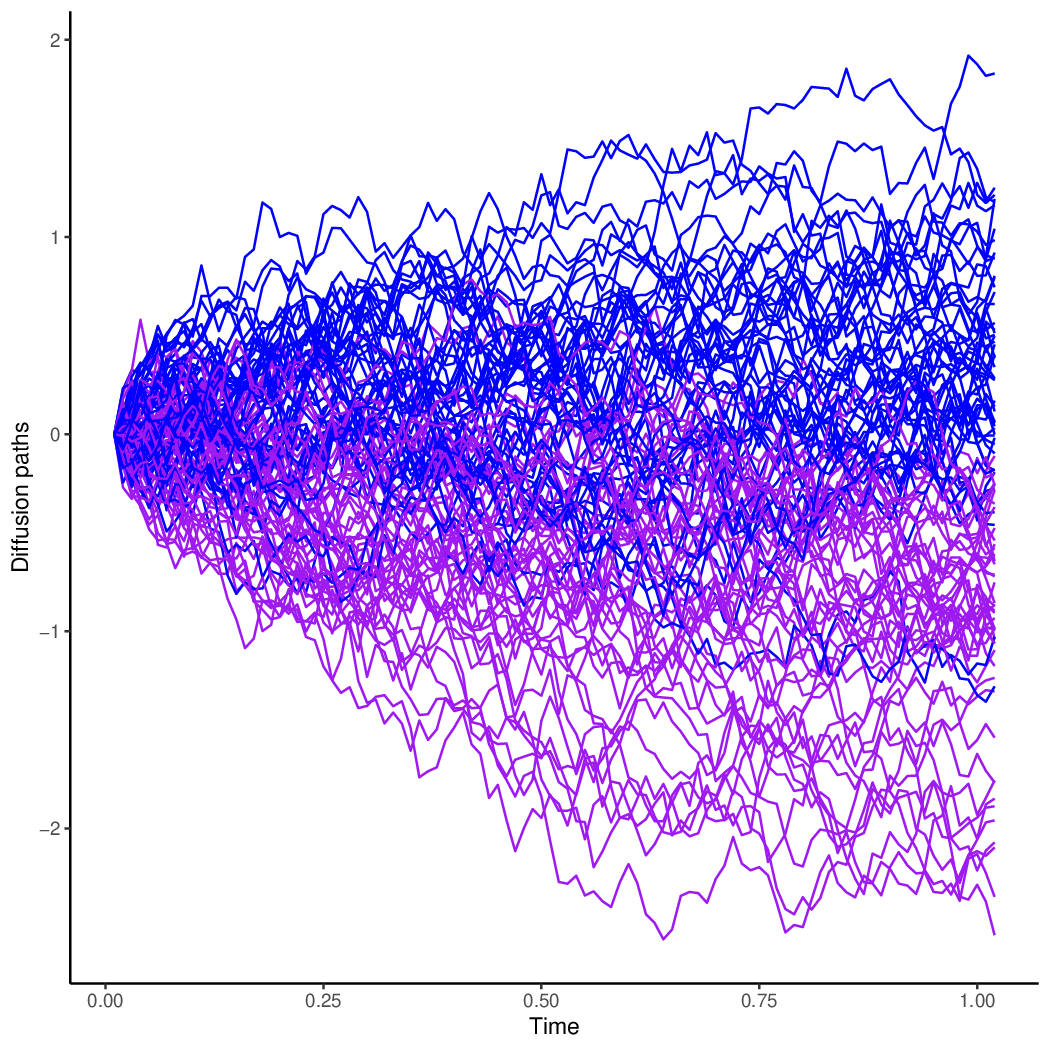}
\end{minipage}
\caption{{\small \textit{Dispersion of diffusion paths from Model 1 (right) and Model 3 (left); with $N=100$ and $n=100$.}}}
\label{fig:paths}
\end{figure}

Figure~\ref{fig:paths} displays the trajectories generated by Models 1 and 3. From these figures, we can anticipate the simplicity of Model 3 with respect to Model 1, since the two classes from Model 3 seem to be more easily separable. 

\subsection{Numerical performances of the Bayes classifiers}

For each diffusion model, the performance of the Bayes classifier is assessed from learning samples of size $N=2000$ with $n=500$. The implementation strategy is detailed in \cite{denis2024nonparametric}, \textit{Section 5.1.1}. We compute the average classification error $\widehat{\mathcal{R}}(g^*)$ of the Bayes classifier $g^*$ on $100$ independent learning samples. More precisely, we have
\begin{equation*}
    \widehat{\mathcal{R}}_N(g^*) = \dfrac{1}{100}\sum_{s \in [\![1, 100]\!]}\dfrac{1}{N}\sum_{(X,y) \in \mathcal{D}_N^s}\mathds{1}_{g^*(X) \neq y},
\end{equation*}
where $\mathcal{D}_N^1, \ldots, \mathcal{D}_N^{100}$ are independent learning samples and $N = 2000$. 

\begin{table}[!h]
\centering
\begin{tabular}{lcccc}
\toprule
 & Model 1 & Model 2 & Model 3 & Model 4\\
\midrule
$\widehat{\mathcal{R}}(g^*)$ & $0.341 ~ (0.010)$ & $0.215 ~ (0.010)$ & $0.161 ~ (0.006)$ & $0.194 ~ (0.009)$ \\
\bottomrule
\end{tabular}
\caption{\small \textit{Assessment of the the average classification error $\widehat{\mathcal{R}}(g^*)$ of the Bayes classifier of each diffusion model.}}
\label{tab:bayes}
\end{table}

The numerical results shown in Table~\ref{tab:bayes} characterize the level of separation of classes, thereby numerically illustrating the simplicity of models such as Model 3 with respect to Model 1, which appears to be more complex. 

\subsection{Numerical results on the Plug-in classifiers}

The performance of each plug-in classifier is evaluated by calculating their average excess risk on $100$ independent learning samples. 

\begin{table}[!h]
\centering
\begin{tabular}{lccc}
\toprule
 & $N=100, ~ n=100$ & $N=500, ~ n=100$ & $N=1000, ~ n=100$ \\
\midrule
Model 1 & $0.391 ~~ (0.062)$ & $0.365 ~~ (0.0055)$ & $0.360 ~~ (0.054)$ \\
\midrule
Model 2 & $0.243 ~~ (0.102)$ & $0.223 ~~ (0.066)$ & $0.215 ~~ (0.041)$ \\
\midrule
Model 3 & $0.172 ~~ (0.038)$ & $0.166 ~~ (0.037)$ & $0.163 ~~ (0.040)$ \\
\midrule
Model 4 & $0.205 ~~ (0.040)$ & $0.197 ~~ (0.035)$ & $0.195 ~~ (0.040)$ \\
\bottomrule
\end{tabular}
\caption{\small \textit{Classification error of plug-in classifiers built from Models 1, 2, 3 and 4.}}
\label{tab:excess_risk}
\end{table}

The results of Table~\ref{tab:excess_risk} illustrate the consistency of the empirical classification rules. The convergence rate depends on the smoothness parameter of the drift and diffusion coefficients. The fact that the rate for Model 4 is one of the lowest is explained by the regularity of the drift coefficients, implying that $\beta = 1$. Focusing on Model 4, we consider the numerical constants $c = 0.009$ and $C = 0.055$, and study the evolution of excess risk $\widehat{\mathcal{R}}_N(\widehat{g}) - \widehat{\mathcal{R}}_N(g^*)$ with respect to the size $N$ of the learning sample. We also consider the sequences $\mathrm{Upper}(N) = CN^{-1/3}$ and $\mathrm{Lower}(N) = cN^{-1/3}$. Figure~\ref{fig:optimal_rate} illustrates the optimality of the rate of order $N^{-1/3}$ for the excess risk of the resulting plug-in classifier. This result is a numerical illustration of the respective results of Theorems~\ref{thm:Minimax-PlugIn} and \ref{thm:Minimax-PlugIn2} for the case of compactly supported drift coefficients. The choice of numerical constants $c>0$ and $C>0$ may draw the reader's attention to the dependency of the constants $C>0$ in Theorem~\ref{thm:Minimax-PlugIn} and $c>0$ in Theorem~\ref{thm:Minimax-PlugIn2} on the model parameters.

\begin{figure}[hbtp]
\centering
  \includegraphics[width=0.55\linewidth, height=0.3\textheight]{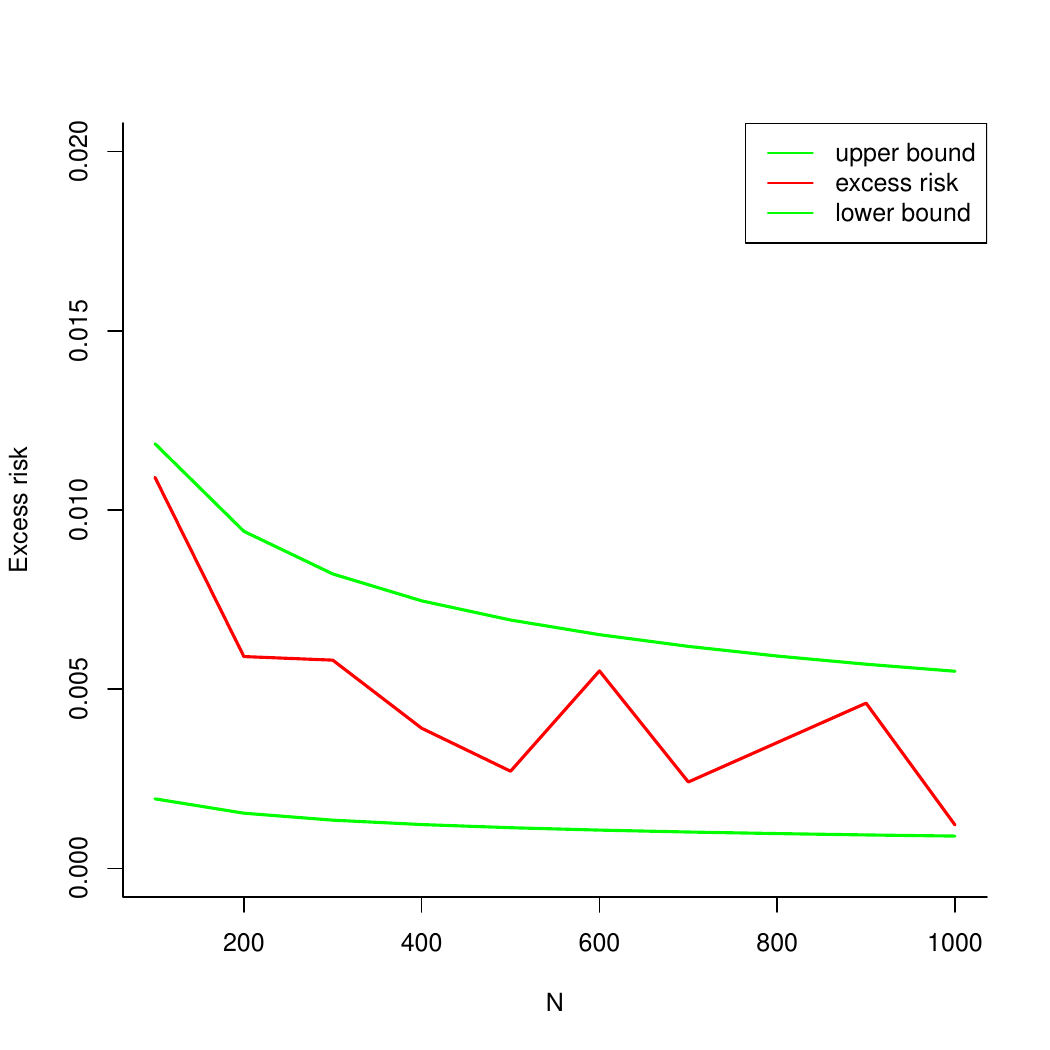}
\caption{{\small \textit{Illustration of the optimality of the convergence rate of the plug-in classifier for Model 4, the optimal rate being $N^{-1/3}$ and $N$ ranging from $100$ to $1000$.}}}
\label{fig:optimal_rate}
\end{figure}

\section{Conclusion}
\label{sec:conclusion}

In conclusion, this paper studied the minimax convergence rates of a binary classification procedure for time-homogeneous SDE paths under different sets of assumptions on the coefficients of the diffusion model. This work extends previous contributions on the construction of nonparametric multiclass classification procedures for SDE paths (see \cite{denis2024nonparametric}). Establishing rates of order $N^{-\beta/(2\beta+1)}$ (with an extra factor) for $b_i^*$ and a rate faster than $N^{-\beta/(2\beta+1)}$ for $\sigma^*$ is particularly challenging, as it requires the use of the exact formula of the transition density provided in \cite{dacunha1986estimation}, leading to very technical proofs for the control of $\left\|\Psi_m^{-1}\right\|_{\mathrm{op}}$ where $\Psi_m$ is the Gram matrix of basis of projection and $m$ its dimension, or the control of probability $\mathbb{P}(|X_t| > A_N)$ for all $t \in [0,1]$, where $A_N = \mathrm{O}(\sqrt{\log(N)})$. Moreover, for diffusion models with unbounded drift coefficients and $\sigma^* = 1$, we also extend the equivalence relation between the empirical norms $\|.\|_{n}$ and $\|.\|_{n,i}, ~ i \in \mathcal{Y}$, leading to a rate of order $N^{-\beta/(2\beta+1)}\exp\left(\sqrt{c\log(N)}\right)$ for the plug-in classifier. Establishing this result required a strong assumption on $b_0^*$ and $b_1^*$ for technical reasons detailed in the proof of Proposition~\ref{prop:ChangeProba}. These two non-trivial extensions were followed with an enhancement of the convergence rate of the plug-in classifier assuming that $b_0^*$ and $b_1^*$ are integrable on $\mathbb{R}$, or the drift functions are compactly supported and $\sigma^*=1$. For the last case, an optimal convergence rate of order $N^{-\beta/(2\beta+1)}$ is derived, based on the optimal rate of order $N^{-\beta/(2\beta+1)}$ for estimators of drift coefficients. 

As perspectives, we can study faster convergence rates of the excess risk of the plug-in classifier under low-noise conditions on the regression function as studied in \cite{gadat2020optimal} in the context of binary classification of the plug-in type for diffusion paths modeled by Gaussian processes.  

\section{Proofs}
\label{sec:proofs}

The section is devoted to the proofs of the main results of the paper. We denote the constants by $C>0$ or $c>0$, and these constants can change from one line to another. Moreover, we also denote by $C_{\theta}>0$ constants that depend on an integer or a parameter $\theta$. Finally, for simplicity, we denote the time step by $\Delta$. 

\subsection{Proofs of Section~\ref{sec:upper-bound-plug-in}}

\subsubsection{Proof of Theorem~\ref{thm:upper-bound-drift}}
\label{subsec:ProofTheorem3.1}

\begin{proof}
    For each $i \in \mathcal{Y}$, denote by $I_i \subset \mathbb{R}$ the set on which the drift estimation is performed. Then $I_i = \mathrm{Supp}(b_i^*)$ when $b_i^*$ is compactly supported, and $I_i = [-A_{N_i}, A_{N_i}]$ when the support of $b_i^*$ is the whole line $\mathbb{R}$. 
    
    \subsubsection*{Case of compactly supported drift coefficients}

    From \cite{denis2021ridge}, \textit{Equation (D.6) in the proof of Proposition 4.4 and Lemma D.2}, conditional on event $\{N_i > 1\}$, there exists a constant $C>0$ such that
    \begin{equation*}
        \mathbb{E}^{(i)}\left(\left\|\widehat{b}_i - b_i^*\right\|_{n, N_i}^2\mathds{1}_{\Omega_{n,N_i,K_{N_i}}}\right) \leq C\left[\left(\dfrac{M+1}{K_{N_i}}\right)^{2\beta} + \dfrac{K_{N_i}}{N_i} + \Delta\right],
    \end{equation*}
    where 
    $$\Omega_{n,N_i,K_{N_i}} =\bigcap_{h \in \mathcal{S}_{K_{N_i}, I_i}} \left\{\left|\dfrac{\|h\|_{n,N_i}^2}{\|h\|_{n,i}^2} - 1\right| \leq \dfrac{1}{2}\right\},$$
    and for all $h \in \mathcal{S}_{N_i, I_i}, ~~ (1/2)\|h\|_{n,i}^2 \leq \|h\|_{n,N_i}^2 \leq (3/2)\|h\|_{n,i}^2$. We deduce that conditional on event $\{N_i > 1\}$,
    \begin{equation}\label{eq:cp-bound1}
        \begin{aligned}
            \mathbb{E}^{(i)}\left[\left\|\widehat{b}_i - b_i^*\right\|_{n, N_i}^2\right] \leq &~ \mathbb{E}^{(i)}\left[\left\|\widehat{b}_i - b_i^*\right\|_{n, N_i}^2\mathds{1}_{\Omega_{n,N_i,K_{N_i}}}\right] + \mathbb{E}^{(i)}\left[\left\|\widehat{b}_i - b_i^*\right\|_{n, N_i}^2\mathds{1}_{\Omega_{n,N_i,K_{N_i}}^{c}}\right]\\
            \leq &~ C\left[\left(\dfrac{M+1}{K_{N_i}}\right)^{2\beta} + \dfrac{K_{N_i}}{N_i} + \Delta\right] + 2(K_{N_i} + M)|I_i|\log(N_i)\mathbb{P}^{(i)}\left(\Omega_{n,N_i,K_{N_i}}^{c}\right).
        \end{aligned}
    \end{equation}
    From \cite{ella2025minimax}, \textit{Lemma 2.6}, the Gram matrix $\Psi_{K_{N_i}}$ of the spline basis built on the compact set $I_i = \mathrm{Supp}(b_i^*)$ is invertible and satisfies 
    $$\mathcal{L}(K_{N_i} + M)\left\|\Psi_{K_{N_i}}^{-1}\right\|_{\mathrm{op}} = \mathrm{O}\left(K_{N_i}\log(N_i)\right),$$
    where
    $$\mathcal{L}(K_{N_i} + M) := \underset{x \in \mathbb{R}}{\sup}{\sum_{\ell = -M}^{K_{N_i} - 1}B_{\ell}^2(x)}.$$
    Then, from \cite{denis2024nonparametric}, \textit{Equation (14)}, there exists a constant $C>0$ such that conditional on event $\{N_i > 1\}$,
    \begin{equation}\label{eq:cp-bound2}
        \mathbb{P}^{(i)}\left(\Omega_{n,N_i,K_{N_i}}^{c}\right) \leq 2(K_{N_i} + M)\exp\left(-\dfrac{N_i}{|I_i|\left\|\Psi_{K_{N_i}}^{-1}\right\|_{\mathrm{op}}}\right) \leq 2(K_{N_i} + M)\exp\left(-\dfrac{N_i}{K_{N_i}\log(N_i)}\right).
    \end{equation}
    For $K_{N_i} = N_i^{1/(2\beta+1)}, ~ \Delta = \mathrm{O}\left(N^{-1}\right)$ and conditional on $\{N_i > 1\}$, we deduce from Equations~\eqref{eq:cp-bound2} and \eqref{eq:cp-bound1} that for $N_i$ a.s. large enough,
    \begin{equation*}
        \mathbb{E}^{(i)}\left[\left\|\widehat{b}_i - b_i^*\right\|_{n, N_i}^2\right] \leq CN_i^{-2\beta/(2\beta+1)},
    \end{equation*}
    where $C>0$ is a new constant. It follows that
    \begin{equation*}
         \begin{aligned}
             \mathbb{E}\left[\left\|\widehat{b}_i - b_i^*\right\|_{n, N_i}^2\right] = &~ \mathbb{E}\left[\left\|\widehat{b}_i - b_i^*\right\|_{n, N_i}^2\mathds{1}_{N_i>1}\right] + \mathbb{E}\left[\left\|b_i^*\right\|_{n, N_i}^2\mathds{1}_{N_i \leq 1}\right]\\
             \leq &~ \mathbb{E}\left[\mathds{1}_{N_i>1}\mathbb{E}^{(i)}\left(\left\|\widehat{b}_i - b_i^*\right\|_{n, N_i}^2\right)\right] + \left\|b_i^*\right\|_{\infty}^2\mathbb{P}(N_i \leq 1)\\
             \leq &~ C\mathbb{E}\left[\left(\dfrac{\mathds{1}_{N_i>1}}{N_i}\right)^{2\beta/(2\beta+1)}\right] + \left\|b_i^*\right\|_{\infty}^2\mathbb{P}(N_i \leq 1).
         \end{aligned}
    \end{equation*}
    Since $N_i \sim \mathcal{B}(N, p_i^*)$, we have
    \begin{equation}\label{eq:cp-bound4}
        \mathbb{P}\left(N_i \leq \right) = \mathbb{P}(N_i = 0) + \mathbb{P}(N_i = 1) \leq 2Np_i^*(1-p_i^*)^{N-1}.
    \end{equation}
    Then, using Jensen's inequality and from \cite{gyorfi2006distribution}, \textit{Chapter 4, Lemma 4.1, p.61}, we obtain the following
    \begin{equation}\label{eq:cp-bound3}
         \begin{aligned}
             \mathbb{E}\left[\left\|\widehat{b}_i - b_i^*\right\|_{n, N_i}^2\right] \leq &~ C\mathbb{E}\left[\left(\dfrac{\mathds{1}_{N_i>1}}{N_i}\right)^{2\beta/(2\beta+1)}\right] + \left\|b_i^*\right\|_{\infty}^2\mathbb{P}(N_i \leq 1)\\
             \leq &~ C\left(\mathbb{E}\left[\dfrac{\mathds{1}_{N_i > 1}}{N_i}\right]\right)^{2\beta/(2\beta+1)} + 2p_i^*\left\|b_i^*\right\|_{\infty}^2N(1 - p_i^*)^{N-1}\\
             \leq &~ CN^{-2\beta/(2\beta+1)},
         \end{aligned}
    \end{equation}
    where $C>0$ is a new constant depending on $\beta$ and $p_i^* \in (0,1)$. To conclude this first part of the proof, using the equivalence relation between the empirical norm $\|.\|_{n,i}$ and the empirical pseudo-norm $\|.\|_{n,N_i}$, we obtain for all $h \in \mathcal{S}_{K_{N_i}, I_i}$,
    \begin{equation}\label{eq:cp-bound5}
        \begin{aligned}
            \mathbb{E}\left[\left\|\widehat{b}_i - b_i^*\right\|_{n,i}^2\right] \leq &~ 2\mathbb{E}\left[\left\|\widehat{b}_i - h\right\|_{n,i}^2\mathds{1}_{\Omega_{n,N_i,K_{N_i}}}\right] + 2\mathbb{E}\left[\left\|\widehat{b}_i - h\right\|_{n,i}^2\mathds{1}_{\Omega_{n,N_i,K_{N_i}}^{c}}\right]  + 2\mathbb{E}\left[\left\|h - b_i^*\right\|_{n,i}^2\right]\\
            \leq &~ 4\mathbb{E}\left[\left\|\widehat{b}_i - h\right\|_{n,N_i}^2\right] + CN^{1/(2\beta+1)}\log(N)\mathbb{E}\left[\mathds{1}_{N_i > }\mathbb{P}^{(i)}\left(\Omega_{n,N_i,K_{N_i}}^{c}\right)\right]\\
            & + 2\mathbb{E}\left[\left\|h - b_i^*\right\|_{n,i}^2\right]\\
            \leq &~ 8\mathbb{E}\left[\left\|\widehat{b}_i - b_i^*\right\|_{n,N_i}^2\right] + CN^{1/(2\beta+1)}\log(N)\mathbb{E}\left[\mathds{1}_{N_i > }\mathbb{P}^{(i)}\left(\Omega_{n,N_i,K_{N_i}}^{c}\right)\right]\\
            & + 10\mathbb{E}\left[\underset{h \in \mathcal{S}_{K_{N_i}, I_i}}{\inf}\left\|h - b_i^*\right\|_{n,i}^2\right].
        \end{aligned}
    \end{equation}
    Recall that $K_{N_i} = N_i^{1/(2\beta+1)}, ~ \Delta = \mathrm{O}\left(N^{-1}\right)$ and the random variable $N_i$ follows a binomial probability distribution with parameters $(N, p_i^*)$. From Equations~\eqref{eq:cp-bound5}, \eqref{eq:cp-bound3}, \eqref{eq:cp-bound4}, \eqref{eq:cp-bound2}, Lemma D.2 in \cite{denis2021ridge}, Jensen's inequality, Lemma 4.1 in \cite{gyorfi2006distribution}, \textit{Chapter 4, p.61}, there exists a constant $C>0$ such that
    \begin{equation*}
        \begin{aligned}
          \mathbb{E}\left[\left\|\widehat{b}_i - b_i^*\right\|_{n,i}^2\right] \leq &~ CN^{-2\beta/(2\beta+1)}.
        \end{aligned}
    \end{equation*}
    Since the above result is obtained for any $b_i^* \in \Sigma_{c}(\beta, R)$, we obtain
    \begin{equation*}
        \underset{b_i^* \in \Sigma_c(\beta, R)}{\sup}\mathbb{E}\left[\left\|\widehat{b}_i - b_i^*\right\|_{n,i}\right] \leq CN^{-\beta/(2\beta+1)}.
    \end{equation*}
    
    \subsubsection*{Case of a non compactly supported and bounded drift coefficient}
    
    Now, the support of $b_i^*$ is the whole line $\mathbb{R}$, and the estimation interval is $I_i = [-A_{N_i}, A_{N_i}]$. We have the following:
    \begin{equation}\label{eq:DecompRisk}
        \begin{aligned}
            \mathbb{E}\left[\left\|\widehat{b}_i - b_i^*\right\|_{n,i}\right] \leq &~  \mathbb{E}\left[\left\|\widehat{b}_i - b_{I_i}\right\|_{n,i}\mathds{1}_{N_i>1}\right] + \mathbb{E}\left[\left\|b_{I_i} - b_i^*\right\|_{n,i}\mathds{1}_{N_i > 1}\right] + \left\|b_i^*\right\|_{\infty}\mathbb{P}\left(N_i \leq 1\right)\\
            \leq &~ \mathbb{E}\left[\left\|\widehat{b}_i - b_{I_i}\right\|_{n,i}\mathds{1}_{N_i>1}\right] + \left\|b_i^*\right\|_{\infty}\mathbb{E}\left[\mathds{1}_{N_i>1}\underset{t \in [0,1]}{\sup}{~\mathbb{P}^{(i)}\left(\left|X_t\right| > A_{N_i}\right)}\right]\\
            & + 2\left\|b_i^*\right\|_{\infty}N(1-p_i^*)^{N-1}.
        \end{aligned}
    \end{equation}
    First, under Assumptions~\ref{ass:Reg}, \ref{ass:Ell}, \ref{ass:Restrict-Model}, From \cite{ella2025minimax}, \textit{Lemma 2.6 and Equation (12)}, the Gram matrix $\mathbf{\Psi}_{K_{N_i}}$ is invertible and satisfies
    \begin{equation*}
        \left\|\mathbf{\Psi}^{-1}_{K_{N_i}}\right\|_{\mathrm{op}} \leq C\dfrac{N_i}{\log^{2}(N_i)},
    \end{equation*}
    with $A_{N_i} = \sqrt{\frac{2\beta}{2\beta+1}\log(N_i)}$ and $K_{N_i} = N_i^{1/(2\beta+1)}\log^{-5/2}(N_i)$, we obtain from \cite{denis2024nonparametric}, \textit{Theorem 5} that there exists a constant $C>0$ such that
    \begin{equation}\label{eq:Term1}
        \mathbb{E}\left[\left\|\widehat{b}_i - b_{A_{N_i}}\right\|_{n,i}\mathds{1}_{N_i>1}\right] \leq \sqrt{\mathbb{E}\left[\left\|\widehat{b}_i - b_{A_{N_i}}\right\|_{n,i}^2\mathds{1}_{N_i>1}\right]} \leq C\log^{3\beta}(N)N^{-\beta/(2\beta+1)}.
    \end{equation}
    Second, from \cite{ella2025minimax}, \textit{Proof of Proposition 3.2, Equation (50)}, there exist constants $c,C > 0$ such that 
    \begin{equation*}
        \mathbb{E}\left[\mathds{1}_{N_i > 1}\underset{t \in [0,1]}{\sup}{\mathbb{P}^{(i)}\left(|X_t| > A_{N_i}\right)}\right] \leq C\exp(cA_N)\mathbb{E}\left[\mathds{1}_{N_i > 1}\exp\left(-\dfrac{A_{N_i}}{2}\right)\right].
    \end{equation*}
    For $A_{N_i} = \sqrt{\frac{2\beta}{2\beta+1}\log(N_i)}$, there exists a constant $C>0$ such that
    \begin{equation}\label{eq:ProbaOut1}
        \begin{aligned}
            \mathbb{E}\left[\mathds{1}_{N_i > 1}\underset{t \in [0,1]}{\sup}{\mathbb{P}^{(i)}\left(|X_t| > A_{N_i}\right)}\right] \leq &~ C\exp\left(cA_N\right)\mathbb{E}\left[\mathds{1}_{N_i > 1}N_i^{-\beta/(2\beta+1)}\right]\\
            \leq &~ C\exp(cA_N)N^{-\beta/(2\beta+1)}.
        \end{aligned}
    \end{equation}
    We deduce from Equations~\eqref{eq:ProbaOut1}, \eqref{eq:Term1} and \eqref{eq:DecompRisk}, that there exist $C,c>0$ depending on $\sigma_1^*$, $\|b_i^*\|_{\infty}$, $\mathbf{p}^*$ and $\beta$ such that
    \begin{equation*}
        \underset{b_i^* \in \Sigma(\beta, R)}{\sup}{~\mathbb{E}\left[\left\|\widehat{b}_i - b_i^*\right\|_{n,i}\right]} \leq C\exp\left(c\sqrt{\log(N)}\right)N^{-\beta/(2\beta+1)}.
    \end{equation*}

    \subsubsection*{Case of bounded drift $b_i^*$ such that $\int_{\mathbb{R}}|b_i^*(x)|dx < \infty, ~ i \in \mathcal{Y}$}

    The proof is almost exactly the same as that of Proposition 3.2 in \cite{ella2025minimax}. the benefit with an integrable drift on $\mathbb{R}$ is that the function $x \mapsto H_i(x)$ from the exact formula of the transition density (see proof of Lemma~\ref{lm:ChangeNorm} in Appendix) is bounded on $\mathbb{R}$. Consequently, there exist constants $C, C^{\prime}>0$ such that
    \begin{equation*}
        \mathbb{E}\left[\mathds{1}_{N_i > 1}\underset{t \in [0,1]}{\sup}{\mathbb{P}^{(i)}\left(|X_t| > A_{N_i}\right)}\right] \leq C\mathbb{E}\left[\mathds{1}_{N_i > 1}\exp\left(-\dfrac{A_{N_i}}{2}\right)\right] \leq C^{\prime}N^{-\beta/(2\beta+1)},
    \end{equation*}
    which implies from Equation~\eqref{eq:Term1} and \eqref{eq:DecompRisk} that
    \begin{equation*}
        \underset{b_i^* \in \Sigma_{\mathrm{int}}(\beta, R)}{\sup}{~\mathbb{E}\left[\left\|\widehat{b}_i - b_i^*\right\|_{n,i}\right]} \leq C\log^{3\beta}(N)N^{-\beta/(2\beta+1)}.
    \end{equation*}
\end{proof}

\subsubsection{Proof of Theorem~\ref{thm:Minimax-PlugIn}}
\label{subsec:ProofTheorem3.2}

The proof of Theorem~\ref{thm:Minimax-PlugIn} relies on the following lemma.

\begin{lemme}\label{lm:ChangeNorm}
    Suppose that the drift functions $b_0^*$ and $b_1^*$ are integrable on $\mathbb{R}$. Under Assumptions~\ref{ass:Reg}, \ref{ass:Ell} and \ref{ass:Restrict-Model}, there exists a constant $C>0$ such that for all $i,j \in \mathcal{Y}$ such that $i \neq j$,
    \begin{equation*}
        \left\|\widehat{b}_i - b_i^*\right\|_{n,j}^2 \leq C\left\|\widehat{b}_i - b_i^*\right\|_{n,i}^2,
    \end{equation*}
    where $\widehat{b}_i$ is a nonparametric estimator of $b_i^*$.
\end{lemme}
There is a substantial improvement compared to the equivalence relation established in \cite{denis2024nonparametric}, \textit{Proposition 4}, due to the integrability condition on the drift coefficients $b_i^*, ~ i \in \mathcal{Y}$. The proof of Lemma~\ref{lm:ChangeNorm} is provided in the appendix.

\begin{proof}[\textbf{Proof of Theorem}~\ref{thm:Minimax-PlugIn}]
For any $\mathbf{f}^* \in \mathbf{F}(\beta, R)$ under Assumptions~\ref{ass:Reg}, \ref{ass:Ell} and \ref{ass:Novikov}, and from \cite{denis2024nonparametric}, \textit{Theorem 1}, the excess risk of $\widehat{g} = g_{\widehat{\mathbf{f}}}$ satisfies
\begin{equation}\label{eq:EQ0}
    \mathbb{E}\left[\mathcal{R}(g_{\widehat{\mathbf{f}}}) - \mathcal{R}(g_{\mathbf{f}^*})\right] \leq C\left(\sqrt{\Delta} + \dfrac{1}{p_{\min}\sqrt{N}} + \mathbb{E}\left[\dfrac{b_{\max}}{\sigma_0^{2}}\sum_{i=0}^{1}{\left\|\w{b}_i - b_i^*\right\|_n}\right] + \mathbb{E}\left[\dfrac{1}{\sigma_0^2}\left\|\w{\sigma}^2 - \sigma^{*2}\right\|_n\right]\right),
\end{equation}
where $C>0$ is a constant independent of $N$, $p_{\min} = \min(p_0^*, p_1^*) > 0$, the quantities
$$b_{\max} = \max\left\{\left\|\w{b}_0\right\|_{\infty}, \left\|\w{b}_1\right\|_{\infty}\right\} = (\log(N))^{1/2} ~ \mathrm{and} ~ \sigma_0^2 = 1/\log(N) \leq \underset{x \in \mathbb{R}}{\inf}{~\w{\sigma}^2(x)}$$ 
are logarithmic factors  depending on the size $N$ of the learning sample. Under Assumptions~\ref{ass:Reg}, \ref{ass:Ell} and \ref{ass:Restrict-Model}, we obtain from \cite{ella2025minimax}, \textit{Proposition 3.2} that there exist constants $C, c > 0$ such that
\begin{equation}\label{eq:EQ3}
    \mathbb{E}\left[\left\|\widetilde{\sigma}^2 - \sigma^{*2}\right\|_{n}\right] \leq C\exp\left(c\sqrt{\log(N)}\right)N^{-3\beta/2(2\beta+1)}.
\end{equation}
Moreover, from Theorem~\ref{thm:upper-bound-drift}, we have
\begin{equation*}
    \mathbb{E}\left[\left\|\w{b}_i - b_i^*\right\|_{n,i}\right] \leq C\exp\left(c\sqrt{\log(N)}\right)N^{-\beta/(2\beta+1)}, ~~ i \in \mathcal{Y},
\end{equation*}
and from \cite{denis2024nonparametric} \textit{Proposition 4}, there exist constants $C,c>0$ such that for all $i \in \mathcal{Y}$,
\begin{equation}\label{eq:EQ4}
    \begin{aligned}
        \mathbb{E}\left[\left\|\w{b}_i - b_i^*\right\|_{n}\right] \leq &~ C\exp\left(c\sqrt{\log(N)}\right)\mathbb{E}\left[\left\|\w{b}_i - b_i^*\right\|_{n,i}\right] + C\dfrac{\log^2(N)}{N} \\
        \leq &~ C\exp\left(c\sqrt{\log(N)}\right)N^{-\beta/(2\beta+1)}.
    \end{aligned}
\end{equation}
For $\Delta = \mathrm{O}(N^{-1})$, we finally obtain from Equations~\eqref{eq:EQ4}, \eqref{eq:EQ3} and \eqref{eq:EQ0} that
\begin{equation*}
    \underset{\mathbf{f}^* \in \mathbf{F}(\beta, R)}{\sup}{~\mathbb{E}\left[\mathcal{R}(g_{\widehat{\mathbf{f}}}) - \mathcal{R}(g_{\mathbf{f}^*})\right]} \leq C\exp\left(c\sqrt{\log(N)}\right)N^{-\beta/(2\beta+1)},
\end{equation*}
where $C>0$ and $c>0$ are new constants.

\subsubsection*{Case of bounded drift $b_i^*$ such that $\int_{\mathbb{R}}|b_i^*(x)|dx < \infty, ~ i \in \mathcal{Y}$}

From Theorem~\ref{thm:upper-bound-drift} and Lemma~\ref{lm:ChangeNorm}, there exists a constant $C> 0$ such that
\begin{equation}\label{eq:EQ5}
    \begin{aligned}
        \mathbb{E}\left[\left\|\w{b}_i - b_i^*\right\|_{n}\right] \leq &~ C\log^{3\beta}(N)N^{-\beta/(2\beta+1)}, ~~ i \in \mathcal{Y}.
    \end{aligned}
\end{equation}
It follows from Equations~\eqref{eq:EQ5}, \eqref{eq:EQ3} and \eqref{eq:EQ0} with $\Delta = \mathrm{O}\left(N^{-1}\right)$ that 
\begin{equation*}
    \underset{\mathbf{f}^* \in \mathbf{F}_{\mathrm{int}}(\beta, R)}{\sup}{~\mathbb{E}\left[\mathcal{R}(g_{\widehat{\mathbf{f}}}) - \mathcal{R}(g_{\mathbf{f}^*})\right]} \leq C\log^{3\beta+2}(N)N^{-\beta/(2\beta+1)}.
\end{equation*}

\subsubsection*{Case of compactly supported drift functions with $\sigma^* = 1$}

From \cite{denis2024nonparametric}, \textit{Proof of Theorem 1}, we have for all $\mathbf{f}^* \in \mathbf{F}_c(\beta, R)$ and under Assumption~\ref{ass:Reg},
\begin{equation}\label{eq:cp-ex-risk1}
    \begin{aligned}
        \mathbb{E}\left[\mathcal{R}(g_{\widehat{\mathbf{f}}}) - \mathcal{R}(g_{\mathbf{f}^*})\right] \leq 2\sum_{i \in \mathcal{Y}}\mathbb{E}\left[\left|F_{\widehat{\mathbf{b}}}^i(\bar{X}) - F_{\mathbf{b}^*}^i(\bar{X}) \right|\right] + C\left(N^{-1/2} + \Delta\right),
    \end{aligned}
\end{equation}
where the constant $C>0$ depends on $\mathbf{p}^*$ and $\sigma^*$. For all $i \in \mathcal{Y}$ and $b_i^* \in \Sigma_c(\beta, R)$, we have
\begin{equation}\label{eq:cp-ex-risk2}
    \begin{aligned}
        F_{\widehat{\mathbf{b}}}^i(\bar{X}) - F_{\mathbf{b}^*}^i(\bar{X}) = &~ \sum_{k=0}^{n-1}\left(\widehat{b}_i(X_{k\Delta}) - b_i^*(X_{k\Delta})\right)\left(X_{(k+1)\Delta} - X_{k\Delta}\right) - \dfrac{\Delta}{2}\sum_{k=0}^{n-1}\left(\widehat{b}_i^2(X_{k\Delta}) - b_i^{*2}(X_{k\Delta})\right)\\
        = &~ \sum_{k=0}^{n-1}\left(\widehat{b}_i(X_{k\Delta}) - b_i^*(X_{k\Delta})\right)\left[\left(W_{(k+1)\Delta} - W_{k\Delta}\right) + \int_{k\Delta}^{(k+1)\Delta}b_Y^*(X_s)ds - \Delta b_i^*(X_{k\Delta})\right]\\
        & -\dfrac{\Delta}{2}\sum_{k=0}^{n-1}\left(\widehat{b}_i(X_{k\Delta}) - b_i^*(X_{k\Delta})\right)^2. 
    \end{aligned}
\end{equation}
For all $k \in \{0, \ldots, n-1\}$, set
\begin{equation*}
    V_k = \left(\widehat{b}_i(X_{k\Delta}) - b_i^*(X_{k\Delta})\right)\left(W_{(k+1)\Delta} - W_{k\Delta}\right).
\end{equation*}
Using the independence of the increments of the Brownian motion $(W_t)_{t \geq 0}$ and the Cauchy-Schwarz inequality, we obtain for all $i \in \mathcal{Y}$ and $b_i^* \in \Sigma_c(\beta, R)$,
\begin{equation}\label{eq:cp-ex-risk3}
    \begin{aligned}
        &\mathbb{E}\left[\left|\sum_{k=0}^{n-1}\left(\widehat{b}_i(X_{k\Delta}) - b_i^*(X_{k\Delta})\right)\left(W_{(k+1)\Delta} - W_{k\Delta}\right)\right|\right]\\
        &\leq \left(\mathbb{E}\left[\left|\sum_{k=0}^{n-1}\left(\widehat{b}_i(X_{k\Delta}) - b_i^*(X_{k\Delta})\right)\left(W_{(k+1)\Delta} - W_{k\Delta}\right)\right|^2\right] \right)^{1/2}\\
        &= \left(\sum_{k=0}^{n-1}\mathbb{E}\left[\left(\widehat{b}_i(X_{k\Delta}) - b_i^*(X_{k\Delta})\right)^2\right]\mathbb{E}\left[\left(W_{(k+1)\Delta} - W_{k\Delta}\right)^2\right] \right)^{1/2} \\
        &= \left(\mathbb{E}\left[\left\|\widehat{b}_i - b_i^*\right\|_n^2\right]\right)^{1/2}.
    \end{aligned}
\end{equation}
In addition, since the drift functions are compactly supported, using the Cauchy-Schwarz inequality, we obtain
\begin{equation}\label{eq:cp-ex-risk4}
    \begin{aligned}
        \mathbb{E}\left[\left|\sum_{k=0}^{n-1}\left(\widehat{b}_i(X_{k\Delta}) - b_i^*(X_{k\Delta})\right)\int_{k\Delta}^{(k+1)\Delta}b_Y^*(X_s)ds - \Delta b_i^*(X_{k\Delta})\right|\right] \leq \underset{i \in \mathcal{Y}}{\max}{\|b_i^*\|_{\infty}}\left(\mathbb{E}\left[\left\|\widehat{b}_i - b_i^*\right\|_n^2\right]\right)^2.
    \end{aligned}
\end{equation}
We deduce from Equations~\eqref{eq:cp-ex-risk4}, \eqref{eq:cp-ex-risk3} and \eqref{eq:cp-ex-risk2} that there exists a constant $C>0$ depending on $\|b_0^*\|_{\infty}$ and $\|b_1^*\|_{\infty}$ such that for all $b_i^* \in \Sigma_c(\beta, R)$,
\begin{equation*}
    \begin{aligned}
        \mathbb{E}\left[\left|F_{\widehat{b}_i}(\bar{X}) - F_{b_i^*}(\bar{X})\right|\right] \leq C\left(\mathbb{E}\left[\left\|\widehat{b}_i - b_i^*\right\|_n^2\right]\right)^2.
    \end{aligned}
\end{equation*}
Moreover, from Theorem~\ref{thm:upper-bound-drift} and Lemma~\ref{lm:ChangeNorm}, we have for all $i \in \mathcal{Y}$ and $b_i^* \in \Sigma_c(\beta, R)$,
\begin{equation}\label{eq:cp-ex-risk6}
    \begin{aligned}
        \mathbb{E}\left[\left|F_{\widehat{b}_i}(\bar{X}) - F_{b_i^*}(\bar{X})\right|\right]\leq C\left(\mathbb{E}\left[\left\|\widehat{b}_i - b_i^*\right\|_{n,i}\right]\right)^{1/2} \leq CN^{-\beta/(2\beta+1)}.
    \end{aligned}
\end{equation}
Finally, for $\Delta = \mathrm{O}\left(N^{-1}\right)$, we deduce from Equations~\eqref{eq:cp-ex-risk6} and \eqref{eq:cp-ex-risk1} that
\begin{equation*}
     \underset{\mathbf{f}^* \in \mathbf{F}_{c}(\beta, R)}{\sup}{~\mathbb{E}\left[\mathcal{R}(g_{\widehat{\mathbf{f}}}) - \mathcal{R}(g_{\mathbf{f}^*})\right]}  \leq CN^{-\beta/(2\beta+1)}.
\end{equation*}
\end{proof}

\subsubsection{Proof of Theorem~\ref{thm:Minimax-PlugIn2}}

Let $x \mapsto \mathcal{L}^x$ be the local time of a diffusion process $X$ defined for all $x \in \mathbb{R}$ by
\begin{equation}\label{eq:local-time}
    \mathcal{L}^x = \underset{\varepsilon \rightarrow 0}{\lim}{~\dfrac{1}{2\varepsilon}\int_{0}^{1}{\mathds{1}_{|X_s - s|\leq \varepsilon}ds}}.
\end{equation}
The proof of Theorem~\ref{thm:Minimax-PlugIn2} is based on the following lemmas. 
\begin{lemme}\label{lm:Proba-LT}
    Let $\theta > 0$ be a positive real number, and $A,B \in \mathbb{R}$ be two other real numbers such that $A < B$. Under Assumptions~\ref{ass:Reg} and \ref{ass:Ell}, there exist constants $c, C > 1$ such that,
    \begin{equation*}
        \mathbb{P}\left(\underset{x \in [A,B]}{\inf}{~\mathcal{L}^x} \geq \theta\right) \geq  \dfrac{\theta}{C}\left[1 - \exp\left(-\dfrac{\theta^2}{2\sigma_1^{*2}}\right)\right]\exp\left(-c(|A| + |B| + \mathbf{b}_{\infty}^* + 3\theta)^2\right),
    \end{equation*}
    where  $\mathbf{b}_{\infty}^* = \|b_0^*\|_{\infty} + \|b_1^*\|_{\infty}$.
\end{lemme}

\begin{lemme}\label{lm:Identity-NormX}
    Suppose that $N,n \rightarrow \infty$ and consider the following diffusion model:
    $$dX_t = b(X_t)dt + \sigma(X_t)dW_t,$$
    where the drift coefficient $b \neq 0$ is nonnegative and compactly supported, and $b$ and $\sigma$ satisfy Assumptions~\ref{ass:Reg} and \ref{ass:Ell}. Consider a projection estimator $\widehat{b}$ of $b$ built from independent copies of $\bar{X} = (X_{k\Delta_n})_{k \in [\![0, n]\!]}$, discrete observation of $X$, such that $\widehat{b}(x) = 0$ for all $x \notin \mathrm{Supp}(b)$. $\widehat{b}$ is consistent, that is,
    $$\mathbb{E}\left[\left\|\widehat{b} - b\right\|_n^2\right] \rightarrow 0 ~~ \mathrm{as} ~~ N,n \rightarrow \infty.$$
    There exists a constant $c \in (0,1)$ independent of $N$ and $n$ such that
    \begin{equation*}
        \mathbb{P}\left(\int_{0}^{1}(\widetilde{b}-b)(X_s)ds \geq \alpha \left\|\widetilde{b}-b\right\|_X\right) \geq c,
    \end{equation*}
    where $\alpha = \min\left\{(1/4)\int_{\mathrm{Supp}(b)}b(x)dx, 1/3\right\} > 0$ is a fixed real number strictly smaller than $1$ and $\widetilde{b}$ is a truncated estimator of $b$ given for all $x \in \mathbb{R}$ by
    $$\widetilde{b}(x) = \widehat{b}(x)\mathds{1}_{0 \leq \widehat{b}(x) \leq \log^{1/2}(N)} + \log^{1/2}(N)\mathds{1}_{\widehat{b}(x) > \log(N)}.$$
\end{lemme}

The proofs of Lemmas~\ref{lm:Proba-LT} and \ref{lm:Identity-NormX} are provided in the appendix.

\begin{proof}[\textbf{Proof of Theorem}~\ref{thm:Minimax-PlugIn2}]
Consider the following diffusion model
\begin{equation}\label{eq:model_inf}
    dX_t = b_Y^*(X_t)dt + \sigma^*(X_t)dW_t, ~~ t \in [0,1], ~~ X_0 = 0,
\end{equation}
where the coefficients satisfy Assumptions~\ref{ass:Reg}, \ref{ass:Ell}, and $Y \in \{0,1\}$ is independent of $W = (W_t)_{t \geq 0}$ with discrete law $\mathbf{p}^* = (1/2, 1/2)$. We set $b_0^* = 0$ and $b_1^* = \sigma^{*2}$. Then, the regression function $\Phi_{\mathbf{f}^*}$ is given for any diffusion process $X = (X_t)_{t \in [0,1]}$, solution of Equation~\eqref{eq:model_inf}, by
\begin{equation*}
    \begin{aligned}
        \Phi_{\mathbf{f}^*}(X) = &~ \mathbb{P}_{X,Y}(Y=1|X) = \dfrac{\exp\left(F_{\mathbf{b}^*}^1(X)\right)}{\exp\left(F_{\mathbf{b}^*}^0(X)\right) + \exp\left(F_{\mathbf{b}^*}^1(X)\right)}\\
        = &~ \dfrac{\exp\left(F_{\mathbf{b}^*}^1(X)\right)}{1 + \exp\left(F_{\mathbf{b}^*}^1(X)\right)} = \dfrac{Q_{b_1^*}(X)}{1 + Q_{b_1^*}(X)} = \Phi_{b_1^*}(X),
    \end{aligned}
\end{equation*}
where
\begin{align*}
    F_{\mathbf{b}^*}^0(X) := &~ \int_{0}^{1}{\dfrac{b_0^*}{\sigma^{*2}}(X_s)dX_s} - \dfrac{1}{2}\int_{0}^{1}{\dfrac{b_0^{*2}}{\sigma^{*2}}(X_s)ds} = 0\\
    F_{\mathbf{b}^*}^1(X) := &~ \int_{0}^{1}{\dfrac{b_1^*}{\sigma^{*2}}(X_s)dX_s} - \dfrac{1}{2}\int_{0}^{1}{\dfrac{b_1^{*2}}{\sigma^{*2}}(X_s)ds} = X_1 - \dfrac{1}{2}\int_{0}^{1}{b_1^*(X_s)ds}\\
    Q_{b_1^*}(X) = &~ \exp\left(F_{\mathbf{b}^*}^1(X)\right) = \exp\left(X_1 - \dfrac{1}{2}\int_{0}^{1}{b_1^*(X_s)ds}\right).
\end{align*}
For any $\mathbf{f}^* \in \mathbf{F}(\beta, R)$, we have
\begin{equation*}
    g_{\mathbf{f}^*}(X) = \mathds{1}_{\Phi_{\mathbf{f}^*}(X) \geq 1/2} = \mathds{1}_{Q_{b_1^*}(X) \geq 1} = g_{b_1^*}(X).
\end{equation*}
Since $b_1^* = \sigma^{*2} \geq 0$, for any nonparametric estimator $\widehat{b}_1$ of $b_1^*$, we consider the truncated estimator $\w{b}_1$ given for all $x \in \mathbb{R}$ by
\begin{equation}\label{eq:TrEst}
    \w{b}_1(x) = \widehat{b}_1(x)\mathds{1}_{0 \leq \widehat{b}_1(x) \leq \log^{1/2}(N)} + \log^{1/2}(N)\mathds{1}_{\widehat{b}_1(x) \geq \log^{1/2}(N)} \geq 0,
\end{equation}
The plug-in classifier $\widehat{g}$ is given by
\begin{equation*}
    \widehat{g}(X) = \mathds{1}_{F_{\w{\mathbf{b}}}^1(X) \geq 0} = g_{\w{b}_1}(X).
\end{equation*}
Then we obtain
\begin{equation}\label{eq:Lower-Bound-0}
    \begin{aligned}
        \underset{\widehat{\mathbf{f}}}{\inf}{~\underset{\mathbf{f}^* \in \mathbf{F}(\beta, R)}{\sup}{~\mathbb{E}\left[\mathcal{R}(g_{\widehat{\mathbf{f}}}) - \mathcal{R}(g_{\mathbf{f}^*})\right]}} = &~\underset{\widehat{b}_1}{\inf}{~\underset{b_1^* \in \Sigma(\beta, R)}{\sup}{~\mathbb{E}\left[\mathcal{R}(g_{\w{b}_1}) - \mathcal{R}(g_{b_1^*})\right]}}\\
        \geq &~ \underset{\widehat{b}_1}{\inf}{~\underset{b_1^* \in \w{\Sigma}(\beta, R)}{\sup}{~\mathbb{E}\left[\mathcal{R}(g_{\w{b}_1}) - \mathcal{R}(g_{b_1^*})\right]}},
    \end{aligned}
\end{equation}
where 
\begin{equation*}
    \w{\Sigma}(\beta, R) := \left\{f \in \Sigma(\beta, R) \cap \mathcal{C}_b(\mathbb{R}) : ~ f \geq 0\right\}. 
\end{equation*}
We have the following:
\begin{equation*}
\begin{aligned}
    \mathbb{E}\left[\mathcal{R}(g_{\w{b}_1}) - \mathcal{R}(g_{b_1^*})\right] = &~ \mathbb{E}\left[\left|2\Phi_{b_1^*}(X) - 1\right|\mathds{1}_{g_{\w{b}_1}(X) \neq g_{{b}_1^*}(X)}\right]\\
    \geq &~ \varepsilon\mathbb{P}\left(\left\{g_{\w{b}_1}(X) \neq g_{{b}_1^*}(X)\right\} \cap \left\{\left|2\Phi_{{b}_1^*}(X) - 1\right| > \varepsilon\right\}\right).
\end{aligned}
\end{equation*}
The random events $\left\{g_{\w{b}_1}(X) \neq g_{{b}_1^*}(X)\right\}$ and $\left\{\left|2\Phi_{{b}_1^*}(X) - 1\right| > \varepsilon\right\}$ satisfy the following:
\begin{align*}
    \left\{g_{\w{b}_1}(X) \neq g_{b_1^*}(X)\right\} = &~ \left\{\Phi_{b_1^*}(X) > \dfrac{1}{2} \geq \Phi_{\w{b}_1}(X)\right\} \cup \left\{\Phi_{b_1^*}(X) \leq \dfrac{1}{2} < \Phi_{\w{b}_1}(X)\right\}, \\
    \left\{\left|2\Phi_{b_1^*}(X) - 1\right| \geq \varepsilon\right\} = &~ \left\{\Phi_{{b}_1^*}(X) \geq \dfrac{1+\varepsilon}{2}\right\} \cup \left\{\Phi_{b_1^*}(X) \leq \dfrac{1-\varepsilon}{2}\right\}.
\end{align*}
We deduce that
\begin{equation*}
     \begin{aligned}
        &\left\{g_{\w{b}_1}(X) \neq g_{b_1^*}(X)\right\} \cap \left\{\left|2\Phi_{b_1^*}(X) - 1\right| \geq \varepsilon\right\}\\
        &=  \left\{\Phi_{b_1^*}(X) \geq \dfrac{1+\varepsilon}{2}, ~~ \Phi_{\w{ b}_1}(X) \leq \dfrac{1}{2}\right\} \cup \left\{\Phi_{b_1^*}(X) \leq \dfrac{1-\varepsilon}{2}, ~~ \Phi_{\w{b}_1}(X) > \dfrac{1}{2} \right\}\\
        &=  \left\{-\left(\Phi_{\w{b}_1}(X) - \Phi_{b_1^*}(X)\right) \geq \dfrac{\varepsilon}{2}\right\} \cup \left\{\left(\Phi_{\w{b}_1}(X) - \Phi_{b_1^*}(X)\right) \geq \dfrac{\varepsilon}{2}\right\}\\
        &= \left\{\left|\Phi_{\w{b}_1}(X) - \Phi_{b_1^*}(X)\right| \geq \dfrac{\varepsilon}{2}\right\}.
     \end{aligned}
\end{equation*}
We deduce that
\begin{equation}\label{eq:Lower-Bound-1}
    \begin{aligned}
        \underset{\widehat{b}_1}{\inf}{~\underset{b_1^* \in \w{\Sigma}(\beta, R)}{\sup}{~\mathbb{E}\left[\mathcal{R}(g_{\w{b}_1}) - \mathcal{R}(g_{b_1^*})\right]}} \geq &~ \varepsilon ~ \underset{\widehat{b}_1}{\inf}{\underset{b_1^* \in \w{\Sigma}(\beta, R)}{\sup}{~\mathbb{P}\left(\left|\Phi_{\w{b}_1}(X) - \Phi_{b_1^*}(X)\right| \geq \dfrac{\varepsilon}{2}\right)}}.
    \end{aligned}
\end{equation}
For any estimator $\widehat{b}_1$ of $b_1^*$, we have:
\begin{equation*}
    \begin{aligned}
        \Phi_{\w{b}_1}(X) - \Phi_{{b}_1^*}(X) = &~ \dfrac{Q_{\w{b}_1}(X)}{1 + Q_{\w{b}_1}(X)} - \dfrac{Q_{{b}_1^*}(X)}{1 + Q_{{b}_1^*}(X)}\\
        = &~ \dfrac{Q_{\w{b}_1}(X)\left(1 + Q_{{b}_1^*}(X)\right) - Q_{{b}_1^*}(X)\left(1 + Q_{{\widetilde{b}}_1}(X)\right)}{\left(1 + Q_{\w{b}_1}(X)\right)\left(1 + Q_{{b}_1^*}(X)\right)}\\
        = &~ \dfrac{Q_{\w{b}_1}(X) - Q_{{b}_1^*}(X)}{\left(1 + Q_{\w{b}_1}(X)\right)\left(1 + Q_{{b}_1^*}(X)\right)}.
    \end{aligned}
\end{equation*}
Focusing on $Q_{\w{b}_1}(X)$, and since $\w{b}_1 \geq 0$, we obtain under the event $\{X_1 \leq 0\}$
\begin{equation}\label{eq:Bound-q1chap}
     Q_{\w{b}_1}(X) = \exp\left(X_1 - \dfrac{1}{2}\int_{0}^{1}{\w{b}_1(X_{t})dt}\right) \leq 1.  
\end{equation}
It follows from Equation~\eqref{eq:Bound-q1chap} that for any estimator $\widehat{b}_1$ of $b_1^*$,
\begin{equation*}
    \begin{aligned}
        \mathbb{P}\left(\left|\Phi_{\w{b}_1}(X) - \Phi_{{b}_1^*}(X)\right| \geq \dfrac{\varepsilon}{2}\right) \geq &~ \mathbb{P}\left(\dfrac{\left|Q_{\w{b}_1}(X) - Q_{{b}_1^*}(X)\right|}{1 + Q_{{b}_1^*}(X)} \geq \dfrac{\varepsilon}{2}\left(1 + Q_{\w{b}_1}(X)\right)\right)\\
        \geq &~ \mathbb{P}\left(\dfrac{\left|Q_{\w{b}_1}(X) - Q_{{b}_1}(X)\right|}{1 + Q_{{b}_1^*}(X)} \geq \varepsilon \Biggm\vert \mathcal{A}_{\gamma}\right)\mathbb{P}\left(\mathcal{A}_{\gamma}\right),
    \end{aligned}
\end{equation*}
where $\gamma > 0$ is a numerical constant to be chosen later and
\begin{equation*}
    \mathcal{A}_{\gamma} := \left\{-\gamma \leq X_1 \leq 0\right\}.
\end{equation*}
Moreover, since $b_1^* \in \w{\Sigma}(\beta, R)$, for all $x \in \mathbb{R}$, $b_1^*(x) \leq \left\|b_1^*\right\|_{\infty} < \infty$. We deduce that 
\begin{equation}\label{eq:InfQ1}
    Q_{b_1^*}(X) \geq \exp\left(X_1 - \dfrac{\left\|b_1^*\right\|_{\infty}}{2}\right),
\end{equation}
and we obtain from Equation~\eqref{eq:InfQ1} that
\begin{equation}\label{eq:Lower-Bound-2}
    \begin{aligned}
        \mathbb{P}\left(\left|\Phi_{\w{b}_1}(X) - \Phi_{{b}_1^*}(X)\right| \geq \dfrac{\varepsilon}{2}\right) \geq &~ \mathbb{P}\left(\dfrac{Q_{{b}_1^*}(X) \left|\dfrac{Q_{\w{b}_1}(X) }{Q_{{b}_1^*}(X) } - 1\right|}{1 + Q_{{b}_1^*}(X)} \geq \varepsilon \Biggm\vert \mathcal{A}_{\gamma}\right) \mathbb{P}\left(\mathcal{A}_{\gamma}\right)\\
        = &~ \mathbb{P}\left(\left|\dfrac{Q_{\w{b}_1}(X) }{Q_{{b}_1^*}(X)} - 1\right| \geq \varepsilon\left(1 + \dfrac{1}{Q_{{b}_1^*}(X)}\right) \Biggm\vert \mathcal{A}_{\gamma}\right)\mathbb{P}\left(\mathcal{A}_{\gamma}\right)\\
       \geq &~ \mathbb{P}\left(\left|\dfrac{Q_{\w{b}_1}(X) }{Q_{{b}_1^*}(X)} - 1\right| \geq \varepsilon\left(1 + \exp\left(\gamma + \dfrac{\left\|b_1^*\right\|_{\infty}}{2}\right)\right) \Biggm\vert \mathcal{A}_{\gamma}\right) \mathbb{P}\left(\mathcal{A}_{\gamma}\right).
    \end{aligned}
\end{equation}
From Equation~\eqref{eq:Transition-density}, we have
\begin{equation}\label{eq:ProbInf}
    \mathbb{P}\left(\mathcal{A}_{\gamma}\right) = \mathbb{P}\left(- \gamma \leq X_1 \leq 0\right) = \int_{-\gamma}^{0}{\Gamma_X(1,x)dx} \geq \dfrac{1}{C}\int_{-\gamma}^{0}{\exp\left(-cx^2\right)dx} \geq \dfrac{\gamma}{C}\exp(-c\gamma^2),
\end{equation}
where $C,c>1$ are constants. Set 
\begin{equation*}
    \mathbf{P}(.) = \mathbb{P}\left(. \Biggm\vert \mathcal{A}_{\gamma}\right) ~~ \mathrm{and} ~~ \Xi = 1 + \exp\left(\gamma + \dfrac{\left\|b_1^*\right\|_{\infty}}{2}\right).
\end{equation*}
Then for all $\varepsilon > 0$, we obtain from Equations~\eqref{eq:ProbInf}, \eqref{eq:Lower-Bound-2}, \eqref{eq:Lower-Bound-1} and \eqref{eq:Lower-Bound-0}, 
\begin{equation}\label{eq:Lower-Bound-3}
    \begin{aligned}
        \underset{\widehat{\mathbf{f}}}{\inf}{~\underset{\mathbf{f}^* \in \mathbf{F}(\beta, R)}{\sup}{~\mathbb{E}\left[\mathcal{R}(g_{\widehat{\mathbf{f}}}) - \mathcal{R}(g_{\mathbf{f}^*})\right]}} \geq &~ \varepsilon\underset{\widehat{b}_1}{\inf}{~\underset{b_1^* \in \w{\Sigma}(\beta, R)}{\sup}{~\mathbb{P}\left(\left|\Phi_{\w{b}_1}(X) - \Phi_{{b}_1^*}(X)\right| \geq \dfrac{\varepsilon}{2}\right)}}\\
        \geq &~ \dfrac{\gamma \varepsilon}{C}\exp(-c\gamma^2) ~ \underset{\widehat{b}_1}{\inf}{~\underset{b_1^* \in \w{\Sigma}(\beta, R)}{\sup}{~\mathbf{P}\left(\left|\dfrac{Q_{\w{b}}(X)}{Q_{{b}_1^*}(X) } - 1\right| \geq \varepsilon\Xi\right)}}.
    \end{aligned}
\end{equation}
Moreover, we have
\begin{equation*}
    \begin{aligned}
        \mathbf{P}\left(\left|\dfrac{Q_{\w{b}_1}(X) }{Q_{{b}_1^*}(X)} - 1\right| \geq \varepsilon\Xi\right) \geq &~ \mathbf{P}\left(\left|\exp\left(-\dfrac{1}{2}\int_{0}^{1}(\widetilde{b}_1 - b_1^*)(X_s)ds\right) - 1\right| \geq \varepsilon\Xi\right)\\
        \geq &~ \mathbf{P}\left(\exp\left(-\dfrac{1}{2}\int_{0}^{1}(\widetilde{b}_1 - b_1^*)(X_s)ds\right) - 1 \leq -\varepsilon\Xi\right)\\
        = &~ \mathbf{P}\left(\int_{0}^{1}(\widetilde{b}_1 - b_1^*)(X_s)ds \geq -2\log\left(1-\varepsilon\Xi\right)\right).
    \end{aligned}
\end{equation*}
Note that for all $x \in (-1, +\infty)$, one has $\log(1+x) \geq x - x^2/2$. Then, for all $\varepsilon \in \left(0, \Xi^{-1}\right)$, we have $\log(1-\varepsilon\Xi) \geq -\varepsilon\Xi - \dfrac{\varepsilon^2\Xi^2}{2}$, and for $\varepsilon$ close enough to $0$,
\begin{equation}\label{eq:Lower-Bound-4}
\begin{aligned}
    \mathbf{P}\left(\left|\dfrac{Q_{\w{b}_1}(X) }{Q_{{b}_1}(X)} - 1\right| \geq \varepsilon\Xi\right) \geq &~ \mathbf{P}\left(\int_{0}^{1}(\widetilde{b}_1 - b_1^*)(X_s)ds \geq 2\varepsilon\Xi + \varepsilon^2\Xi^2\right)\\
    \geq &~ \mathbf{P}\left(\int_{0}^{1}(\widetilde{b}_1 - b_1^*)(X_s)ds \geq 3\varepsilon\Xi\right).
\end{aligned}
\end{equation}
From now on, we define a finite subset of $\w{\Sigma}(\beta, R)$, that is, the set of hypotheses $\Sigma^M$ that contains functions $b_1^* \neq 0$ that are $\beta$-H\"older and compactly supported. For this purpose, consider a numerical constant $c_0 > 0$ whose choice is discussed in \cite{denis2021ridge} for the study of the lower bound of the
estimation risk of the drift coefficient and set  
    \begin{align*}
        m = \left\lceil c_0N^{1/(2\beta+1)} \right\rceil, ~~ h = \dfrac{1}{m}, ~~ x_k = \dfrac{k - 1/2}{m},
    \end{align*}
    \begin{equation*}
        \mathfrak{K}_k(x) = 2Rh^{\beta}K\left(\dfrac{x - x_k}{h}\right), ~~ k = 1, \ldots, m, ~~ x \in [0,1],
    \end{equation*}
    where the function $K: \mathbb{R} \longrightarrow [0,\infty)$ satisfies
    $$K \in \Sigma(\beta,1/2) \cap \mathcal{C}^{\infty}(\mathbb{R}), ~~ \mathrm{and} ~~ K(u) > 0 \iff u \in (-1/2, 1/2),$$
    and the functions $\mathfrak{K}_k$ belong to the H\"older class $\Sigma(\beta, R)$ and satisfy 
    \begin{equation*}
        \forall ~ k \in \{1, \ldots, m\}, ~~ \mathfrak{K}_k(x) > 0 \iff x \in \left(\dfrac{k-1}{m}, \dfrac{k}{m}\right),
    \end{equation*}
    (see \cite{tsybakov2008introduction}, page 92). We choose the finite set $\Sigma^{M}$ as follows:
    \begin{equation*}
        \Sigma^{M} := \left\{\mathfrak{b}_j = \sum_{k = 1}^{m}{w^{j}_k\mathfrak{K}_k}, ~ w^{j} = (w^{j}_{1},\cdots,w^{j}_{m-1})\in \{0,1\}^{m} \setminus \{(0, \ldots, 0)\}, ~ j = 0, \cdots, M\right\}.
    \end{equation*}
    Then, for all $b_1^* \in \Sigma^M, ~ \int_{[0,1]}b_1^*(x)dx > 0$, and set: 
    \begin{equation}\label{eq:alpha0}
        \alpha_0 = \underset{1\leq j \leq M}{\min}\int_{[0,1]}\mathfrak{b}_j(x)dx > 0.
    \end{equation}
    We deduce from Equations~\eqref{eq:Lower-Bound-4} and \eqref{eq:Lower-Bound-3} that
\begin{equation}\label{eq:Lower-Bound-5}
     \begin{aligned}
         \underset{\widehat{\mathbf{f}}}{\inf}{~\underset{\mathbf{f}^* \in \mathbf{F}(\beta, R)}{\sup}{~\mathbb{E}\left[\mathcal{R}(g_{\widehat{\mathbf{f}}}) - \mathcal{R}(g_{\mathbf{f}^*})\right]}} \geq &~  \dfrac{\gamma\varepsilon}{C}\exp(-c\gamma^2) ~ \underset{\widehat{b}_1}{\inf}{~\underset{b_1^* \in \w{\Sigma}(\beta, R)}{\sup}{~\mathbf{P}\left(\left|\dfrac{Q_{\w{b}_1}(X) }{Q_{{b}_1^*}(X) } - 1\right| \geq \varepsilon\Xi\right)}}\\
         \geq &~ \dfrac{\gamma\varepsilon}{C}\exp(-c\gamma^2) ~ \underset{\widehat{b}_1}{\inf}{~\underset{b_1^* \in \Sigma^M}{\sup}{~\mathbf{P}\left(\left|\dfrac{Q_{\w{b}_1}(X) }{Q_{{b}_1^*}(X) } - 1\right| \geq \varepsilon\Xi\right)}}\\
         \geq &~ \dfrac{\gamma\varepsilon}{C}\exp(-c\gamma^2) ~ \underset{\widehat{b}_1}{\inf}{~\underset{b_1^* \in \Sigma^M}{\sup}{~\mathbf{P}\left(\int_{0}^{1}(\widetilde{b}_1 - b_1^*)(X_s)ds \geq 3\varepsilon\Xi\right)}}.
     \end{aligned}
\end{equation}
    Since $b_1^* \in \Sigma^M$, all assumptions of Lemma~\ref{lm:Identity-NormX} are satisfied and $\int_{[0,1]}b_1^*(x)dx > 0$. Then, set
    $$\mathcal{A}_{\alpha}^M = \bigcup_{\mathfrak{b} \in \Sigma^M}\left\{\int_{0}^{1}(\w{\mathfrak{b}} - \mathfrak{b})(X_s)ds \geq \alpha \left\|\w{\mathfrak{b}} - \mathfrak{b}\right\|_X\right\} ~~ \mathrm{with} ~~ \alpha = \min\left\{\alpha_0/4, 1/3\right\},$$
    where $\alpha_0$ is given by Equation~\eqref{eq:alpha0}. From Lemma~\ref{lm:Identity-NormX} and Equation~\eqref{eq:Lower-Bound-5}, there exists a constant $\mathfrak{c} \in (0,1)$ such that
    \begin{equation}\label{eq:Lower-Bound-6}
        \begin{aligned}
            &\underset{\widehat{\mathbf{f}}}{\inf}{~\underset{\mathbf{f}^* \in \mathbf{F}(\beta, R)}{\sup}{~\mathbb{E}\left[\mathcal{R}(g_{\widehat{\mathbf{f}}}) - \mathcal{R}(g_{\mathbf{f}^*})\right]}}\\
            & \geq \dfrac{\gamma\varepsilon}{C}\exp(-c\gamma^2) ~ \underset{\widehat{b}_1}{\inf}{~\underset{b_1^* \in \Sigma^M}{\sup}{~\mathbf{P}\left(\int_{0}^{1}(\widetilde{b}_1 - b_1^*)(X_s)ds \geq 3\varepsilon\Xi \biggm\vert \mathcal{A}_{\alpha}^M\right)\mathbf{P}\left(\mathcal{A}_{\alpha}^M\right)}}\\
            & \geq \dfrac{\mathfrak{c}\gamma\varepsilon}{C}\exp(-c\gamma^2) ~ \underset{\widehat{b}_1}{\inf}{~\underset{b_1^* \in \Sigma^M}{\sup}{~\mathbf{P}\left(\left\|\w{b}_1 - b_1^*\right\|_X \geq \dfrac{3\varepsilon\Xi}{\alpha} \biggm\vert \mathcal{A}_{\alpha}^M\right)}}.
        \end{aligned}
    \end{equation}
    Moreover, from the occupation formula, we have
    \begin{equation}\label{eq:lower-Bound-7}
       \left\|\w{b}_1 - b_1^*\right\|_{X}^2 = \int_{0}^{1}{(\w{b}_1 - b_1^*)^2(X_s)ds} = \int_{0}^{1}{(\w{b}_1 - b_1^*)^2(x)\mathcal{L}^x dx} \geq \left(\underset{x \in [0, 1]}{\inf}{\mathcal{L}^x}\right)\left\|\w{b}_1 - b_1^*\right\|^2,
    \end{equation}
    where $x \mapsto \mathcal{L}^x$ is the local time given by Equation~\eqref{eq:local-time}. Let $\theta > 0$ be a real number to be chosen carefully, and consider the events 
    $$\mathcal{A}_{\theta} = \left\{\underset{x \in [0, 1]}{\inf}{\mathcal{L}^x} \geq \theta\right\} ~~ \mathrm{and} ~~ \mathcal{A} = \mathcal{A}_{\gamma} \cap \mathcal{A}_{\alpha}^M \cap \mathcal{A}_{\theta}.$$
    We now proceed to the choices of $\gamma > 0$ and $\theta > 0$ taking into account the assumptions made in the respective proofs of Lemmas~\ref{lm:Proba-LT} and \ref{lm:Identity-NormX}. First, the result of Lemma~\ref{lm:Proba-LT} is proved using the event $\mathcal{A}(x) = \left\{X_1 \leq -|x| + x - \|b_0^*\|_{\infty} - \|b_1^*\|_{\infty} - 2\theta\right\} = \left\{X_1 \leq -|x| + x - \|b_1^*\|_{\infty} - 2\theta\right\}$ where $x$ lies in a compact interval, in the present case, $x \in [0, 1]$. Second, from the proof of Lemma~\ref{lm:Identity-NormX}, $\theta$ is chosen as follows:
    \begin{equation*}
        \theta = \dfrac{\alpha\mathfrak{C}/2 + \underset{\mathfrak{b} \in \Sigma^M}{\max}\|\mathfrak{b}\|_{\infty}}{(1/2)\underset{\mathfrak{b} \in \Sigma^M}{\min}\int_{[0,1]}\mathfrak{b}(u)du}, ~~ \mathrm{where} ~~ \mathfrak{C} = \max\left\{4C, \dfrac{4\underset{\mathfrak{b} \in \Sigma^M}{\max}\|\mathfrak{b}\|_{\infty}}{\underset{\mathfrak{b} \in \Sigma^M}{\min}\int_{[0,1]}\mathfrak{b}(u)du - 2\alpha}\right\} ~~ \mathrm{and} ~~ C>1.
    \end{equation*}
    Finally, the choice of $\gamma > 0$ should ensure that for all $x \in [0, 1], ~~ \mu(\mathcal{A}(x) \cap \mathcal{A}_{\gamma}) > 0$, where $\mu$ is the Lebesgue measure. More precisely, $\gamma < - \|b_1^*\|_{\infty} - 2\theta = -|x| + x - \|b_0^*\|_{\infty} - \|b_1^*\|_{\infty} - 2\theta$ for all $x \in [0, 1]$. Consequently, we set
    $$\gamma = -\gamma_0\left(1 + \underset{\mathfrak{b} \in \Sigma^M}{\max}\|\mathfrak{b}\|_{\infty} + 2\theta\right),$$
    where the numerical constant $\gamma_0 > 1$ is assumed to be large enough. For simplicity in the notation, we re-define the probability $\mathbf{P}$ as follows:
    $$\mathbf{P} = \mathbb{P}\left(. \biggm\vert \mathcal{A}\right) = \mathbb{P}\left(. \biggm\vert \mathcal{A}_{\gamma} \cap \mathcal{A}_{\alpha}^M \cap \mathcal{A}_{\theta}\right),$$
    and from Lemma~\ref{lm:Proba-LT} and Equations~\eqref{eq:lower-Bound-7} and \eqref{eq:Lower-Bound-6}, we obtain:
    \begin{equation*}
        \begin{aligned}
             \underset{\widehat{\mathbf{f}}}{\inf}{~\underset{\mathbf{f}^* \in \mathbf{F}(\beta, R)}{\sup}{~\mathbb{E}\left[\mathcal{R}(g_{\widehat{\mathbf{f}}}) - \mathcal{R}(g_{\mathbf{f}^*})\right]}} \geq &~ \dfrac{\mathfrak{c}\gamma\varepsilon}{C}\exp(-c\gamma^2) ~ \underset{\widehat{b}_1}{\inf}{~\underset{b_1^* \in \Sigma^M}{\sup}{~\mathbf{P}\left(\left\|\w{b}_1 - b_1^*\right\| \geq \dfrac{3\varepsilon\Xi}{\alpha\sqrt{\theta}} \right)\mathbf{P}\left(\underset{x \in [-1,1]}{\inf}{\mathcal{L}^x} \geq \theta\right)}}\\
             \geq &~ \Lambda\varepsilon~\underset{\widehat{b}_1}{\inf}{~\underset{b_1^* \in \Sigma^M}{\sup}{~\mathbf{P}\left(\left\|\w{b}_1 - b_1^*\right\| \geq \dfrac{3\varepsilon\Xi}{\alpha\sqrt{\theta}} \right)}},
        \end{aligned}
    \end{equation*}
    where
    $$\Lambda = \dfrac{\mathfrak{c}\gamma\theta}{C^2}\exp(-c\gamma^2)\left[1 - \exp\left(-\dfrac{\theta^2}{2\underset{\mathfrak{b} \in \Sigma^M}{\max}\|\mathfrak{b}\|_{\infty}^{2}}\right)\right]\exp\left(-c(1 + \underset{\mathfrak{b} \in \Sigma^M}{\max}\|\mathfrak{b}\|_{\infty} + 3\theta)^2\right).$$
   Finally, from the proof of Theorem 4.7 in \cite{denis2021ridge} and from \cite{tsybakov2008introduction}, \textit{Chapter 2, Theorem 2.5, p.99}, there exists a numerical constant $\eta_0 > 0$ such that 
   $$\dfrac{3\varepsilon\Xi}{\alpha\sqrt{\theta}} = \eta_0 N^{-\beta/(2\beta+1)} \iff \varepsilon =  \dfrac{\alpha\eta_0\sqrt{\theta}}{3\Xi}N^{-\beta/(2\beta+1)},$$
   and 
   $$\underset{\widehat{b}_1}{\inf}{~\underset{b_1^* \in \Sigma^M}{\sup}{~\mathbf{P}\left(\left\|\w{b}_1 - b_1^*\right\| \geq \dfrac{3\varepsilon\Xi}{\alpha\sqrt{\theta}} \right)}} \geq \dfrac{\sqrt{M}}{1+\sqrt{M}}\left(1 - 2\lambda - \sqrt{\dfrac{2\lambda}{\log(M)}}\right) > 0,$$
   where $M \geq 2^{m/8}$ (see Lemma 2.9 in \cite{tsybakov2008introduction}, Chapter 2, p.104) and $\lambda \in (0, 1/8)$, which concludes the proof.
\end{proof}

\subsubsection{Proof of Theorem~\ref{thm:PlugIn-Unbounded}}

\begin{proof}
    We start with the upper bound of the worst excess risk of $\widehat{g} = g_{\widehat{\mathbf{f}}}$, and recall that since the diffusion coefficient $\sigma^*$ is known and $\sigma^* = 1$, we deduce from Equation~\eqref{eq:EQ0} that for any $\mathbf{f}^* \in \mathbf{F}(\beta, R)$,
    \begin{equation}\label{eq:Rcl0}
        \mathbb{E}\left[\mathcal{R}(g_{\widehat{\mathbf{f}}}) - \mathcal{R}(g_{\mathbf{f}^*})\right] \leq C\left(\sqrt{\Delta} + \dfrac{1}{p_{\min}\sqrt{N}} + \log(N)\mathbb{E}\left[\sum_{i=0}^{1}{\left\|\w{b}_i - b_i^*\right\|_n}\right]\right).
    \end{equation}
    For $i \in \mathcal{Y}$, set $b_{A_{N_i}} = b^{*}_{i}\mathds{1}_{[-A_{N_i}, A_{N_i}]}$. Since the drift function $b_i^*$ satisfies
    \begin{equation*}
        |b_i^*(x)| \leq C^*(1 + |x|), ~ x \in \mathbb{R},
    \end{equation*}
 by Proposition~\ref{prop:ChangeProba} and Equation~\eqref{eq:Moments}, there exists a constant $C>0$ such that
    \begin{equation}\label{eq:Rb0}
        \begin{aligned}
            \mathbb{E}\left[\left\|\w{b}_i - b_i^*\right\|_n\right] \leq &~ C\exp\left(\sqrt{c\log(N)}\right)\mathbb{E}\left[\left\|\w{b}_i - b_i^*\right\|_{n,i}\right]\\
            \leq &~ C\exp\left(\sqrt{c\log(N)}\right)\mathbb{E}\left[\left\|\w{b}_i - b_{A_{N_i}}\right\|_{n,i}\mathds{1}_{N_i > 1}\right]\\
            & + C\sqrt{\mathbb{E}\left[\mathds{1}_{N_i > 1}\underset{t \in [0,1]}{\sup}{\mathbb{P}_i\left(|X_t| > A_{N_i}\right)}\right] + \mathbb{P}\left(N_i \leq 1\right)}.
        \end{aligned}
    \end{equation}
    For $A_{N_i} = \sqrt{\dfrac{6\beta}{2\beta+1}\log(N_i)}$ and $K_{N_i} = N_i^{1/(2\beta+1)}\log^{-5/2}(N_i)$, we obtain from \cite{denis2024nonparametric}, \textit{Theorem 5},
    \begin{equation}\label{eq:Rb1}
        \mathbb{E}\left[\left\|\w{b}_i - b_{A_{N_i}}\right\|_{n,i}\mathds{1}_{N_i > 1}\right] \leq C\log^{3\beta}(N)N^{-\beta/(2\beta+1)},
    \end{equation}
    where $C>0$ is a constant, and from the proof of \textit{Theorem 6} in \cite{denis2024nonparametric},
    \begin{equation}\label{eq:Rb2}
        \mathbb{E}\left[\mathds{1}_{N_i > 1}\underset{t \in [0,1]}{\sup}{\mathbb{P}^{(i)}\left(|X_t| > A_{N_i}\right)}\right] \leq \mathbb{E}\left[\dfrac{\mathds{1}_{N_i > 1}}{A_{N_i}}\exp\left(-\dfrac{A_{N_i}^2}{3}\right)\right],
    \end{equation}
    and finally, 
    \begin{equation}\label{eq:Rb3}
        \mathbb{P}\left(N_i \leq 1\right) = (1-p_i^*)^N + Np_i^*(1-p_i^*)^{N-1} \leq 2N(1-p_i^*)^{N-1}.
    \end{equation}

    For $A_{N_i} = \sqrt{\dfrac{6\beta}{2\beta+1}\log(N_i)}$ and using Jensen's inequality and \textit{Lemma 4.1} in \cite{gyorfi2006distribution}, \textit{Chapter 4, p.61}, we deduce from Equations~\eqref{eq:Rb0}, \eqref{eq:Rb1}, \eqref{eq:Rb2} and \eqref{eq:Rb3} that
    \begin{equation}\label{eq:Rcl1}
        \begin{aligned}
            \mathbb{E}\left[\left\|\w{b}_i - b_i^*\right\|_n\right] \leq &~ C\exp\left(\sqrt{c\log(N)}\right)\left[N^{-\beta/(2\beta+1)} + \mathbb{E}\left[\left(\dfrac{\mathds{1}_{N_i > 1}}{N_i}\right)^{-\beta/(2\beta+1)}\right]\right]\\
            \leq &~ C\exp\left(\sqrt{c\log(N)}\right)\left[N^{-\beta/(2\beta+1)} + \left(\mathbb{E}\left[\dfrac{\mathds{1}_{N_i > 1}}{N_i}\right]\right)^{-\beta/(2\beta+1)}\right]\\
            \leq &~ C\exp\left(\sqrt{c\log(N)}\right)N^{-\beta/(2\beta+1)}.
        \end{aligned}
    \end{equation}
    Finally, for $\Delta = \mathrm{O}(N^{-1})$, we conclude from Equations~\eqref{eq:Rcl0} and \eqref{eq:Rcl1} that
    \begin{equation*}
        \underset{\mathbf{f}^* \in \mathbf{F}(\beta, R)}{\sup}{~\mathbb{E}\left[\mathcal{R}(g_{\widehat{\mathbf{f}}}) - \mathcal{R}(g_{\mathbf{f}^*})\right]} \leq C\exp\left(\sqrt{c\log(N)}\right)N^{-\beta/(2\beta+1)}.
    \end{equation*}
\end{proof}

\section*{Declarations}

\paragraph{Conflict of interest} I have no conflict of interest to declare that is relevant to the content of this article. No
funding was received to assist with the preparation of this paper.

\bibliographystyle{ScandJStat}
\bibliography{mabiblio.bib}

@article{cadre2013supervised,
  title={Supervised classification of diffusion paths},
  author={Cadre, B.},
  journal={Math. Methods Statist.},
  volume={22},
  number={3},
  pages={213--225},
  year={2013},
  publisher={Springer}
}

@article{comte2007penalized,
  title={Penalized nonparametric mean square estimation of the coefficients of diffusion processes},
  author={Comte, F. and Genon-Catalot, V. and Rozenholc, Y. and others},
  journal={Bernoulli},
  volume={13},
  number={2},
  pages={ 514 - 543},
  year={2007},
  publisher={Bernoulli Society for Mathematical Statistics and Probability}
}

@article{comte2020nonparametric,
  title={Nonparametric drift estimation for i.i.d. paths of stochastic differential equations},
  author={Comte, F. and Genon-Catalot, V.},
  journal={Ann. Statist.},
  volume={48},
  number={6},
  pages={3336--3365},
  year={2020},
  publisher={Institute of Mathematical Statistics}
}

@book{crow2017introduction,
  title={An introduction to population genetics theory},
  author={Crow, James Franklin},
  year={2017},
  publisher={Scientific Publishers}
}

@article{dacunha1986estimation,
  title={Estimation of the coefficients of a diffusion from discrete observations},
  author={Dacunha-Castelle, Didier and Florens-Zmirou, Danielle},
  journal={Stochastics: An International Journal of Probability and Stochastic Processes},
  volume={19},
  number={4},
  pages={263--284},
  year={1986},
  publisher={Taylor \& Francis}
}

@article{denis2020classif,
  title={Consistent procedures for multiclass classification of discrete diffusion paths},
  author={Denis, C. and Dion-Blanc, C. and Martinez, M.},
  journal={Scand. J. Stat.},
  volume={47},
  number={2},
  pages={516--554},
  year={2020}
}

@article{denis2021ridge,
  title={A ridge estimator of the drift from discrete repeated observations of the solution of a stochastic differential equation},
  author={Denis, Christophe and Dion-Blanc, Charlotte and Martinez, Miguel},
  journal={Bernoulli},
  volume={27},
  number={4},
  pages={2675--2713},
  year={2021},
  publisher={Bernoulli Society for Mathematical Statistics and Probability}
}

@article{denis2024nonparametric,
  title={Nonparametric plug-in classifier for multiclass classification of SDE paths},
  author={Denis, Christophe and Dion-Blanc, Charlotte and Ella-Mintsa, Eddy and Tran, Viet Chi},
  journal={Scandinavian Journal of Statistics},
  volume={51},
  number={3},
  pages={1103--1160},
  year={2024},
  publisher={Wiley Online Library}
}

@article{ella2024nonparametric,
  title={Nonparametric estimation of the diffusion coefficient from iid sde paths},
  author={Ella-Mintsa, Eddy},
  journal={Statistical Inference for Stochastic Processes},
  volume={27},
  number={3},
  pages={585--640},
  year={2024},
  publisher={Springer}
}

@article{ella2025minimax,
  title={Minimax rates of convergence for the nonparametric estimation of the diffusion coefficient from time-homogeneous SDE paths},
  author={Ella-Mintsa, Eddy-Michel},
  journal={Statistical Inference for Stochastic Processes},
  volume={28},
  number={3},
  pages={17},
  year={2025},
  publisher={Springer}
}

@phdthesis{etchegaray2016modelisation,
  title={Mod{\'e}lisation math{\'e}matique et num{\'e}rique de la migration cellulaire},
  author={Etchegaray, Christèle},
  year={2016},
  school={Universit{\'e} Paris-Saclay (ComUE)}
}

@book{friedman1975stochastic,
  title={Stochastic Differential Equations and Applications},
  author={Friedman, A.},
  volume={1},
  year={1975},
  publisher={Academic Press}
}

@article{gadat2020optimal,
  title={Optimal functional supervised classification with separation condition},
  author={Gadat, S. and Gerchinovitz, S. and Marteau, C.},
  journal={Bernoulli},
  volume={26},
  number={3},
  pages={1797--1831},
  year={2020},
  publisher={Bernoulli Society for Mathematical Statistics and Probability}
}

@article{gobet2002lan,
  title={LAN property for ergodic diffusions with discrete observations},
  author={Gobet, E.},
  journal={Ann. Inst. Henri Poincar\'e Probab. Stat.},
  volume={38},
  number={5},
  pages={711--737},
  year={2002},
  publisher={Elsevier}
}

@book{gyorfi2006distribution,
  title={A distribution-free theory of nonparametric regression},
  author={Gy{\"o}rfi, L. and Kohler, M. and Krzyzak, A. and Walk, H.},
  year={2006},
  publisher={Springer Science \& Business Media}
}

@article{hoffmann1999lp,
  title={Lp estimation of the diffusion coefficient},
  author={Hoffmann, M.},
  journal={Bernoulli},
  volume = {5},
  number = {3},
  pages={447--481},
  year={1999},
  publisher={JSTOR}
}

@book{karatzas2014brownian,
	title={Brownian motion and stochastic calculus},
	author={Karatzas, I. and Shreve, S.},
	volume={113},
	year={2014},
	publisher={springer}
}

@article{karoui1998robustness,
  title={Robustness of the Black and Scholes formula},
  author={El Karoui, Nicole and Jeanblanc-Picqu{\`e}, Monique and Shreve, Steven E},
  journal={Mathematical finance},
  volume={8},
  number={2},
  pages={93--126},
  year={1998},
  publisher={Wiley Online Library}
}

@book{kazamaki2006continuous,
  title={Continuous exponential martingales and BMO},
  author={Kazamaki, Norihiko},
  year={2006},
  publisher={Springer}
}

@book{lamberton2011introduction,
  title={Introduction to stochastic calculus applied to finance},
  author={Lamberton, Damien and Lapeyre, Bernard},
  year={2011},
  publisher={Chapman and Hall/CRC}
}

@article{nagai1983asymptotic,
  title={Asymptotic behavior for a nonlinear degenerate diffusion equation in population dynamics},
  author={Nagai, Toshitaka and Mimura, Masayasu},
  journal={SIAM Journal on Applied Mathematics},
  volume={43},
  number={3},
  pages={449--464},
  year={1983},
  publisher={SIAM}
}

@book{revuz2013continuous,
	title={Continuous martingales and Brownian motion},
	author={Revuz, D. and Yor, M.},
	volume={293},
	year={2013},
	publisher={Springer Science \& Business Media}
}

@book {revuzyor1999,
    AUTHOR = {Revuz, D. and Yor, M.},
     TITLE = {Continuous martingales and {B}rownian motion},
    SERIES = {Grundlehren der mathematischen Wissenschaften [Fundamental
              Principles of Mathematical Sciences]},
    VOLUME = {293},
   EDITION = {Third},
 PUBLISHER = {Springer-Verlag, Berlin},
      YEAR = {1999},
     PAGES = {xiv+602},
      ISBN = {3-540-64325-7},
   MRCLASS = {60G44 (60G07 60H05)},
  MRNUMBER = {1725357},
       DOI = {10.1007/978-3-662-06400-9},
       URL = {https://doi.org/10.1007/978-3-662-06400-9},
}

@article{rivera2000time,
  title={Time evolution of a natural clinoptilolite in aqueous medium: conductivity and pH experiments},
  author={Rivera, A and Rodr{\i}guez-Fuentes, G and Altshuler, E},
  journal={Microporous and Mesoporous Materials},
  volume={40},
  number={1-3},
  pages={173--179},
  year={2000},
  publisher={Elsevier}
}

@article{romanczuk2012active,
  title={Active Brownian particles: From individual to collective stochastic dynamics},
  author={Romanczuk, Pawel and B{\"a}r, Markus and Ebeling, Werner and Lindner, Benjamin and Schimansky-Geier, Lutz},
  journal={The European Physical Journal Special Topics},
  volume={202},
  pages={1--162},
  year={2012},
  publisher={Springer}
}

@book{tsybakov2008introduction,
  title={Introduction to nonparametric estimation},
  author={Tsybakov, A.-B.},
  year={2008},
  publisher={Springer Science \& Business Media}
}

@article{van1995exponential,
  title={Exponential inequalities for martingales, with application to maximum likelihood estimation for counting processes},
  author={Van-de-Geer, S.},
  journal={Ann. Statist.},
  volume={23},
  number={5},
  pages={1779--1801},
  year={1995},
  publisher={JSTOR}
}

@article{venalainen2002meteorological,
  title={Meteorological data for agricultural applications},
  author={Ven{\"a}l{\"a}inen, Ari and Heikinheimo, Martti},
  journal={Physics and Chemistry of the Earth, Parts A/B/C},
  volume={27},
  number={23-24},
  pages={1045--1050},
  year={2002},
  publisher={Elsevier}
}


\section*{Appendix}

\begin{proof}[\textbf{Proof of Lemma}~\ref{lm:ChangeNorm}]
    For all $i,j \in \mathcal{Y}$ such that $i \neq j$, we have
    \begin{equation}\label{eq:Equiv-i-j}
        \begin{aligned}
            \left\|\widehat{b}_i - b_i^*\right\|_{n,j}^2 = &~ \dfrac{1}{n}(\widehat{b}_i - b_i^*)^2(0) + \dfrac{1}{n}\sum_{k=1}^{n-1}\int_{\mathbb{R}}(\widehat{b}_i - b_i^*)^2(x)\Gamma_{j,X}(k\Delta, x)dx\\
            \leq &~ \dfrac{1}{n}(\widehat{b}_i - b_i^*)^2(0) + \dfrac{1}{n}\sum_{k=1}^{n-1}\int_{\mathbb{R}}(\widehat{b}_i - b_i^*)^2(x)\Gamma_{i,X}(k\Delta, x)\dfrac{\Gamma_{j,X}(k\Delta, x)}{\Gamma_{i,X}(k\Delta, x)}dx.
        \end{aligned}
    \end{equation}
    Under Assumptions~\ref{ass:Reg}, \ref{ass:Ell} and \ref{ass:Restrict-Model}, from \cite{dacunha1986estimation}, the transition density $(t,x)\mapsto \Gamma_{j,X}(t,x), ~ j \in \mathcal{Y}$ is given for all $(t,x) \in (0,1] \times \mathbb{R}$ by
    \begin{equation*}
        \begin{aligned}
            \Gamma_{j,X}(t,x) = \dfrac{1}{\sqrt{2\pi t\sigma^{*2}(x)}}\exp\left(-\dfrac{1}{2t}S^2(x) + H_j(x)\right)
        \end{aligned}\mathbf{E}\left[\exp\left(t\int_{0}^{1}G_j(z_u(0, S(x)) + \sqrt{t}\widetilde{W}_u)du\right)\right],
    \end{equation*}
    where
    \begin{align*}
        S(x) := &~ \int_{0}^{x}\dfrac{du}{\sigma(u)}, ~~ H_j(x) := \int_{0}^{S(x)}\left(\dfrac{b_j^*}{\sigma^*} - \dfrac{\sigma^{*\prime}}{2}\right) \circ S^{-1}(u)du, ~~ z_u(x,y) = ux + (1-u)y,\\
        G_j := &~ -\dfrac{1}{2}\left[\left(\dfrac{b_j^*}{\sigma^*} - \dfrac{\sigma^{*\prime}}{2}\right)^2\circ S^{-1} + \sigma^* \circ S^{-1} \times \left(\dfrac{b_j^{*\prime}\sigma^* - b_j^* \sigma^{*\prime}}{\sigma^{*2}} - \sigma^{*\prime\prime}\right) \circ S^{-1}\right].
    \end{align*}
    Under Assumption~\ref{ass:Restrict-Model}, for each $j \in \mathcal{Y}, ~ G_j$ is bounded on $\mathbb{R}$. Then, from Equation~\eqref{eq:Equiv-i-j}, there exists a constant $C>0$ depending on $\underset{j \in \mathcal{Y}}{\max}\|G_j\|_{\infty} < \infty$ such that
    \begin{equation}\label{eq:Equiv-i-j-2}
        \begin{aligned}
            \left\|\widehat{b}_i - b_i^*\right\|_{n,j}^2 \leq &~ \dfrac{1}{n}(\widehat{b}_i - b_i^*)^2(0) + \dfrac{C}{n}\sum_{k=1}^{n-1}\int_{\mathbb{R}}(\widehat{b}_i - b_i^*)^2(x)\Gamma_{X,i}(k\Delta, x)\exp\left(\left|H_j(x)\right| + \left|H_i(x)\right|\right)dx.
        \end{aligned}
    \end{equation}
    Focusing on the function $x \mapsto H_j(x)$ for each $j \in \mathcal{Y}$, we have the following:
    \begin{equation*}
        \begin{aligned}
            \forall x \in \mathbb{R}, ~~ H_j(x) = &~ \int_{0}^{S(x)}\left(\dfrac{b_j^*}{\sigma^*} - \dfrac{\sigma^{*\prime}}{2}\right) \circ S^{-1}(u)du = \int_{0}^{x}\dfrac{1}{\sigma^*(v)}\left(\dfrac{b_j^*}{\sigma^*} - \dfrac{\sigma^{*\prime}}{2}\right)(v)dv\\
            = &~ \int_{0}^{x}\dfrac{b_j^*(v)}{\sigma^{*2}(v)}dv - \dfrac{1}{2}\log\left(\dfrac{\sigma^*(x)}{\sigma^*(0)}\right).
        \end{aligned}
    \end{equation*}
    Since $b_j^*$ is integrable on $\mathbb{R}$ and $\sigma^*$ elliptic under Assumption~\ref{ass:Ell}, we obtain
    \begin{equation}\label{eq:Equiv-i-j-3}
        \begin{aligned}
            \forall x \in \mathbb{R}, ~~ \left|H_j(x)\right| \leq &~ \dfrac{1}{\sigma_0^{*2}}\int_{\mathbb{R}}\left|b_j^*(v)\right|dv + \dfrac{1}{2}\left|\log\left(\dfrac{\sigma_1^*}{\sigma^*(0)}\right)\right| < \infty.
        \end{aligned}
    \end{equation}
    The final result is derived from Equations~\eqref{eq:Equiv-i-j-3} and \eqref{eq:Equiv-i-j-2}.
\end{proof}

\begin{proof}[\textbf{Proof of Lemma}~\ref{lm:Proba-LT}]
    From the Tanaka formula (see \cite{revuz2013continuous}, \textit{Chapter VI, Theorem 1.2, p. 222}), for all $x \in \mathbb{R}$,
\begin{align*}
    |X_1 - x| = &~ |X_0 - x| + \int_{0}^{1}{\mathrm{sgn}(X_{s} - x)dX_{s}} + \mathcal{L}^x\\
              = &~ |x| + \int_{0}^{1}\mathrm{sgn}(X_s - x)b_Y^*(X_s)ds + \int_{0}^{1}{\mathrm{sgn}(X_{s} - x)\sigma^*(X_s)dW_{s}} + \mathcal{L}^x 
\end{align*}
where for all $x \neq 0, ~ \mathrm{sgn}(x) = x/|x|$ and $\mathrm{sgn}(0) = -1$. For all $x \in \mathbb{R}$ and for all $\theta > 0$, we have
\begin{equation}\label{eq:res1}
    \begin{aligned}
        \mathbb{P}\left(\mathcal{L}^x \geq  \theta\right) = &~ \mathbb{P}\left(|X_1 - x| - \int_{0}^{1}\mathrm{sgn}(X_s - x)b_Y^*(X_s)ds - \int_{0}^{1}{\mathrm{sgn}(X_{s} - x)\sigma^*(X_s)dW_{s}} \geq |x| + \theta\right)\\
        \geq &~ \mathbb{P}\left(-\int_{0}^{1}{\mathrm{sgn}(X_s - x)\sigma^*(X_s)dW_s} \geq |x| + \theta - |X_1 - x| + \|b_0^*\|_{\infty} + \|b_1^*\|_{\infty}\right)\\
        \geq &~ \mathbb{P}\left(-\int_{0}^{1}{\mathrm{sgn}(X_s - x)\sigma^*(X_s)dW_s} \geq |x| + \theta + X_1 - x + \|b_0^*\|_{\infty} + \|b_1^*\|_{\infty}\right)\\
        \geq &~ \mathbb{P}\left(\int_{0}^{1}{\mathrm{sgn}(X_s - x)\sigma^*(X_s)dW_s} \leq \theta \Biggm\vert \mathcal{A}(x)\right)\mathbb{P}\left(\mathcal{A}(x)\right)\\
        = &~ \left[1 - \mathbb{P}\left(\int_{0}^{1}{\mathrm{sgn}(X_s - x)\sigma^*(X_s)dW_s} \geq \theta \Biggm\vert \mathcal{A}(x)\right)\right] \mathbb{P}\left(\mathcal{A}(x)\right),
    \end{aligned}
\end{equation}
where $\mathcal{A}(x) = \left\{X_1 \leq -|x| + x - \|b_0^*\|_{\infty} - \|b_1^*\|_{\infty} - 2\theta\right\}$. From \cite{van1995exponential}, \textit{Lemma 2.1}, we have
\begin{equation}\label{eq:res2}
    \mathbb{P}\left(\int_{0}^{1}{\mathrm{sgn}(X_s - x)\sigma^*(X_s)dW_s} \geq \theta \Biggm\vert \mathcal{A}(x)\right) \leq \exp\left(-\dfrac{\theta^2}{2\sigma_1^{*2}}\right).
\end{equation}
Setting $\psi(x, \theta) = -|x| + x - \|b_0^*\|_{\infty} - \|b_1^*\|_{\infty} - 2\theta$ for all $ x \in \mathbb{R}$ and from Equation~\eqref{eq:Transition-density}, there exist constants $c,C>1$ such that
\begin{equation}\label{eq:res3}
    \begin{aligned}
        \mathbb{P}\left(\mathcal{A}(x)\right) = &~ \int_{-\infty}^{\psi(x, \theta)}{\Gamma_{X}(1, y)dy} \geq \dfrac{1}{C}\int_{-\psi(x, \theta)}^{+\infty}{\exp\left(-cy^2\right)dy}\\
       \geq &~ \dfrac{1}{C}\int_{-\psi(x, \theta)}^{-\psi(x, \theta) + \theta}{\exp\left(-cy^2\right)dy} \geq \dfrac{\theta}{C}\exp\left(- c(-\psi(x, \theta) + \theta)^2\right).
    \end{aligned}
\end{equation}
We finally deduce from Equations~\eqref{eq:res1}, \eqref{eq:res2} and \eqref{eq:res3} that for all $x \in \mathbb{R}$ and for all $\theta > 0$,
\begin{equation*}
    \mathbb{P}\left(\mathcal{L}^x \geq \theta\right) \geq \dfrac{\theta}{C}\left[1 - \exp\left(-\dfrac{\theta^2}{2\sigma_1^{*2}}\right)\right]\exp\left(-c(-\psi(x, \theta) + \theta)^2\right).
\end{equation*}
We consider in this paper a version of the local time $x \mapsto \mathcal{L}^x$ that is almost surely continuous (see \cite{revuzyor1999}, \textit{Chap. VI, Corollary 1.8, p. 226}).  Then, there exists $x^* \in [A, B]$ such that
\begin{equation*}
    \underset{x \in [A,B]}{\inf}{~\mathcal{L}^x} = \mathcal{L}^{x^*} ~~~ a.s.
\end{equation*}
Finally, setting $\mathbf{b}_{\infty}^* = \|b_0^*\|_{\infty} + \|b_1^*\|_{\infty}$, we obtain for all $\theta > 0$,
\begin{equation*}
    \begin{aligned}
        \mathbb{P}\left(\underset{x \in [A,B]}{\inf}{~\mathcal{L}^x} \geq \theta\right) = &~ \mathbb{P}\left(\mathcal{L}^{x^*} \geq \theta\right)\\
        \geq &~ \dfrac{1}{C}\left[1 - \exp\left(-\dfrac{\theta^2}{2\sigma_1^{*2}}\right)\right]\exp\left(-c(-\psi(x^*, \theta) + \theta)^2\right)\\
        \geq &~ \dfrac{\theta}{C}\left[1 - \exp\left(-\dfrac{\theta^2}{2\sigma_1^{*2}}\right)\right]\exp\left(-c(|A| + |B| + \mathbf{b}_{\infty}^* + 3\theta)^2\right).
    \end{aligned}
\end{equation*}
\end{proof}

\begin{proof}[\textbf{Proof of Lemma~\ref{lm:Identity-NormX}}]
Using the occupation formula (see, \textit{e.g.}, \cite{ella2024nonparametric}, \textit{Lemma 3.1}), we obtain for any nonnegative, continuous, and compactly supported function $f$,
    \begin{equation*}
        \int_{0}^{1}f(X_s)ds = \int_{\mathrm{Supp}(f)}f(x)\mathcal{L}^xdx.
    \end{equation*}
    Moreover, we have the following.
    \begin{align*}
        \left(\int_{\mathrm{Supp}(f)}f(x)dx\right)\underset{x \in \mathrm{Supp}(f)}{\inf}{~\mathcal{L}^x} \leq \int_{\mathrm{Supp}(f)}f(x)\mathcal{L}^xdx \leq \left(\int_{\mathrm{Supp}(f)}f(x)dx\right)\underset{x \in \mathrm{Supp}(f)}{\sup}{~\mathcal{L}^x}.
    \end{align*}
    Applying the above result to the function $\w{b} - b$, we obtain
    \begin{equation*}
        \begin{aligned}
            \left\|\w{b} - b\right\|_X^2 = &~ \int_{\mathrm{Supp}(b)}(\w{b} - b)^2(x)\mathcal{L}^xdx \leq \left(\underset{x \in \mathrm{Supp}(f)}{\sup}{~\mathcal{L}^x}\right)\left\|\w{b} - b\right\|^2.
        \end{aligned}
    \end{equation*}
    Since $\mathbb{E}\left[\left\|\widehat{b} - b\right\|_n^2\right] \rightarrow 0$ as $N,n \rightarrow \infty$ and 
    \begin{equation*}
        \mathbb{E}\left[\left\|\w{b} - b\right\|_n^2\right] \leq \mathbb{E}\left[\left\|\widehat{b} - b\right\|_n^2\right],
    \end{equation*}
    we obtain $\mathbb{E}\left[\left\|\w{b} - b\right\|_n^2\right]\rightarrow 0$ as $N,n \rightarrow \infty$. In addition, we have
    \begin{align*}
        \mathbb{E}\left[\left\|\w{b} - b\right\|_n^2\right] \geq \int_{\mathrm{Supp}(b)}(\w{b} - b)^2(x)f_n(x)dx,  
    \end{align*}
    where $f_n: x \mapsto (1/n)\sum_{k=1}^{n-1}\Gamma_X(k\Delta, x)$ is a continuous function that is lower bounded by a strictly positive constant on the compact subset $\mathrm{Supp}(b)$ of $\mathbb{R}$ (see \cite{denis2021ridge}, \textit{lemma 4.3}). Then, there exists a constant $\pi_0>0$ such that
     \begin{align*}
        \mathbb{E}\left[\left\|\w{b} - b\right\|_n^2\right] \geq \pi_0\mathbb{E}\left[\int_{\mathrm{Supp}(b)}(\w{b} - b)^2(x)dx\right] = \pi_0\mathbb{E}\left[\left\|\w{b} - b\right\|^2\right].  
    \end{align*}
    It follows that $\mathbb{E}\left[\left\|\w{b} - b\right\|^2\right] \rightarrow 0$ as $N,n \rightarrow \infty$. Using the Markov inequality, for all $\delta > 0$, we obtain
    \begin{equation}\label{eq:cv1}
        \begin{aligned}
            \mathbb{P}\left(\left\|\w{b} - b\right\| > \delta\right) \leq \delta^{-2}\mathbb{E}\left[\left\|\w{b} - b\right\|^2\right] \longrightarrow 0 ~~ \mathrm{as} ~~ N,n \rightarrow \infty.
        \end{aligned}
    \end{equation}
    We now study the convergence in probability of the random variable $\int_{\mathrm{Supp}(b)}\w{b}(x)dx$ to $\int_{\mathrm{Supp}(b)}b(x)dx$. To this end, we have, on the one hand,
    \begin{align*}
        \mathbb{E}\left[\left|\int_{\mathrm{Supp}(b)}\w{b}(x)dx - \int_{\mathrm{Supp}(b)}b(x)dx\right|\right] \leq &~ \mathbb{E}\left[\int_{\mathrm{Supp}(b)}\left|\w{b} - b\right|(x)dx\right] \leq \left(\mathbb{E}\left[\left\|\w{b} - b\right\|^2\right]\right)^{1/2},
    \end{align*}
    on the other hand, using Markov's inequality, we have for all $\delta > 0$,
    \begin{equation}\label{eq:cv2}
        \begin{aligned}
            \mathbb{P}\left(\left|\int_{\mathrm{Supp}(b)}\w{b}(x)dx - \int_{\mathrm{Supp}(b)}b(x)dx\right| > \delta\right) \leq &~ \delta^{-1}\mathbb{E}\left[\left|\int_{\mathrm{Supp}(b)}\w{b}(x)dx - \int_{\mathrm{Supp}(b)}b(x)dx\right|\right]\\
        \leq &~ \delta^{-1}\left(\mathbb{E}\left[\left\|\w{b} - b\right\|^2\right]\right)^{1/2}.
        \end{aligned}
    \end{equation}
    It follows that
    \begin{align*}
        \int_{\mathrm{Supp}(b)}\w{b}(x)dx \underset{N,n \rightarrow \infty}{\overset{\mathbb{P}}{\longrightarrow}} \int_{\mathrm{Supp}(b)}b(x)dx > 0.
    \end{align*}
    Let $\mathfrak{C} > 0$ be a real number to be chosen later and set $\mathcal{O} = \mathcal{O}_1 \cap \mathcal{O}_2 \cap \mathcal{O}_3$ where
    \begin{align*}
        & \mathcal{O}_1 = \left\{\left\|\w{b} - b\right\| \leq 1/2\right\}, ~~ \mathcal{O}_2 = \left\{\underset{x \in \mathrm{Supp}(b)}{\sup}\mathcal{L}^x \leq \mathfrak{C}\right\},\\ 
        \mathrm{and} ~~ & \mathcal{O}_3 = \left\{\left|\int_{\mathrm{Supp}(b)}\w{b}(x)dx - \int_{\mathrm{Supp}(b)}b(x)dx\right| \leq \dfrac{1}{2}\int_{\mathrm{Supp}(b)}b(x)dx\right\}.
    \end{align*} 
    For any $\alpha \in (0,1)$ and for $N,n$ large enough, we have the following.
\begin{align*}
    \mathbb{P}\left(\int_{0}^{1}(\w{b}-b)(X_s)ds \geq \alpha\left\|\w{b}-b\right\|_X\right)\geq &~ \mathbb{P}\left(\int_{0}^{1}(\w{b}-b)(X_s)ds \geq \alpha\left(\underset{x \in \mathrm{Supp}(b)}{\sup}\mathcal{L}^x\right)^{1/2}\left\|\w{b}-b\right\|\right)\\
    \geq &~ \mathbb{P}\left(\int_{0}^{1}(\w{b}-b)(X_s)ds \geq \alpha\mathfrak{C}/2 \biggm\vert \mathcal{O}\right)\mathbb{P}\left(\mathcal{O}\right)\\
    \geq &~ \mathbb{P}\left(\int_{0}^{1}\w{b}(X_s)ds \geq \alpha\mathfrak{C}/2 + \left\|b\right\|_{\infty} \biggm\vert \mathcal{O}\right)\mathbb{P}\left(\mathcal{O}\right).
\end{align*}
   Since 
   $$\int_{0}^{1}\w{b}(X_s)ds \geq \left(\int_{\mathrm{Supp}(b)}\w{b}(x)dx\right)\underset{x \in \mathrm{Supp}(b)}{\inf}{~\mathcal{L}^x} ~~ \mathrm{and} ~~ \underset{x \in \mathrm{Supp}(b)}{\sup}\mathcal{L}^x \perp\!\!\!\perp \left\|\w{b}-b\right\|,$$
    We obtain the following:
    \begin{multline}\label{eq:Eq1}
         \mathbb{P}\left(\int_{0}^{1}(\w{b}-b)(X_s)ds \geq \alpha\left\|\w{b}-b\right\|_X\right) \\
        \geq  \mathbb{P}\left(\underset{x \in \mathrm{Supp}(b)}{\inf}{~\mathcal{L}^x} \geq \dfrac{\alpha\mathfrak{C}/2 + \|b\|_{\infty}}{\int_{\mathrm{Supp}(b)}\w{b}(u)du} \biggm\vert \mathcal{O}\right)\mathbb{P}\left(\mathcal{O}_1 \cap \mathcal{O}_2\right)\mathbb{P}\left(\mathcal{O}_3\right)\\
        \geq  \mathbb{P}\left(\underset{x \in \mathrm{Supp}(b)}{\inf}{~\mathcal{L}^x} \geq \dfrac{\alpha\mathfrak{C}/2 + \|b\|_{\infty}}{(1/2)\int_{\mathrm{Supp}(b)}b(u)du} \biggm\vert \mathcal{O}\right)\mathbb{P}\left(\mathcal{O}_1 \cap \mathcal{O}_2\right)\mathbb{P}\left(\mathcal{O}_3\right).
    \end{multline}
  From the result of Lemma~\ref{lm:Proba-LT} with 
  $$\theta = \dfrac{\alpha\mathfrak{C}/2 + \|b\|_{\infty}}{(1/2)\int_{\mathrm{Supp}(b)}b(u)du}, $$
  there exist constants $C,c>1$ such that
  \begin{equation}\label{eq:Eq2}
      \begin{aligned}
          \mathbb{P}\left(\underset{x \in \mathrm{Supp}(b)}{\inf}{~\mathcal{L}^x} \geq \theta \biggm\vert \mathcal{O}\right) \geq \dfrac{\theta}{C}\left[1 - \exp\left(-\dfrac{\theta^2}{2\sigma_1^{*2}}\right)\right]\exp\left(-c(C_b + 3\theta)^2\right),
      \end{aligned}
  \end{equation}
  where $C_b = \left|\min\{\mathrm{Supp}(b)\}\right| + \left|\max\{\mathrm{Supp}(b)\}\right| + \|b\|_{\infty}$. Focusing on $\mathbb{P}\left(\mathcal{O}_3\right)$, recall that the local time $x \mapsto \mathcal{L}^x$ is chosen to be almost surely continuous. Then, there exists $x^{**} \in \mathrm{Supp}(b)$ such that 
   $$\underset{x \in \mathrm{Supp}(b)}{\sup}\mathcal{L}^x = \mathcal{L}^{x^{**}} ~~ a.s.$$
   Using the Markov inequality together with Lemma 3.1 in \cite{ella2024nonparametric} and Equation~\eqref{eq:Transition-density}, there exist constants $C,c>1$ such that
   \begin{equation*}
       \begin{aligned}
           \mathbb{P}\left(\mathcal{O}_3\right) = &~ 1 - \mathbb{P}\left(\mathcal{L}^{x^{**}} \geq \mathfrak{C}\right)\\
           \geq &~ 1 - \mathfrak{C}^{-1}\mathbb{E}\left[\mathcal{L}^{x^{**}}\right] = 1 - \mathfrak{C}^{-1}\int_{0}^{1}\Gamma_X(s,x^{**})ds\\
            \geq &~ 1 - C\mathfrak{C}^{-1}\int_{0}^{1}s^{-1/2}ds = 1 - 2C\mathfrak{C}^{-1}.
       \end{aligned}
   \end{equation*}
    Now for $\mathbb{P}\left(\mathcal{O}_1 \cap \mathcal{O}_2\right)$, from Equation~\eqref{eq:cv1} with $\delta = 1/2$ and Equation~\eqref{eq:cv2} with $\delta = (1/2)\int_{\mathrm{Supp}(b)}b(x)dx$, we obtain
    \begin{equation*}
        \begin{aligned}
            \mathbb{P}\left(\mathcal{O}_1 \cap \mathcal{O}_2\right) &~ \geq 1 - \mathbb{P}\left(\mathcal{O}_1^c\right) - \mathbb{P}\left(\mathcal{O}_2^c\right)\\
            \geq &~ 1 - 4\mathbb{E}\left[\left\|\w{b} - b\right\|^2\right] - \left(\mathbb{E}\left[\left\|\w{b} - b\right\|^2\right]\right)^{1/2}\dfrac{2}{\int_{\mathrm{Supp}(b)}b(x)dx}.
        \end{aligned}
    \end{equation*}
    Since $\mathbb{E}\left[\left\|\w{b} - b\right\|^2\right] \rightarrow 0$ as $N,n \rightarrow \infty$, for $N$ and $n$ large enough, we have 
    $$\mathbb{E}\left[\left\|\w{b} - b\right\|^2\right] \leq \left(8 + \dfrac{4}{\int_{\mathrm{Supp}(b)}b(x)dx}\right)^{-2},$$
    which implies that $\mathbb{P}\left(\mathcal{O}_1 \cap \mathcal{O}_2\right) \geq 1/2$. Then, for $\mathfrak{C} \in (4C, +\infty)$, we obtain the following:
    \begin{equation}\label{eq:Eq3}
         \mathbb{P}\left(\mathcal{O}_1 \cap \mathcal{O}_2\right)\mathbb{P}\left(\mathcal{O}_3\right) \geq \dfrac{1}{2} - C\mathfrak{C}^{-1} > 0.
    \end{equation}
    Finally, since $\alpha = \min\left\{(1/4)\int_{\mathrm{Supp}(b)}b(x)dx, 1/3\right\} \in (0,1)$, we choose $\mathfrak{C}$ so that
    $$\mathfrak{C} = \max\left\{4C, \dfrac{4\|b\|_{\infty}}{\int_{\mathrm{Supp}(b)}b(x)dx - 2\alpha}\right\} > 0.$$
    We deduce from Equations~\eqref{eq:Eq3}, \eqref{eq:Eq2} and \eqref{eq:Eq1} that
    \begin{equation}\label{eq:Eq4}
        \begin{aligned}
            \mathbb{P}\left(\int_{0}^{1}(\w{b}-b)(X_s)ds \geq \alpha\left\|\w{b}-b\right\|_X\right) \geq c,
        \end{aligned}
    \end{equation}
    where
    \begin{align*}
        c = \dfrac{\theta}{C}\left(1 - C\mathfrak{C}^{-1}\right)\left[1 - \exp\left(-\dfrac{\theta^2}{2\sigma_1^{*2}}\right)\right]\exp\left(-c(C_b + 3\theta)^2\right) > 0.
    \end{align*}
\end{proof}

\begin{proof}[\textbf{Proof of Proposition}~\ref{prop:ChangeProba}]
    For each $i \in \mathcal{Y}$, set $\mathbb{P}_{i} = \mathbb{P}(.|Y=i)$. Under Assumptions~\ref{ass:Reg}, \ref{ass:Ell} and \ref{ass:Novikov}, denote by $\mathbb{Q}$, the probability under which the diffusion process $X = (X_t)_{t \in [0,1]}$ is solution of the stochastic differential equation $dX_t = d\w{W}_t$, where $\w{W} = (\w{W}_t)_{t \in [0,1]}$ is a Brownian motion under $\mathbb{Q}$. From the Girsanov's theorem (see \cite{revuz2013continuous}, Chapter VIII), we have for each $i \in \mathcal{Y}$,
    \begin{equation*}
        \dfrac{d\mathbb{P}_i}{d\mathbb{Q}}(X^{(t)}) := \exp\left(\int_{0}^{t}{b_i^*(X_s)dX_s - \dfrac{1}{2}\int_{0}^{t}{b_i^{*2}(X_s)ds}}\right) = \exp\left(\int_{0}^{t}{b_i^*(X_s)dW_s + \dfrac{1}{2}\int_{0}^{t}{b_i^{*2}(X_s)ds}}\right),
    \end{equation*}
    where $X^{(t)} = (X_s)_{s \in [0,t]}$, and $X^{(1)} = X$. Then, under Assumption~\ref{ass:Reg2}, for all $i,j \in \mathcal{Y}$ such that $i \neq j$, there exist constants $C, C_1>0$ such that
    \begin{equation*}
        \begin{aligned}
            \dfrac{d\mathbb{P}_i}{d\mathbb{P}_j}(X^{(t)}) = &~ \exp\left(\int_{0}^{t}{(b_i^* - b_j^*)(X_s)dW_s + \dfrac{1}{2}\int_{0}^{t}{(b_i^{*2} - b_j^{*2})(X_s)ds}}\right)\\ 
            \leq &~  \exp\left(M^{i,j}_t\right)\exp\left(\dfrac{C^*}{2}\int_{0}^{t}{(|b_i^*(X_s)|+|b_j^*(X_s)|)ds}\right)\\
            \leq &~ C\exp\left(M^{i,j}_{k\Delta}\right)\exp\left(C_1\underset{t \in [0,1]}{\sup}{|X_t|}\right),
        \end{aligned}
    \end{equation*}
    where $(M^{i,j}_t)_{t \in [0,1]}$ is a martingale given by
    \begin{equation*}
        M^{i,j}_t := \int_{0}^{t}{(b_i^* - b_j^*)(X_s)dW_s}, ~~ t \in [0,1],
    \end{equation*}
    and under Assumption~\ref{ass:Reg2}, for all $t \in [0,1]$,
    \begin{equation}\label{eq:CP00}
        \left<M^{i,j}, M^{i,j}\right>_t = \int_{0}^{t}{(b_i^* - b_j^*)^2(X_s)ds} \leq C^{*2}.
    \end{equation}
    For $i,j \in \mathcal{Y}$ such that $ i \neq j$, under Assumption~\ref{ass:Reg}, there exists a constant $C_1>0$ such that
    \begin{equation*}
        \begin{aligned}
            \left\|\w{b}_i - b_i^*\right\|_{n,j}^2 = &~ \dfrac{1}{n}\sum_{k=0}^{n-1}{\mathbb{E}_{X|Y=j}\left[\left(\w{b}_i - b_i^*\right)^2(X_{k\Delta})\right]} = \dfrac{1}{n}\sum_{k=0}^{n-1}{\mathbb{E}_{X|Y=i}\left[\left(\w{b}_i - b_i^*\right)^2(X_{k\Delta})\dfrac{d\mathbb{P}_j}{d\mathbb{P}_i}(X^{(k\Delta)})\right]}\\
            \leq &~ \dfrac{C}{n}\sum_{k=0}^{n-1}{\mathbb{E}_{X|Y=i}\left[\left(\w{b}_i - b_i^*\right)^2(X_{k\Delta})\exp\left(M^{j,i}_{k\Delta}\right)\exp\left(C_1\underset{t \in [0,1]}{\sup}{|X_t|}\right)\right]}.
        \end{aligned}
    \end{equation*}
    
    For all $a > 0$, using the Cauchy Schwarz inequality, we have
    \begin{equation}\label{eq:EQij1}
        \begin{aligned}
            \left\|\w{b}_i - b_i^*\right\|_{n,j}^2 \leq &~ C\exp(C_1a)\left\|\w{b}_i - b_i^*\right\|_{n,i}^2 + \Xi_1 + \Xi_2 + \Xi3,
        \end{aligned}
    \end{equation}
    where
    \begin{equation}\label{eq:Xi}
        \begin{aligned}
            \Xi_1 := &~ \dfrac{C\exp(a)}{n}\sum_{k=0}^{n-1}{\mathbb{E}_{X|Y=i}\left[\left(\w{b}_i - b_i^*\right)^2(X_{k\Delta})\exp\left(C_1\underset{t \in [0,1]}{\sup}{|X_t|}\right)\mathds{1}_{\underset{t \in [0,1]}{\sup}{|X_t|} > a}\right]}\\
            \Xi_2 := &~ \dfrac{C\exp(C_1a)}{n}\sum_{k=0}^{n-1}{\mathbb{E}_{X|Y=i}\left[\left(\w{b}_i - b_i^*\right)^2(X_{k\Delta})\exp\left(M^{j,i}_{k\Delta}\right)\mathds{1}_{M_{k\Delta}^{i,j} > a}\right]}\\
            \Xi_3 := &~  \dfrac{C}{n}\sum_{k=0}^{n-1}{\mathbb{E}_{X|Y=i}\left[\left(\w{b}_i - b_i^*\right)^2(X_{k\Delta})\exp\left(M^{j,i}_{k\Delta}\right)\exp\left(C_1\underset{t \in [0,1]}{\sup}{|X_t|}\right)\mathds{1}_{M_{k\Delta}^{i,j} > a}\mathds{1}_{\underset{t \in [0,1]}{\sup}{|X_t|} > a}\right]}.
        \end{aligned}
    \end{equation}
    We essentially focus on $\Xi_3$. The upper bounds on $\Xi_1$ and $\Xi_2$ will be deduce from the proof of the upper bound of $\Xi_3$. Then, we have
    \begin{equation}\label{eq:Xi3}
        \begin{aligned}
            \Xi_3 \leq &~ \dfrac{C}{n}\sum_{k=0}^{n-1}{\left(\mathbb{E}_{X|Y=i}\left[(\w{b}_i - b_i^*)^4(X_{k\Delta})\right]\right)^{1/2}\left(\mathbb{E}_{X|Y=i}\left[T_{k\Delta}^2V_1^2\right]\right)^{1/2}}\\
             \leq &~ \dfrac{C}{n}\sum_{k=0}^{n-1}{\left(\mathbb{E}_{X|Y=i}\left[(\w{b}_i - b_i^*)^4(X_{k\Delta})\right]\right)^{1/2}\left(\mathbb{E}_{X|Y=i}\left[T_{k\Delta}^4\right]\right)^{1/4}\left(\mathbb{E}_{X|Y=i}\left[V_1^4\right]\right)^{1/4}},
        \end{aligned}
    \end{equation}
    where
    \begin{align*}
        T_{k\Delta} = \exp\left(M_{k\Delta}^{j,i}\right)\mathds{1}_{M_{k\Delta}^{j,i} > a}, ~~ V_1 = \exp\left(C_1\underset{t \in [0,1]}{\sup}{|X_t|}\right)\mathds{1}_{\underset{t \in [0,1]}{\sup}{|X_t|} > a}.
    \end{align*}

     \subsection*{Upper bound of $\mathbb{E}_{X|Y=i}\left[(\w{b}_i - b_i^*)^4(X_{k\Delta})\right]$}

    Under Assumption~\ref{ass:Reg} and from Equation~\eqref{eq:Moments}, for all $k \in \{0, \ldots, n-1\}$ and for $N$ large enough,
    \begin{equation}\label{eq:EQij2}
        \begin{aligned}
            \mathbb{E}_{X|Y=i}\left[(\w{b}_i - b_i^*)^4(X_{k\Delta})\right] \leq &~ 4\mathbb{E}_{X|Y=i}\left[\w{b}_i^4(X_{k\Delta}) + b_i^{*4}(X_{k\Delta})\right]\\
            \leq &~ 4\left\|\w{b}_i\right\|_{\infty}^4 + C\left(1 + \mathbb{E}_{X|Y=i}\left[\underset{t \in [0,1]}{\sup}{|X_t|^4}\right]\right)\\
            \leq &~ C\log^4(N).
        \end{aligned}
    \end{equation}

    \subsection*{Upper bound of $\mathbb{E}_{X|Y=i}\left[T_{k\Delta}^4\right]$}
    
    For all $k \in \{0, \ldots, n-1\}$, we have
    \begin{equation*}
        \begin{aligned}
            \mathbb{E}_{X|Y=i}\left[T_{k\Delta}^4\right] = &~ \mathbb{E}_{X|Y=i}\left[\exp\left(4M_{k\Delta}^{j,i}\right)\mathds{1}_{M_{k\Delta}^{j,i} > a}\right]\\
            \leq &~  \sqrt{\mathbb{P}_i\left(M_{k\Delta}^{j,i} > a\right)\mathbb{E}_{X|Y=i}\left[\exp\left(4M_{k\Delta}^{j,i} - 8\left<M^{j,i}, M^{j,i}\right>_{k\Delta}\right)\exp\left(8\left<M^{j,i}, M^{j,i}\right>_{k\Delta}\right)\right]}.
        \end{aligned}
    \end{equation*}
    From \cite{van1995exponential}, \textit{Lemma 2.1}, there exists a constant $c > 0$ such that 
    $$\mathbb{P}_i\left(M_{k\Delta}^{j,i} > a\right) \leq \exp(-a^2/c),$$ 
    and from Equation~\eqref{eq:CP00}, $\exp\left(8\left<M^{j,i}, M^{j,i}\right>_{t}\right) < \infty$ for all $t \in [0,1]$, which implies that
    \begin{equation*}
        \underset{t \in [0,1]}{\sup}{\mathbb{E}_{X|Y=i}\left[\exp\left(8\left<M^{j,i}, M^{j,i}\right>_{t}\right)\right]} < \infty.
    \end{equation*}
    Then from \cite{kazamaki2006continuous}, \textit{Chapter 1, Theorem 1.6, page 9}, $\left(\exp\left(4M_t^{j,i} - 8\left<M^{j,i}, M^{j,i}\right>_t\right)\right)_{t \in [0,1]}$ is a uniformly integrable martingale with respect to natural filtration $\mathcal{F}^M = \left(\mathcal{F}_t^{M}\right)_{t \in [0,1]}$, and we have
    \begin{equation*}
        \begin{aligned}
        &\mathbb{E}_{X|Y=i}\left[\exp\left(4M_{k\Delta}^{j,i} - 8\left<M^{j,i}, M^{i,j}\right>_{k\Delta}\right)\right]\\
        &= \mathbb{E}_{X|Y=i}\left[\mathbb{E}_{X|Y=i}\left(\exp\left(4M_{k\Delta}^{j,i} - 8\left<M^{j,i}, M^{j,i}\right>_{k\Delta}\right) \Biggm\vert \mathcal{F}_0^M\right)\right]\\
        &= \mathbb{E}_{X|Y=i}\left[\exp\left(4M_{0}^{j,i} - 8\left<M^{j,i}, M^{j,i}\right>_{0}\right)\right] = 1.
        \end{aligned}
    \end{equation*}
    We deduce that there exists a constant $C>0$ such that
    \begin{equation}\label{eq:EQij3}
        \mathbb{E}_{X|Y=i}\left[T_{k\Delta}^4\right] \leq C\exp\left(-\frac{a^2}{2c}\right).
    \end{equation}    

    \subsection*{Upper bound of $\mathbb{E}_{X|Y=i}\left[V_1^4\right]$}

    Using the Cauchy Schwarz inequality, we have
    \begin{equation*}
        \begin{aligned}
            \mathbb{E}_{X|Y=i}\left[V_1^4\right] = &~ \mathbb{E}_{X|Y=i}\left[\exp\left(4C_1\underset{t \in [0,1]}{\sup}{|X_t|}\right)\mathds{1}_{\underset{t \in [0,1]}{\sup}{|X_t|} > a}\right]\\
            \leq &~ \sqrt{\mathbb{E}_{X|Y=i}\left[\exp\left(8C_1\underset{t \in [0,1]}{\sup}{|X_t|}\right)\right]\mathbb{P}_i\left(\underset{t \in [0,1]}{\sup}{|X_t|} > a\right)}.
        \end{aligned}
    \end{equation*}
    Since the process $X$ is continuous, there exists $t_0 \in (0,1]$ such that $\underset{t \in [0,1]}{\sup}{|X_t|} = |X_{t_{0}}| ~~ a.s.$, and from Equation~\eqref{eq:Transition-density}, we have
    \begin{equation*}
        \begin{aligned}
            \mathbb{E}_{X|Y=j}\left[\exp\left(8C_1\underset{t \in [0,1]}{\sup}{|X_t|}\right)\right] \leq &~ \exp\left(8C_1\delta\right)\mathbb{E}_{X|Y=j}\left[\exp\left(8C_1|X_{t_{\delta}}|\right)\right]\\
            \leq &~ 2\exp\left(8C_1\delta\right)\int_{0}^{+\infty}{\dfrac{1}{C\sqrt{t_{\delta}}}\exp\left(8C_1x - \dfrac{x^2}{ct_{\delta}}\right)dx} < \infty.
        \end{aligned}
    \end{equation*}
    On the other hand, from Equation~\eqref{eq:Transition-density}, we have
    \begin{equation*}
        \begin{aligned}
            \mathbb{P}_i\left(\underset{t \in [0,1]}{\sup}{|X_t|} > a\right) = &~ \mathbb{P}\left(|X_{t_0}| > a\right) = \int_{|x|>a}\Gamma_X(t_0, x)dx\\
            \leq &~ \dfrac{2C}{\sqrt{t_0}}\int_{a}^{+\infty}\exp\left(-\dfrac{x^2}{ct_0}\right)dx\\
            \leq &~ \dfrac{C}{a}\exp\left(-\dfrac{a^2}{ct_0}\right),
        \end{aligned}
    \end{equation*}
    where $C>0$ is a new constant depending on $t_0>0$. Then, there exists a constant $C>0$ such that 
    \begin{equation}\label{eq:EQij4}
        \mathbb{E}_{X|Y=i}\left[V_1^4\right] \leq \dfrac{C}{\sqrt{a}}\exp\left(-\dfrac{a^2}{2ct_0}\right).
    \end{equation}
    From Equations~\eqref{eq:EQij4}, \eqref{eq:EQij3}, \eqref{eq:EQij2} and \eqref{eq:Xi3}, there exists a constant $C>0$ such that
    \begin{equation}\label{eq:Xi3bis}
        \Xi_3 \leq \dfrac{C}{\sqrt{a}}\left[\log^4(N)\exp\left(-\dfrac{a^2}{c} - \dfrac{a^2}{2ct_0}\right)\right].
    \end{equation}
    Moreover, using the Cauchy-Schwarz inequality, we obtain the following result.
    \begin{equation}\label{eq:Xi1-2}
        \Xi_1 \leq C\dfrac{\exp(a)}{\sqrt{a}}\exp\left(-\dfrac{a^2}{2ct_0}\right), ~~~ \Xi_2 \leq C\exp(C_1a)\exp\left(-\dfrac{a^2}{c}\right).
    \end{equation}
    Finally, for $a = \sqrt{8c\log(N)}$, we deduce from Equations~\eqref{eq:Xi1-2}, \eqref{eq:Xi3} and \eqref{eq:EQij1}  that there exist constants $C, c > 0$ such that
    \begin{equation*}
        \left\|\w{b}_i - b_i^*\right\|_{n,j}^2 \leq C\exp\left(\sqrt{c\log(N)}\right)\left\|\w{b}_i - b_i^*\right\|_{n,i}^2 + C\exp\left(\sqrt{c\log(N)}\right)N^{-1}.
    \end{equation*}
\end{proof}


\end{document}